\documentclass[12pt]{article}
\usepackage{geometry}
\usepackage{amsmath,amsthm,amscd,amssymb,latexsym, enumerate,bm,lscape,epsfig,nicefrac,hyperref,amsfonts,cite,mathtools,todonotes,mathrsfs}

\newcommand{\N}{\mathbb{N}}
\newcommand{\Z}{\mathbb{Z}}

\newcommand{\R}{\mathbb{R}}

\newcommand{\E}{\mathbb{E}}

\newcommand{\supp}{\text{supp}}

\newtheorem{info}{}
\newtheorem{theorem}[info]{Theorem}
\newtheorem{corr}[info]{Corollary}

\newtheorem{defin}[info]{Definition}

\newtheorem{lem}[info]{Lemma}
\newtheorem{prop}[info]{Proposition}
\newtheorem{remark}[info]{Remark}
\numberwithin{info}{section}

\numberwithin{equation}{section}
\renewcommand{\[}{\begin{equation}}
	\renewcommand{\]}{\end{equation}}
\makeatletter

\makeatletter
\g@addto@macro\normalsize{%
	\setlength\abovedisplayskip{5pt}
	\setlength\belowdisplayskip{5pt}
	\setlength\abovedisplayshortskip{4pt}
	\setlength\belowdisplayshortskip{4pt}}
\makeatother

\newcommand{\lam}{\lambda}
\newcommand{\Om}{\Omega}

\renewcommand{\P}{\mathbb{P}}
\renewcommand{\cal}{\mathcal}

\newcommand{\sq}{\sqrt}

\newcommand{\<}{\langle}
\renewcommand{\>}{\rangle}
\newcommand{\Sum}{\mathrm{\Sum}}

\newcommand{\To}{\Rightarrow}

\newcommand{\D}{\nabla}

\newcommand{\disteq}{\stackrel{d}{=}}

\renewcommand{\l}{\left}
\renewcommand{\r}{\right}

\renewcommand{\b}{\bm}
\newcommand{\tr}{\mathrm{tr}}

\newcommand{\qc}{q_c}
\newcommand{\zetac}{\zeta_c}

\begin{document}
	\title{The Larkin Mass and Replica Symmetry Breaking in the Elastic Manifold}
	
	\author{Gérard Ben Arous, Pax Kivimae}
	
	\maketitle
	
	\begin{abstract}
		
		This is the second of a series of three papers about the Elastic Manifold model. This classical model proposes a rich picture due to the competition between the inherent disorder and the smoothing effect of elasticity. In this paper, we analyze our variational formula for the free energy obtained in our first companion paper \cite{Paper1}.
		
		We show that this variational formula may be simplified to one which is solved by a unique saddle point. We show that this saddle point may be solved for in terms of the corresponding critical point equation. Moreover, its terms may be interpreted in terms of natural statistics of the model: namely the overlap distribution and effective radius of the model at a given site. 
		
		Using this characterization, obtain a complete characterization of the replica symmetry breaking phase. From this we are able to confirm a number of physical predictions about this boundary, namely those involving the Larkin mass \cite{mezardparisi,balentsRG,FRGDouWeise}, an important critical mass for the system. The zero-temperature Larkin mass has recently been shown to be the topological trivialization threshold, following work of Fyodorov and Le Doussal \cite{fyodorov-manifold,fyodorov-manifold-minimum}, made rigorous by the first author, Bourgade and McKenna \cite{gerardbenpaul,gerardbenpaulCompanionPaper}.
	\end{abstract}
	
	\tableofcontents

	\section{Introduction}
	
	This paper is the second in a series of works focusing on the elastic manifold, a paradigmatic model in the study of disordered elastic systems. Our first paper \cite{Paper1}, computed the asymptotic quenched free energy of the model in the large-$N$ limit in terms of a variational problem over a certain Parisi functional. In this paper, we show that the asymptotic quenched free energy may be computed in terms of a simpler variational problem, which we show is solved by a unique saddle point pair $(\qc,\zetac)$. The critical points equations for this pair are explicit, and thus allow for a concrete method to compute the value of the free energy.
	
	Next, we show that both entries of the saddle point may be interpreted as the asymptotic values of important statistics from the Gibbs measure in the large-$N$ limit. In particular, we show that $\zetac$ is given by the limiting overlap distribution of two samples from the Gibbs measure, taken at any given site. In addition, we show that $\qc$ is the square of the effective radius of the model, in the sense that Gibbs measure, which is supported on a Euclidean space, concentrates in a neighborhood of a product of spheres of radius $\sqrt{N\qc}$. In particular, our equations also give a method to relate and compute both of these quantities.
	
	Employing this, we study the replica symmetry breaking in this model. We find a complete characterization of the boundary between the replica symmetry breaking phases (RSB) and replica symmetric (RS), which we use to confirm a number of partial characterizations in the physics literature. In particular, we confirm that the zero-temperature critical mass is precisely the Larkin mass. We further show that the positive temperature Larkin mass is a critical boundary, as predicted, but show that the Larkin equation, which captures local stability of the RS-solution, does not always coincide with the RS-RSB phase boundary. Moreover, in the RS-phase, we explicitly compute both $(\qc,\zetac)$ and the value of the asymptotics quenched free energy. In addition, we show a series of analogous results for a family of spherical models with elastic interaction.
	
	We now turn to introducing the elastic manifold. Its study was first undertaken by Fisher \cite{fischerRG}, who used it as an effective model for the behavior of disordered elastic systems. Such manifolds are usually rugged, shaped by a competition between an elastic self-interaction that seeks to flatten the object and random spatial impurities which encourage rougher configurations. Some well-studied models in this class are random polymers \cite{polymerbook1,polymerbook2} and the domain walls in the random field Ising model and the Edwardson-Anderson model \cite{EAinterface}. For more background on these models, as well as an overview of their physical phenomenology, we suggest the review \cite{elasticReview}.
	
	We now give a heuristic description of the elastic manifold model. Its phase space is the collection of functions $\b{u}:\Omega\to \R^N$, for some choice of $\Omega\subset \R^d$. The Hamiltonian is given by
	\[\cal{H}(\b{u})=\frac{1}{2}\int_{\Omega}\left(\mu \|\b{u}(x)\|^2-t(\b{u}(x),\Delta \b{u}(x))\right)dx^d+\int_{\Omega}V(x,\b{u}(x))dx^d,\]
	where $V(x,\b{u}(x))$ is the disorder potential, which is taken to be a centered Gaussian process on $\Omega\times \R^N$ with covariance given for $(x,u),(x',u')\in \Omega\times \R^N$ by
	\[\E[V(x,u)V(x',u')]=\delta(x-x')B(\|u-u'\|^2),\]
	where $\delta(x-x')$ denotes the Dirac $\delta$-function and $B$ is some fixed correlation function.
	
	
	A major breakthrough in the study of this came from the work of M{\'e}zard and Parisi \cite{mezardparisi,mezardparisi2}, who realized that results could be obtained if one worked in the large-$N$ limit. While non-rigorous, their results not only allowed them to compute the limiting quenched free energy of the model, but predicted that Gibbs measure should enter a replica symmetry breaking phase when the strength of the disorder is large compared to the temperature and mass, which our results partially confirm. For an introduction to the physical predictions about such phases, we recommend the work \cite{RSBintro}.
	
	
	This phase is also expected to be matched by an abundance of meta-stable states in the energy landscape, thought to cause a notable difference in the equilibrium behavior of the model. These results mirror predictions for the ``glassy" phase in various other random manifold models. As before, we turn to the overview \cite{elasticReview} for a more complete introduction, but in particular, such aspects have appeared in the studies of flux vortices in type-II superconductors \cite{vortex1,vortex2,vortex3}, as well in the study of the folding of heterogeneous proteins \cite{bio1,bio2,bio3}.
	

	\subsection{Results for the Elastic Manifold Model}
	
	Here we give our main results for the elastic manifold model. Namely, we give our simplified formula for the limiting free energy in two forms, Theorems \ref{theorem:intro:main:critical points euclidean} and \ref{theorem:intro:main:Euclidean}. Theorem \ref{theorem:intro:main:Euclidean} expresses this quantity as a variational formula, while Theorem \ref{theorem:intro:main:critical points euclidean} gives it as the value of the functional on saddle point pair, which is uniquely specified by the saddle point equations listed in the theorem. We then give Theorem \ref{theorem:intro:main:Euclidean parameters identification}, which identifies the entries of this pair in terms of limits of certain statistics of the Gibbs measure.
	
	Now to begin, we recall the (discrete) elastic manifold model. As in our previous paper \cite{Paper1}, we specifically study the discrete version introduced by Fyodorov and Le Doussal \cite{fyodorov-manifold}\footnote{This differs from the model in \cite{mezardparisi}, who instead work with the continuum model on $[1,L]^d$ with energy cut offs.}. To begin, fix integers $L\ge 1$ and $d\ge 0$, which will be referred to as the length and internal dimension, respectively. Additionally, fix positive numbers $\mu $ and $t $, referred to as the mass and interaction strength. We will denote by $\Om$ the graph $[[1,L]]^d\subseteq \Z^d$ understood in the periodic sense. Next, we let $V_N:\R^N\to \R$ be a centered Gaussian random field with isotropic covariance given for $u,v\in \R^N$ by 
	\[\E[V_N(u)V_N(v)]=N B\left(\|u-v\|^2_N\right).\label{eqn:intro:V-def}\]
	where here $\|x\|^2_N=(x,x)_N$ where $(x,y)_N=N^{-1}\sum_{i=1}^{N}x_iy_i$ and $B:[0,\infty)\to [0,\infty)$ is some fixed function. It is a result of Schoenberg \cite{schoenberg} that for a fixed function $B$, such a function $V_N$ exists for each $N\ge 1$ if and only if $B$ admits a representation of the form
	\[B(x)=c_0+\int_0^{\infty}\exp(-\lam^2 x)\nu(d\lam),\label{eqn:B-decomposition}\]
	where $\nu$ is a finite non-negative measure on $(0,\infty)$ and $c_0\ge 0$. For convenience, we will assume that there is $\epsilon>0$ such that $\int_0^\infty \exp(\lam^2 \epsilon) \nu(d\lam)<\infty$ and that $c_0=0$. This implies, in particular, that $B$ extends to a smooth function on $(-\epsilon,\infty)$.
	
	We define our model as the function on point configurations $\b{u}\in (\R^N)^\Om$, which may be thought of as $\R^N$-valued functions on the discrete graph $\Omega$. As such, for $x\in \Om$, we will use the notation we will use the notation $\b{u}(x)\in \R^N$ to denote the point such that $\b{u}(x)_i=\b{u}_{(i,x)}$. However, when $\b{f}$ is a function taking values in $A^{\Om}$, for some space $A$, we will use the notation $f_x(y):=\b{f}(y)(x)$. 
	
	With these notations set, we define our Hamiltonian for $\b{u}\in (\R^N)^\Om$ by
	\[\cal{H}_N(\b{u})=\frac{1}{2}\sum_{x,y\in \Om}(\mu I-t\Delta)_{xy}(\b{u}(x),\b{u}(y))+\sum_{x\in \Om}V_{N,x}(\b{u}(x))+\sq{N}h\sum_{x\in \Om}\b{u}_1(x),\label{eqn:def:main-model-D}\]
	where here $V_{N,x}$ are a collection of i.i.d copies of $V_N$, and $\Delta$ denotes the graph Laplacian of $\Omega$ (where we take the sign convention where $\Delta$ is negative semi-definite, following \cite{mezardparisi}). We define the associated partition function, for any choice of inverse-temperature $\beta>0$, as
	\[Z_{N,\beta}=\int_{(\R^{N})^\Om}e^{-\beta \cal{H}_N(\b{u})}d\b{u},\]
	where $d\b{u}$ is the standard Lebesgue measure on $(\R^N)^{\Om}$. 
	
	The main result of our companion paper \cite{Paper1}, offers a computation of the limiting free energy of this model, which we now recall. Here and elsewhere in this paper, we will use a prefix of I to denote results from \cite{Paper1}. That is Theorem I.1.3 is Theorem 1.3 of \cite{Paper1}.
	
	\begin{theorem}[\cite{Paper1}, Theorem I.1.3]
		\label{theorem:intro:old euclidean result}
		We have a.s. that
		\[
		\lim_{N\to \infty}N^{-1}L^{-d}\log Z_{N,\beta}=\lim_{N\to \infty}N^{-1}L^{-d}\E \log Z_{N,\beta}=F_{\Om}(\beta,t,\mu),
		\label{eqn:intro:euclidean:parisi-formula}
		\]
		where $F_{\Om}(\beta,t,\mu)$ is an explicit deterministic quantity.
	\end{theorem}
	
	Our earlier expression $F_{\Om}(\beta,t,\mu)$ is recalled below in Theorem \ref{theorem:paper 1:euclidean result}, however, our first result is to give a simpler and more effective formula for it.
	
	For this we must define a Parisi functional. First, we will the normalized resolvent of $\Delta$ as 
	\[R_1(u;t):=\tr(uI- t\Delta)^{-1}, \label{def:intro:R}\] 
	where $\tr(A)=\frac{1}{|\Om|}\sum_{x\in \Om}A_{xx}$. The main function of interest is the functional inverse of $R_1$.
	
	\begin{defin}
		Let $K(*;t)$ be the functional inverse of the normalized resolvent $R_1(*;t)$. That, for fixed $t$, $K(*;t):(0,\infty)\to (0,\infty)$ is the unique function which satisfies
		\[\tr(K(u; t)I- t \Delta)^{-1}=u.\label{def:intro:K}\]
	\end{defin}
	
	Now moving to the functional's domain, let $\mathscr{Y}(q)$ denote the subset of probability measures on $[0,q]$ whose support does not contain $q$. Given a choice of $\zeta\in \mathscr{Y}(q)$, we may define for $s\in [0,q]$
	\[\delta(s)=\int_s^q\zeta([0,u])du.\label{eqn:def:delta-P}\]
	Furthermore, let us take a choice of $q_*>0$ such that $\zeta([q_*,q])=0$. Given this data, we define the functional
	
	\[
	\begin{split}
		\cal{P}_{\beta,q}(\zeta)=\frac{1}{2}\bigg(\log\left(\frac{2\pi}{\beta }\right)+&\frac{\beta h^2}{\mu}-\beta\mu q+\beta(q-q_*)K(\beta(q-q_*);t)\\-\frac{1}{L^d}\log\left(K(\beta(q-q_*);t)I-t\Delta\right)
		+&\int_{0}^{q_*}\beta  K(\beta\delta(u);t)du
		-2\beta^2\int_0^q \zeta([0,u])B'(2(q-u)) du\bigg).	
	\end{split}
	\]
	
	\begin{remark}
		\label{remark:beginning point doesn't matter part}
		It is easily verified that if we define
		\[\Lambda^{-t\Delta}(u):=uK(u;t)\\-\frac{1}{L^d}\log\left(K(u;t)I+t\Delta\right),\]
		then we have that $(\Lambda^{-t\Delta})'(u)=K(u;t)$. In particular, for $q>q'>q*$, $\delta(q')=q-q'$, so that
		\[\Lambda^{-t\Delta}(\beta(q-q_*))=\Lambda^{-t\Delta}(\beta(q-q'))+\int_{q_*}^{q'}\beta K(\beta\delta(u);t)du.\]
		From this we see that $\cal{P}_{\beta,q}(\zeta)$ is independent of the choice of $q_*$.
	\end{remark}
	
	Our expression for $F_{\Om}(\beta,t,\mu)$ will be given by this functional evaluated at a specific choice of $(\zeta,q)$.
	
	\begin{theorem}[A New Parisi Formula]
		\label{theorem:intro:main:critical points euclidean}
		We have that
		\[F_{\Om}(\beta,t,\mu)=\cal{P}_{\beta,\qc}(\zetac),\]
		where  here the pair $(\qc,\zetac)$ is the unique $(q,\zeta)$ which solves the following equations:
		\[\beta \int_0^q\zeta([0,u])du=R_1(\mu;t), \label{eqn:intro:minimization eqn:Euclidean: Larkin}\]
		\[\zeta\left(\big\{s\in [0,q]:f_{\beta,q}(s)=\sup_{0<s'<q}f_{\beta,q}(s') \big\}\right)=1,\label{eqn:intro:minimization eqn:Euclidean: measure}\]
		where for $s\in (0,q)$, we define
		\[f_{\beta,q}(s)=\int_0^sF_{\beta,q}(u)du,\;\;F_{\beta,q}(s)=-2B'(2(q-s))+\int_0^u  K'(\beta \delta(s);t)\]
		with $\delta$ is as in (\ref{eqn:def:delta-P}).
	\end{theorem}
	
	A related result to this is the following alternative representation of $F_{\Om}(\beta,t,\mu)$ in terms of a variational problem.
	
	\begin{theorem}[A Variational Form of The Parisi Formula]
		\label{theorem:intro:main:Euclidean}
		We have that
		\[F_{\Om}(\beta,t,\mu)=\sup_{q\in (0,\infty)}\left(\inf_{\zeta \in \mathscr{Y}(q)}\cal{P}_{\beta,q}(\zeta)\right).\]
	\end{theorem}
	
	From the perspective of Theorem \ref{theorem:intro:main:Euclidean}, the pair $(\qc,\zetac)$ is a saddle point of the functional $\cal{P}_{\beta,q}(\zeta)$, with equations (\ref{eqn:intro:minimization eqn:Euclidean: Larkin}) and  (\ref{eqn:intro:minimization eqn:Euclidean: measure}) being the associated stationary equations for $q$ and $\zeta$, respectively. While these equations give an effective computational tool, we also show that this pair may be interpreted in terms of the key statistics of the model. 
	
	For this, we recall the Gibbs measure of the model. This is a random probability measure over $(\R^N)^{\Om}$ given by
	\[G_{\beta,N}(d\b{u})=\frac{1}{Z_{N,\beta}}e^{-\beta \cal{H}_N(\b{u})}d\b{u}\]
	where $d\b{u}$ is the standard Lebesgue measure on $(\R^N)^{\Om}$. We will use the traditional notations of $\<f(\b{u})\>$ and $\<f(\b{u},\b{u}')\>$ for the the expectations of a function $f$ with respect to $G_{\beta,N}$ and $G_{\beta,N}^{\otimes 2}$, respectively. We will also adopt this notation elsewhere in the document, when the underlying space is clear.
	
	Our next result effectively shows that, after an $h$-dependent shift, the quantity $\qc$ can be interpreted in terms of the typical value of mean-squared radius at any given site when sampling from the Gibbs measure. Meanwhile $\zetac$ can be interpreted in terms of the overlap distribution at a fixed site, for which reason we will refer to $\zetac$ as the Parisi measure.
	
	\begin{theorem}[Identification of the Overlap Distribution and Effective Radius]
		\label{theorem:intro:main:Euclidean parameters identification}
		For all $x\in \Om$
		\[\lim_{N\to \infty}\E\left\<\left|\|\b{u}(x)\|^2_N-\left(\qc+\frac{h^2}{\mu^2}\right)\right|\right\>=0.\label{eqn:intro:main:Euclidean parameters identification:radius}\]
		Moreover, any bounded continuous function $f:\R \to \R$ and choice of $x\in \Om$
		\[\lim_{N\to \infty}\E\<f((\b{u}(x),\b{u}'(x))_N)\>=\int_{0}^{q} f\left(r+\frac{h^2}{\mu^2}\right)\zetac(dr).\label{eqn:intro:main:Euclidean parameters identification:overlap}\]
	\end{theorem}
	
	\subsection{Results on the Replica Symmetric phase}
	
	We now turn to our results characterizing the Replica Symmetric (RS) phase of the model. We say that the model is RS, for fixed choice of $(\beta,\mu,t,L,d)$, if the corresponding Parisi measure $\zetac$ is a Dirac mass. This phase tends to have fairly simple behavior compared to the Replica Symmetry Breaking (RSB) phase where the Parisi measure $\zetac$ is supported on more than one point.
	
	Our main result provides an explicit analytic characterization of the RS-phase, as well as the value of the limiting free energy, and the pair $(\qc,\zetac)$ in this phase.
	
	\begin{theorem}[Characterization of The RS-Phase]
		\label{theorem:Euclidean RS description}
		Fix a choice of $(\beta,\mu,t,L,d)$. Then the model is RS if and only if 
		\[\sup_{0<s<R_1(\mu;t)}g_{\beta}(s)=g_{\beta}(R_1(\mu;t)),\label{eqn:ignore-luna-2}\]
		where here
		\[
		g_{\beta}(s)=\beta^2 B\left(\frac{2s}{\beta}\right)+\beta sK(s;t)-\frac{1}{L^d}\log(K(s;t)I-t\Delta)
		-s\left(2\beta B'\left(\frac{2}{\beta}R_1(\mu;t)\right)+\mu\right).\label{eqn:ignore-luna-10}
		\]
		Moreover, if we define for $i\ge 1$
		\[R_i(\mu;t)=\tr(\mu I-t \Delta)^{-i},\]
		then when the model is RS, the free energy is given by
		\[
		\cal{P}_{\beta,\qc}(\zetac)=\frac{1}{2}\left(\log\left(\frac{2\pi}{\beta}\right)+\frac{1}{L^d}\log \det(\mu I-t\Delta)+\beta^2 \left(B(0)-B\left(\frac{2}{\beta}R_1(\mu;t)\right)\right)\right),
		\]
		and if we define $q_{L,\beta}:=-2B'\left(\frac{2}{\beta}R_1(\mu;t)\right)R_2(\mu;t)$, the saddle point pair is given by
		\[(\zetac,\qc)=(\delta_{q_{L,\beta}},q_{L,\beta}+\beta^{-1}R_1(\mu;t)).\]
	\end{theorem}
	
	From this and Theorem \ref{theorem:intro:main:Euclidean parameters identification}, we note the following immediate corollary. 
	
	\begin{corr}
		When the model is RS, for any $x\in \Om$ we have that
		\[\lim_{N\to \infty}\E\<\|\b{u}(x)-\b{u}'(x)\|^2_N\>=\beta^{-1}R_1(\mu;t).\]
	\end{corr}	
	
	Note that the expression on the right is independent of the disorder term. In particular, the per-site distance between two samples in the RS-phase is the same if one sets $B=0$. In this case, the Hamiltonian is deterministic, and coincides with the Hamiltonian of the massive discrete Gaussian free field, where this formula is true for all $N$ by a quick computation.
	
	Our next corollary compares the formula for the quenched free energy to that of the annealed free energy. The annealed free energy is defined as $N^{-1}L^{-d}\log \E Z_{N,\beta}$, and is easily computed to be 
	\[N^{-1}L^{-d}\log \E Z_{N,\beta}=\frac{1}{2}\left(\log\left(\frac{2\pi}{\beta}\right)+\frac{1}{L^d}\log \det(\mu I-t\Delta)+\beta^2 B(0)\right).\]
	By Jensen's inequality we have that
	\[N^{-1}L^{-d}\E\log Z_{N,\beta}\le N^{-1}L^{-d}\log \E Z_{N,\beta},\]
	so the annealed free energy always serves as an upper bound for the quenched free energy, but often this inequality is strict. Our result gives an explicit formula for this gap.
	\begin{corr}
		When the model is RS, we have that
		\[\lim_{N\to \infty}\left(N^{-1}L^{-d}\log \E Z_{N,\beta}-N^{-1}L^{-d}\E\log Z_{N,\beta}\right)=\frac{\beta^2}{2}B\left(\frac{2}{\beta}R_1(\mu;t)\right).\]
		In particular,
		\[\lim_{\mu\to 0}\lim_{N\to \infty}\left(N^{-1}L^{-d}\log \E Z_{N,\beta}-N^{-1}L^{-d}\E\log Z_{N,\beta}\right)=0.\]
	\end{corr}
	
	We now turn to comparing our results to those in the physics literature, particularly focusing on the boundary between the RS and RSB phases. This boundary is known as the AT-line, after the pioneering physics work of de Almeida and Thouless \cite{AT} who studied the corresponding line for the Sherrington-Kirkpatrick model. While their description for the AT-line has still not been rigorously established, the works of Toninelli \cite{at-ton}, Talagrand \cite{talagrandII}, 
	Jagannath and Tobasco \cite{at-tobasco}, Chen \cite{at-chen}, Bolthausen \cite{at-bolt}, and Brennecke and Yau \cite{at-yau} have verified their formula to an increasingly small margin of error.
	
	Our above result explicitly gives the AT-line for the elastic manifold model, which we may now compare to results in the physics literature. The most famous of these is the zero-temperature boundary (i.e. $\beta=\infty$). Here, for fixed $t$, the boundary from RS to RSB is expected to occur at the Larkin mass. This quantity, which we denote by $\mu_{Lar}(\infty;t)$ is given by the unique solution to
	\[4B''(0)R_2(\mu_{Lar}(\infty;t);t)=1\label{eqn:larkin-condition:zero-temp}\]
	(see, for example, eqn. (49) of \cite{fyodorov-manifold-minimum}).
	We confirm below that indeed this defines the boundary at sufficiently low temperatures.
	
	\begin{theorem}
		\label{theorem:intro:zero:temperature larkin}
		Fixing $(t,B,L,d)$, and considering the behavior of the model at parameters $(\beta,\mu)$, the following statements hold:
		\begin{itemize}
			\item If $\mu<\mu_{Lar}(\infty;t)$, the model with parameters $(\beta, \mu)$ is RSB for all sufficiently large $\beta$.
			\item If $\mu\ge \mu_{Lar}(\infty;t)$, the model with parameter $(\beta,\mu)$ is RS for all $\beta$.
		\end{itemize}
	\end{theorem}
	
	As well as serving as the zero-temperature RS-RSB boundary, the Larkin mass is also the mass for which the Hamiltonian experiences "topological trivialization" \cite{topologytrivialization2}. This is where the Hamiltonian's expected number of critical points transitions from being exponentially large to one. This was predicted in the works of Fyodorov and Le Doussal \cite{fyodorov-manifold,fyodorov-manifold-minimum}, and made rigorous by the work of Bourgade, McKenna, and the first author \cite{gerardbenpaul} and Xu and Zeng \cite{xuzengelasticmanifold}. In particular, the formula of Xu and Zeng for asymptotic expected minimum value of the Hamiltonian above the Larkin mass is recovered from our free energy formula in Theorem \ref{theorem:Euclidean RS description} after taking $\beta\to \infty$.
	
	We mention that this relation the expected number of critical points, and RS-RSB boundary at zero-temperature was first found in the spherical SK-model by Cugliandolo and Dean \cite{topologytrivialFirst} and studied in more detail by Fyodorov and Le Doussal \cite{topologytrivialization2}. This has been expanded since by to pure spherical spin glass models  by Fyodorov \cite{topology3}, mixed spherical spin glass models by Belius, \v{C}ern\'{y}, Nakajima and Schmidt \cite{topology4}, and multi-species spherical spin models by Huang and Sellke \cite{topologyms}.
	
	Moreover, the Larkin mass is believed to be related to the critical parameter in the pinning-depinning transition (see \cite{larkin-review} for a review of this physical theory, as well as the discussion in \cite{fyodorov-manifold-minimum}). However, no rigorous results in this direction appear to be known.
	
	For our next result, we fix a positive temperature $\beta>0$. Then varying $\mu$, one may also try to determine if there is a critical mass determining the RS-RSB boundary. The quantity expected to fill this role is the positive temperature Larkin mass $\mu_{Lar}(\beta;t)$, and should satisfy the following equation:
	\[4B''\left (\frac{2}{\beta}R_1(\mu;t)\right)R_2(\mu;t)=1.\label{eqn:intro:fake-larkin}\]
	
	This appears in the original work of M{\'e}zard and Parisi \cite{mezardparisi} as (4.8) with $k=0$ (see also eqns. 17-18 of \cite{le2008cusps}). We show how one may arrive at this condition from Theorem \ref{theorem:Euclidean RS description} as a local stability condition for RS-solution. In particular, note that in the notation of Theorem \ref{theorem:Euclidean RS description}, one always has that $g_\beta'(R_1(\mu;t))=0$. Moreover, noting that one has that $K'(R_1(s;t);t)=-R_2(s;t)^{-1}$, one computes
	\[g_\beta''(R_1(\mu;t))=4 B''\left(\frac{2}{\beta}R_1(\mu;t)\right)-R_2(\mu;t)^{-1}.\label{eqn:intro-gprime}\]
	In particular, (\ref{eqn:intro:fake-larkin}) is equivalent to the equation $g_\beta''(R_1(\mu;t))=0$, and so marks the boundary of the regime where $R_1(\mu;t)$ is a local maximum of $g_\beta$ (i.e. $g_\beta''(R_1(\mu;t))> 0$). 
	
	However, while being a local maximum is necessary to be a global maximum, it is not sufficient. In particular, we know that the model is RSB if
	\[4B''\left (\frac{2}{\beta}R_1(\mu;t)\right)R_2(\mu;t)<1,\]
	but it is not clear the AT-line is given by the equation (\ref{eqn:intro:fake-larkin}).
	
	However, we will show that this equation does capture two aspects of the AT-line, both discussed in the work of Le Doussal and Wiese \cite{FRGDouWeise}. Here (\ref{eqn:intro:fake-larkin}) is discussed in more depth, and shown to also be the line at which FRG develops an analytic discontinuity. However, they also point out that (\ref{eqn:intro:fake-larkin}) may lack any solutions for sufficiently high temperature, in which case the RS solution remains stable for all masses, which we confirm. In particular, we show that there is some temperature $\beta_{DW}(t)\ge 0$, such that solutions exist for $\beta\ge \beta_{DW}(t)$, and for $\beta\le \beta_{DW}(t)$, the system is RS for any choice of $\mu$.
	
	However, for temperatures below $1/\beta_{DW}$, we have a more important subtly, namely, that (\ref{eqn:intro:fake-larkin}) may have multiple solutions. To fix this, when $\beta\ge \beta_{DW}(t)$, we define the Larkin mass as the largest value of $\mu$ which solves (\ref{eqn:intro:fake-larkin}). With this definition, we are able to show that $\mu_{Lar}(\beta;t)$ is indeed the largest critical mass. In particular, global stability is implied by the local stability for masses above $\mu_{Lar}(\beta;t)$. 
	
	\begin{theorem}
		\label{theorem:larkin-mass}
		Fixing $(t,B,L,d)$, and considering the behavior of the model at parameters $(\beta,\mu)$, the following statements hold:
		\begin{enumerate}
			\item If $\beta\le \beta_{DW}(t)$, then for any $\mu$, the model with parameters $(\beta,\mu)$ is RS.
			\item If $\beta>\beta_{DW}(t)$, and $\mu\ge \mu_{Lar}(\beta)$, then the model with parameters $(\beta,\mu)$ is RS.
		\end{enumerate}
	\end{theorem}
	
	However, these results leave open the behavior below the Larkin mass at positive temperature. This problem is discussed by \cite{FRGDouWeise}, particularly in their Appendix E, and the region of ambiguity appears as the grey region of their Figure 10. 
	
	In essence, this boundary is resolved by our Theorem \ref{theorem:Euclidean RS description}. However, we also find in this regime a complete break between the local and global stability criterion. Indeed, while the curve given by (\ref{eqn:intro:fake-larkin}) encloses much of the RSB region, and in many cases they coincide, in general they differ even in the case of $L^d=1$. Indeed, as Figure 1 shows, if at a fixed temperature one starts lowering the mass (starting at the Larkin mass), it is possible to exit the RSB phase, re-enter it, and then exit again, all at points which are not solutions to (\ref{eqn:intro:fake-larkin}). In our companion paper \cite{Paper3}, we show that a similar failure can occur for the continuum model with $d=1$ and $L\to \infty$.
	
	\begin{figure}[h]
		\includegraphics[width=17cm]{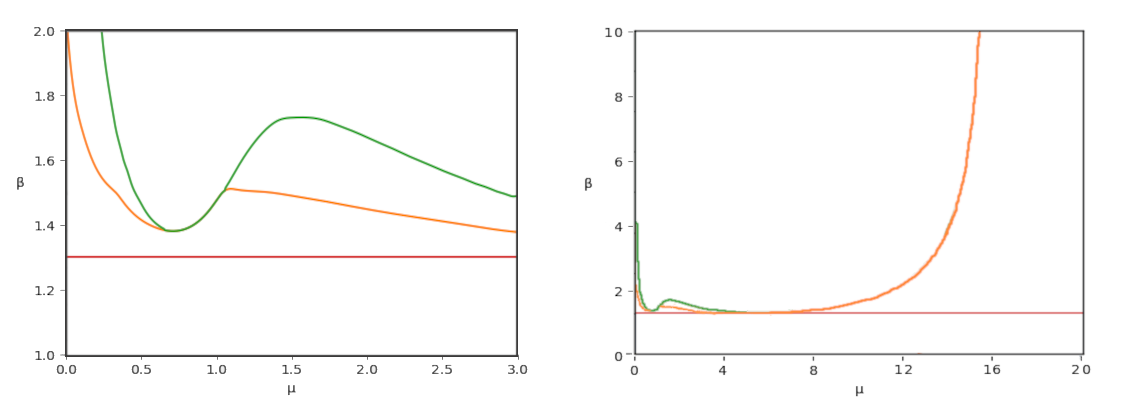}
		\caption{The phase diagram for $L^d=1$ and $B(x)=e^{-x}+e^{-8x}$ in terms of $(\beta,\mu)$. Solutions to (\ref{eqn:intro:fake-larkin}) are plotted in orange, the AT-line is in green, and the fixed inverse-temperature $\beta_{DW}$ is in red.}
		\label{figure}
	\end{figure}

	\subsection{Results for the Spherical Model}
	
	We now introduce an additional model, which is similar to the above except that instead of taking an isotropic field on $\R^N$ in the presence of a quadratic potential, we consider an isotropic field on the sphere. Such isotropic fields are more well known, as they form the family of spherical spin glass models, which have seen extensive study.
	
	In particular, we fix, for each $N\ge 1$, a choice of centered Gaussian random field $H_N:S_N\to \R$ with isotropic covariance given for $\sigma,\tau\in S_N$ by
	\[\E[H_N(\sigma)H_N(\tau)]=N\xi((\sigma,\tau)_N),\label{eqn:intro:spin glass covariance}\]
	where $\xi:[-1,1]\to \R$ is some fixed function. As before, it is a result of Schoenberg \cite{schoenbergSpheres} that for such a family to exist for all $N\ge 1$ it is necessary and sufficient that $\xi$ is an analytic function on $[-1,1]$ with positive coefficients. That is, there are coefficients $\beta_p\ge 0$ such that
	\[\xi(x)=\sum_{p=0}^{\infty}\beta_p^2 x^p,\label{eqn:xi-decomposition}\]
	and such that $\xi(1)<\infty$. We will assume as before that for some  sufficiently small $\epsilon>0$, $\xi(1+\epsilon)<\infty$, so that $\xi$ extends to an analytic function on $(-1-\epsilon,1+\epsilon)$.
	
	We then take a sequence of i.i.d copies of $(H_{N,x})_{x\in \Om}$. We define the spherical version of our Hamiltonian by taking, for $\b{u}\in S_N^{\Om}$
	\[\cal{H}_{N,\rm{Sph}}(\b{u})=\frac{1}{2}\sum_{x,y\in \Om}t \Delta_{xy}(\b{u}(x),\b{u}(y))+\sum_{x\in \Om}H_{N,x}(\b{u}(x))+\sq{N}h\sum_{x\in \Om}\b{u}_1(x).\label{eqn:def:intro-spherical-model}\]
	Associated with this, we may define as well the partition function
	\[Z_{N,\beta,\rm{Sph}}=\int_{S_N^{\Om}}e^{-\beta \cal{H}_{N,\rm{Sph}}(\b{u})}\omega(d\b{u}),\]
	where here $\omega(d\b{u})$ is the surface measure on $S_N^{\Om}$ induced by the inclusion $S_N^{\Om}\subseteq (\R^N)^{\Om}$. As before, in \cite{Paper1} we gave a computation for free energy $N^{-1}L^{-d}\log Z_{N,\beta,\rm{Sph}}$.
	
	\begin{theorem}[\cite{Paper1}, Theorem I.1.6]
		We have a.s. that
		\[
		\lim_{N\to \infty}N^{-1}L^{-d}\log Z_{N,\beta}=\lim_{N\to \infty}N^{-1}L^{-d}\E \log Z_{N,\beta}=F_{\Om,\mathrm{Sph}}(\beta,t),
		\]
		where here $F_{\Om,\mathrm{Sph}}(\beta,t)$ is an explicit deterministic quantity.
	\end{theorem}
	
	As in the case above, our main goal again is to make the formula for $F_{\Om,\mathrm{Sph}}(\beta,t)$ more tractable.
	
	Given a choice of $\zeta\in \mathscr{Y}(1):=\mathscr{Y}$, we define for $s\in [0,1]$
	\[\delta(s)=\int_s^1\zeta([0,u])du.\label{eqn:def:delta-A}\]
	Furthermore, let us take a choice of $q_*>0$ such that $\zeta([q_*,1])=0$. Given this data, we may define the functional
	\[\begin{split}
		\cal{B}_{\beta}(&\zeta)=\frac{1}{2}\bigg(\log \left(\frac{2\pi}{\beta}\right)+\beta h^2K(\beta \delta(0);t)^{-1}+ \beta(1-q_*)K(\beta(1-q_*);t)\\
		&-\frac{1}{L^d}\log (K(\beta(1-q_*);t)I-t\Delta)+ \int_{0}^{q_*}\beta K(\beta \delta(u);t)du+\beta^2\int_0^1 \zeta([0,u])\xi'(u)du\bigg).	
	\end{split}\]
	
	A similar remark as to Remark \ref{remark:beginning point doesn't matter part} implies the choice of $q_*$ doesn't affect the value of the function. Our first result here is a computation of the quenched free energy.
	
	\begin{theorem}[A New Spherical Parisi Formula]
		\label{theorem:intro:main:sphere}
		We have that
		\[F_{\Om}(\beta,t,\mu)=\inf_{\zeta\in \mathscr{Y}}\cal{B}_{\beta}(\zeta).\]
	\end{theorem}
	
	We give this result in the form a minimizer here, as in the case of $L^d=1$, one may check that it coincides with the general formula for the limiting quenched free energy of the normal spherical spin glass model obtained in the even case by Talagrand \cite{talagrandOG} and the general case by Chen \cite{weikuo}. As in the case of \cite{talagrandOG}, we may also use the convexity of the functional to obtain a similar characterization of the unique minimizer.
	
	\begin{theorem}[Identification of the Overlap Distribution]
		\label{theorem:intro:parisi-measure:sphere}
		The functional $\cal{B}_{\beta}$ is strictly convex on $\mathscr{Y}$ with a unique minimizer, $\zetac$, which is the unique solution to the following equation:
		\[\zeta\left(\big\{s\in [0,1]:f_{\xi}(s)=\sup_{0<s'<1}f_{\xi}(s') \big\}\right)=1.\label{eqn:intro:minimization sphere}\]
		where for $s\in (0,1)$, we define
		\[f_{\beta,\xi}(s)=\int_0^s F_{\beta,\xi}(u)du,\;\; F_{\beta,\xi}(s)=-h^2K(\beta \delta(0);t)^{-2}K'(\beta \delta(0);t)+\xi'(s)+\int_0^s K'(\beta \delta(u);t)du.\]
		and $\delta$ is as in (\ref{eqn:def:delta-A}).
	\end{theorem}
	
	Finally, we have a partial result which characteristics $\zetac$ in terms of the model. Unfortunately, we will need an additional assumption on $\xi$ in this case to obtain this characterization, as indeed, even in the one-site case (i.e. when $|\Om|=1$), a general characterization is not known.
	
	The condition we will need is the following: a mixing function $\xi$ is generic if
	\[\text{span}\{x^p:p\in \N,\text{ such that }\beta_p\neq 0\}\label{def:generic}\]
	is dense in $C([-1,1],\R)$ with respect to the uniform topology. Under this assumption, we obtain a version of Theorem \ref{theorem:intro:main:sphere}, which additionally provides an interpretation of the Parisi measure in terms of the Gibbs measure. In this case, this is given by
	\[\<f(\b{u})\>_{\rm{Sph}}=\frac{1}{ Z_{N,\beta,\rm{Sph}}}\int_{S_N^{\Om}}e^{-\beta \cal{H}_{N,\mathrm{Sph}}(\b{u})}\omega(d\b{u}).\]
	Essentially, we again find that $\zetac$ encodes the law of the limiting overlap distribution. 
	
	\begin{theorem}
		\label{theorem:intro:main:sphere-generic}
		Then if $\xi$ is generic, then we have for any continuous function $f:[-1,1]\to \R$ and $x\in \Om$ that
		\[\lim_{N\to \infty}\E\<f((\b{u}(x),\b{u}'(x))_N)\>_{\rm{Sph}}=\int_0^1 f(r)\zeta_c(dr).\]
	\end{theorem}
	
	Finally, we have a similar characterization of the RS-phase in this case, though in this case the results are less explicit.
	
	\begin{theorem}
		\label{theorem:Spherical RS description}
		If the model is RS for some choice of $\beta$ then the free energy is given by
		\[
		\begin{split}
			\frac{1}{2}\inf_{q\in [0,1)}\bigg(&\log\left(\frac{2\pi}{\beta}\right)+\beta h^2K(\beta (1-q);t)^{-1}+\beta K(\beta(1-q);t)\\
			&-\frac{1}{L^d}\log \det(K(\beta(1-q)-t\Delta))+\beta^2(\xi(1)-\xi(q))\bigg),
		\end{split}
		\]
		and the Parisi measure is given by $\zetac=\delta_{q_*}$, where $q_*$ is the minimizer of this expression.
		\indent
		Moreover, a model is RS if and only if the minimizer $q_*$ satisfies
		\[\sup_{q_*\le s\le 1}g_\beta(s)=g_\beta(q_*),\label{eqn:intro:min equation RS for sphere}\]
		where here 
		\[\begin{split}
			g_\beta(s)=\beta^2 \xi(s)+\beta(1-s)K(\beta(1-s);t)-\frac{1}{L^d}\log(K(\beta(1-s);t)-t\Delta)\\-s\left(\beta^2\xi'(q_*)-\beta K(\beta(1-q_*);t\right).
		\end{split}\]
	\end{theorem}

	\subsection{Overview}
	
	We now provide an outline of our proofs, as well as a structural overview of the paper. The core of our method can be explained in how we show that the expression for $F_{\Om}(\beta,t,\mu)$ in Theorem \ref{theorem:intro:old euclidean result} is equal to the expression given in Theorem \ref{theorem:intro:main:Euclidean}. The reason for this is that in the course of proving this result, the methods we employ can be generalized to show most of the other main results.
	
	To outline this method, we need to discuss the expression for $F_{\Om}(\beta,t,\mu)$ derived in Theorem \ref{theorem:intro:old euclidean result}. The exact version of this is recalled below in Theorem \ref{theorem:paper 1:euclidean result}. However, to explain the key improvement, we only need to recall that it expresses the limiting free energy as an optimization problem of the form
	\[F_{\Om}(\beta,t,\mu)=\sup_{\b{q}\in (0,\infty)^{\Om}}\left(\inf_{(\zeta,\b{\Phi})\in \widehat{\mathscr{Y}}(\b{q})}\widehat{\cal{P}}_{\b{q}}(\zeta,\b{\Phi})\right).\]
	The functional $\widehat{\cal{P}}_{\b{q}}(\zeta,\b{\Phi})$ is not important now, as our main focus is on the domains of these optimizations. Already the supremum is over $\b{q}\in (0,\infty)^{\Om}$, a space of dimension $L^d$. Moreover, the infimum is over pairs $(\zeta,\b{\Phi})$, and while $\zeta\in \mathscr{Y}(q_{av})$ where $q_{av}$ is the average of the values of $\b{q}\in (0,\infty)^{\Om}$, the other parameter is a continuous function $\b{\Phi}:[0,q_{av}]\to \prod_{x\in \Om}[0,\b{q}(x)]$, subject to some conditions. In particular, both domains we are optimizing over are enormous, and growing in $L$.
	
	In effect, Theorem \ref{theorem:intro:main:Euclidean} shows that we may obtain the same value for $F_{\Om}(\beta,t,\mu)$ by restricting the supremum to values for which $\b{q}$ is the constant vector, and the infimum to the case where $\b{\Phi}$ is the identity. Not only does this shrink the domain of optimization massively to a now $L$-independent domain, but when evaluated on this domain, the expression for $\hat{\cal{P}}_{\beta,\b{q}}$ simplifies significantly to the $\cal{P}_{\beta,q}(\zeta)$ given above.
	
	We perform this restriction in two steps. In the first, we show that if we assume that $\b{q}\in (0,\infty)^{\Om}$ is a constant vector, then we may restrict the infimum to the case where $\b{\Phi}$ is the identity. We will show that this effectively a special case of the same question for the spherical model. In particular, this restriction follows directly from our formula for the spherical model, Theorem \ref{theorem:intro:main:sphere}.
	
	This is done in Section \ref{section:identification}, which roughly provides the proofs for Theorem \ref{theorem:intro:main:sphere},  as well as Theorems \ref{theorem:intro:parisi-measure:sphere} and \ref{theorem:intro:main:sphere-generic}. For the first theorem, we use a continuity argument in $\xi$ to reduce ourselves to the generic case. For such models, the gradient of the Parisi functional in the coefficients of $\b{\xi}$ provide sufficient information to identity any measure which is an minimizer of $\widehat{\cal{B}}$, the spherical version of $\widehat{\cal{P}}_{\b{q}}$. However, when combined with the inherent symmetry of the model, all of these derivatives coincide, which essentially forces the minimizer to be symmetric. This simultaneously shows the first two theorems. The final theorem then is an exercise in the variational calculus.
	
	The next step in showing Theorem \ref{theorem:intro:main:Euclidean}, is to show that the supremum over $\b{q}\in (0,\infty)^{\Om}$ occurs on a constant vector. This we do over two sections. First, we have Section \ref{section:minimization}, which establishes the existence of minimizers of $\widehat{\cal{B}}$, as well as a set of formulas obeyed by them. The method used here is just a mixture of variational calculus and analysis of the resulting minimization equations. These results immediately give formulas for the minimizers of $\widehat{\cal{P}}_{\b{q}}$ as well. 
	
	Next, in Section \ref{section:concavity}, we complete the proof of Theorem \ref{theorem:intro:main:Euclidean} and Theorem \ref{theorem:intro:main:critical points euclidean} as well. The main result is Proposition \ref{prop:concavity:unique concavity in q}, which shows that the  $\b{q}\mapsto \inf_{(\zeta,\b{\Phi})\in \mathscr{Y}(\b{q})}\cal{P}_{\beta,\b{q}}(\zeta,\b{\Phi})$ is concave and has a unique solution. The main tool are some identifications between subsets of the domains $\mathscr{Y}(\b{q})$ as one varies $\b{q}$. Employing these identifications, the functional $\cal{P}_{\beta,\b{q}}(\zeta,\b{\Phi})$ is essentially linear in the data of $(\zeta,\b{\Phi})$. The main difficulty is showing that we may restrict the domains for the minimization of $\cal{P}_{\beta,\b{q}}(\zeta,\b{\Phi})$, locally around a given point, to use these identifications uniformly. For this, we use the identities satisfied by the minimizers obtain in Section \ref{section:minimization}. This is enough to establish concavity, but not strict concavity, so we are left with the problem of showing there is unique maximizing $\b{q}$. For this we use the convexity to show that the set of maximizers forms a non-empty convex set. Thus if there is not a unique maximizer, this set will contain a line of maximizers. Differentiating along this line, and using the above identifications, this implies an additional equation which contradicts the equations for minimizers established before.
	
	Now we have Section \ref{section:euclidean proofs}, which establishes Theorem \ref{theorem:intro:main:Euclidean parameters identification}. After reducing to the case of $h=0$, the proof of (\ref{eqn:intro:main:Euclidean parameters identification:radius}), which identifies $\qc$ in terms of the effective squared radius, essentially follows from uniqueness of the maximizer of the previous section, as we show that one may identify $\inf_{(\zeta,\b{\Phi})\in \mathscr{Y}(\b{q})}\cal{P}_{\beta,\b{q}}(\zeta,\b{\Phi})$ with the contribution to the partition function from configurations whose squared radii are roughly $\b{q}$. However (\ref{eqn:intro:main:Euclidean parameters identification:overlap}) is much more complicated. Morally, this result should follow in the same way one obtains the spherical result, Theorem \ref{theorem:intro:main:sphere-generic}, as in view of (\ref{eqn:intro:main:Euclidean parameters identification:overlap}) one may imagine that Gibbs measure is essentially concentrated on a product of spheres of radius $\sqrt{\qc}$. If this were true, you could use the fact that the restricted Hamiltonian turns out to be a generic spin glass, and obtain it from Theorem \ref{theorem:intro:main:sphere-generic}. However, justifying this replacement is beyond our means. Instead, we show that we may replace the Gibbs measure with one supported on thin annuli of radius roughly $\sqrt{\qc}$. When this radius is small enough, we may attempt to mimic the proof of Section \ref{section:identification} by finding certain centered Gaussian fields on these annuli and correlate them to our restricted field. Once we produce enough of these, and send the thickness to zero, we are able to proceed in the same spirit as Section \ref{section:minimization}.
	
	Finally, we complete the paper with Section \ref{section:parisi-larkin}, which establishes our results for the RS-regime. In particular, it is where we prove Theorems \ref{theorem:Euclidean RS description},  \ref{theorem:intro:zero:temperature larkin}, \ref{theorem:larkin-mass} and \ref{theorem:Spherical RS description}. These follow from direct analysis of the saddle point equations which specify $(\qc,\zetac)$, and in particular, is essentially self-contained.
	
	\subsection{History and Related Works}
	
	We will now review some additional related works, beginning with those involving the one-site case (i.e. when $|\Om|=1$), where much more is known. Starting with the spherical model (\ref{eqn:def:intro-spherical-model}), an expression for the quenched free energy given by Crisanti and Sommers \cite{crisantisommersOG}, employing the replica-symmetry breaking scheme pioneered by Parisi \cite{parisiOG} compute the quenched free energy of Ising spin glass model. A rigorous derivation of these formulas in the even case was obtained by Talagrand \cite{talagrandOG, talagrandIsingOG}, building upon interpolation method used by Guerra to obtain the upper bound in the Ising case \cite{guerraOG}. The general Ising case was later obtained by Panchenko \cite{panchenkoUnipartite}, with the general spherical case soon after being obtained by Chen \cite{weikuo}.
	
	We also mention that a number of computations for the free energy of different multi-site models have recently been obtained. Many of these rely on the foundational work of Panchenko \cite{panchenkoms}, who computed the free energy of the multi-species Ising spin glass model in the convex case. Following this, Bates and Sohn \cite{erik,erikCS} computed the free energy of the multi-species spherical spin glass model, again in the convex case. We also mention the work of Ko \cite{justin}, which computes the free energy in the spherical case with contained overlaps.
	
	As we have mentioned above, another approach to studying this model is through its topological complexity. The computation of the annealed complexity of the elastic manifold model done by Fyodorov and Le Doussal \cite{fyodorov-manifold,fyodorov-manifold-minimum}, and later made rigorous through the work of Bourgade, McKenna, and the first author \cite{gerardbenpaul,gerardbenpaulCompanionPaper}. The works \cite{fyodorov-manifold,fyodorov-manifold-minimum} also build on related annealed complexity computations of Fyodorov et. al. \cite{Fyo04, FS07, FB08,fyodorovepointinabox}. Many of these results were later made rigorous, and extended, by Auffinger and Zeng \cite{tucazeng1,tucazeng2}.
	
	
	Much more is known about the complexity in the one-site spherical case. In particular, the annealed complexity in this case was computed in works by Auffinger, \v{C}ern\'{y}, and the first author \cite{tucapure,tucamixed}. These results match those obtained earlier in physics literature \cite{complexityOG,complexityminimaOG}, for which they correspond to the zero-temperature case of a more general computation for the complexity of TAP states. 
	
	However, outside of the topological trivialization regime, computation of the annealed complexity only yields a lower bound on the ground state energy. Moreover, this bound is known not to be tight in general (see \cite{chensen}). So in general, to obtain the correct bound, one must compute the quenched complexity. A remarkable result of Subag \cite{subag}, later extended by Subag and Zeitouni \cite{subagofer}, that in the case of the pure $p$-spin glass model, the annealed and quenched complexity coincide. This was further extended to the quenched complexity of low-lying local minima for a certain class of $1$RSB mixed spin glass models by Subag, Zeitouni, and the first author in \cite{subag1RSB}. Not only does this lead to a computation of the ground state energy, but such results have led to a more complete understanding of the geometry of the low-temperature Gibbs measure, as shown in follow-up works by Subag and Zeitouni \cite{pspin-second-application1, pspin-second-application2}. In addition, these have played an important role in studying the associated Langevin dynamics, as in the work of Gheissari and Jagannath \cite{pspin-second-application3}. In addition, the work of Jagannath and the first author \cite{aukoshgerardshattering} discusses further ramifications of this study to concepts such as shattering and metastability.
	
	However, outside of these special cases, such results are scarce, with the only full extension being to the non-gradient generalization of the spherical pure $p$-spin model \cite{kivimae-non-gradient}, some partial results on the spherical bipartite model \cite{kivimae-bipartite}, both by the second author, and a general result for the total local minima of the mixed spin glass model by Belius and Schmidt \cite{david}, though with no control on their energy values. In all cases, one simply shows that the annealed and quenched complexity coincide, as no rigorous method to compute the quenched complexity outside of this case is known.
	
	Finally, we mention that identification of the parameters arising in Parisi-type variational formulas has also been studied in a number of spin glass models. In one-site models, this was approached in the works of Talagrand \cite{talagrand-differential} and Panchenko \cite{panchenko-differential}. Moreover, recent work by Bates and Sohn \cite{InvariantOptimizer1} used these methods of identification to simplify the Parisi-formula for the balanced Potts spin glass model. Following this was a recent result of Issa \cite{InvariantOptimizers2} which did the same for a symmetric version of the multi-species spin glass model. With an added genericness assumption, both of these works also allow one to identify this measure with an appropriate limiting overlap distribution. In addition, both use a method similar to the one we employ in Section \ref{section:identification} to simplify a symmetric multi-site variational formula to a single parameter.

	\section{Proof of Theorems \ref{theorem:intro:main:sphere}, \ref{theorem:intro:parisi-measure:sphere}, and \ref{theorem:intro:main:sphere-generic} \label{section:identification}}
	
	In this section we establish our simplified formula for the spherical model, and in particular give the proofs for Theorems \ref{theorem:intro:main:sphere}, \ref{theorem:intro:parisi-measure:sphere}, and \ref{theorem:intro:main:sphere-generic}. This establishes all of our results for the spherical model except for Theorem \ref{theorem:Spherical RS description}, which describes the model's RS phase. Before proceeding, we note that it is clear that it suffices to demonstrate our results in the case of $\beta=1$, which we will assume for the remainder of the section. In addition, we will write $Z_{N,\mathrm{Sph}}:=Z_{N,1,\mathrm{Sph}}$ and $\cal{B}:=\cal{B}_1$.
	
	We now outline the main steps in our method. To prove Theorem \ref{theorem:intro:main:sphere}, we will first only work in the case of generic $\xi$ (obtaining the general case later by a continuity result). We then recall a more general inhomogeneous Parisi functional, which gives the free energy for a more general class of inhomogeneous models recalled below. We show, though, that if one differentiates the infimum of this Parisi functional in the model parameters, then one can obtain certain statistics of any minimizer of the Parisi functional. When the model is generic, these are enough to uniquely determine the minimizer. 
	
	Moreover, the infimum of this Parisi functional is the limit as $N\to \infty$ of these models average free energy. Using this, we are able to relate the statistics of a minimizer to the limits of certain derivatives of the average free energy, which we relate to certain statistics of the Gibbs measure. As these will obviously be independent of $x\in \Om$ for the original spherical model considered above, this shows that the unique minimizer of the Parisi functional must also be independent of $x\in \Om$. This is enough to demonstrate Theorems \ref{theorem:intro:main:sphere} and \ref{theorem:intro:main:sphere-generic}, and in effect does so simultaneously in the generic case. After this all that is left is to Theorem \ref{theorem:intro:parisi-measure:sphere}, which is followed by a direct computation.
	
	For one-site models, differentiability of the infimum of the Parisi functional goes back to work of Talagrand \cite{talagrand-differential}. This was followed by work of Panchenko \cite{panchenko}, who used a similar method as outlined above to show the one-site version of Theorem \ref{theorem:intro:main:sphere-generic}. We essentially follow their line of reasoning, however in our case it allows us to simplify our formula considerably in the generic case, and in the general case as well by continuity. 
	
	We now recall some results and notations from our companion paper \cite{Paper1}. Let us fix a finite set $\Om$ and a positive semi-definite matrix $D\in \R^{\Om\times \Om}$. Next, take a collection, $\b{\xi}=\{\xi_x\}_{x\in \Om}$, such that each $\xi_x$ is a mixing function as in (\ref{eqn:xi-decomposition}), and let $\{H_{N,x}\}_{x\in \Om}$ be an independent family of functions with $H_{N,x}$ associated to the mixing function $\xi_x$ as in (\ref{eqn:intro:spin glass covariance}). Finally, let us fix a vector of external field strengths $\b{h}\in \R^{\Om}$. We then define a Hamiltonian, generalizing (\ref{eqn:def:intro-spherical-model}) by 
	\[\cal{H}_{N,D,\b{\xi},\b{h},\rm{Sph}}(\b{u})=\frac{1}{2}\sum_{x,y\in \Om}D_{xy}(\b{u}(x),\b{u}(y))+\sum_{x\in \Om}H_{N,x}(\b{u}(x))+\sq{N}\sum_{x\in \Om}\b{h}(x)\b{u}_1(x).\label{eqn:def:general-spherical-model}\]
	Associated with this, we may define as well the partition function
	\[Z_{N,\rm{Sph}}(D,\b{\xi},\b{h})=\int_{S_N^{\Om}}e^{- \cal{H}_{N,D,\b{\xi},\b{h},\rm{Sph}}(\b{u})}\omega(d\b{u}).\]
	
	The free energy of this partition function was computed in Theorem I.1.6. To state this result, we need to set more notation First, if one has $\b{u}\in \R^\Om$ and $D\in \R^{\Om\times \Om}$, we will use the shorthand $D+\b{u}\in \R^{\Om\times \Om}$ to denote the matrix $[D+\b{u}]_{xy}=D_{xy}+\delta_{xy}\b{u}(x)$. When $A$ and $B$ are symmetric matrices, we will use the notation $A>B$ to denote that $A-B$ is positive definite.
	
	Next, we will need some functions associated with $D$. Existence and claims about these functions were shown in Appendix C of \cite{Paper1}. When $\b{u}\in \R^\Om$ is such that $D+\b{u}>0$, we define $\b{R}^D(\b{u})\in (0,\infty)^{\Om}$ by
	\[R_x^D(\b{u})=[(D+\b{u})^{-1}]_{xx}\text{ for } x\in \Om.\]
	Next, we define $\b{K}^D$ as the functional inverse of $\b{R}^D$, so that for $\b{u}\in (0,\infty)^{\Om}$,
	\[[\left(D+\b{K}^D(\b{u})\right)^{-1}]_{xx}=\b{u}(x) \text{ for } x\in \Om.\label{eqn:def:KD}\]
	We also define for $\b{u}\in (0,\infty)^{\Om}$,
	\[\Lambda^D(\b{u})=\frac{1}{|\Om|}\left(\sum_{x\in \Om}K_x^D(\b{u})\b{u}(x)-\log \det(D+\b{K}^D(u))\right),\label{eqn:def:Lambda}\]
	This is in fact the Legendre transform of the concave function $\b{u}\mapsto \frac{1}{|\Om|}\log \det(D +\b{u})$, so that in particular $|\Om|\D \Lambda^D(\b{u})=\b{K}^D(\b{u})$.
	
	Next we focus on the domain of our functional. We first take a measure $\zeta\in \mathscr{Y}$. Then, for each $x\in \Om$, we choose a continuous, coordinate-wise non-decreasing function $\Phi_x:[0,1]\to [0,1]$. We require that for each $s\in [0,1]$ we have that
	\[\frac{1}{|\Om|}\sum_{x\in \Om}\Phi_x(s)=s.\label{eq:condition for Psi-spherical}\]
	Next we will assume as well that there is $0<q_*<q_t$ is such that $\zeta([0,q_*))=0$ and $\Phi_x(q_*)<1$ for each $x\in \Om$. We will denote the set of pairs $(\zeta,\b{\Phi})$ satisfying these conditions as $\widehat{\mathscr{Y}}$.
	Given these choices we define the continuous coordinate-wise non-increasing function $\b{\delta}:[0,1]\to [0,\infty)^{\Om}$ by
	\[\delta_x(s)=\int_{s}^{1}\zeta([0,u])\Phi'_x(u)du.\label{eqn:section2:delta-1}\footnote{Note that as each $\Phi_x$ is non-decreasing, (\ref{eq:condition for Psi}) implies that each $\Phi_x$ is absolutely continuous. In particular, $\Phi_x'$ exists a.e. and is Lebesgue integrable.}\]
	Then, we may define a functional on $(\zeta,\b{\Phi})\in \widehat{\mathscr{Y}}$ by
	\[
	\begin{split}
		\widehat{\cal{B}}_{D,\b{\xi},\b{h}}(\zeta,\b{\Phi})=&\log(\sqrt{2\pi})+\frac{1}{2}\Lambda^D(\b{\delta}(q_*))+\frac{1}{2|\Om|}\sum_{x\in \Om}\bigg(\int_{0}^{q_*}K_x^D(\b{\delta}(q))\Phi'_x(q)dq\\
		&+\int_0^{1}\zeta([0,u])\xi_x'(\Phi_x(q))\Phi'_x(q)dq+\sum_{y\in \Om}[D+\b{K}^D(\b{\delta}(0))^{-1}]_{xy}\b{h}(x)\b{h}(y)\bigg).\label{eqn:def:hat-A}
	\end{split}
	\]
	
	We can now state the main result on spherical models coming from our companion paper.
	\begin{theorem}[Theorem I.1.6]
		\label{theorem:bad Euclidean}
		We have that
		\[\lim_{N\to \infty}|\Om|^{-1}N^{-1}\E \log  Z_{N,\rm{Sph}}(D,\b{\xi},\b{h})= \inf_{(\zeta,\b{\Phi})\in \widehat{\mathscr{Y}}}\widehat{\cal{B}}_{D,\b{\xi},\b{h}}(\zeta,\b{\Phi}).\]
	\end{theorem}
	
	We will now focus on the spherical case. As such, we restrict ourselves for the moment to the case where there is some $(\xi,h)$ such that $\xi_x=\xi$ and $\b{h}(x)=h$ for all $x\in \Om$. In this case, we define
	\[\widehat{\cal{B}}_{D,\xi,h}(\zeta,\b{\Phi}):=\widehat{\cal{B}}_{D,\b{\xi},\b{h}}(\zeta,\b{\Phi}).\]
	Given Theorem I.1.6, we see that to show Theorem \ref{theorem:intro:main:sphere} we only need to show that
	\[\inf_{(\zeta,\b{\Phi})\in \widehat{\mathscr{Y}}}\widehat{\cal{B}}_{-t\Delta,\xi,h}(\zeta,\b{\Phi})=\inf_{\zeta\in \mathscr{Y}}\cal{B}(\zeta).\]
	A path for this is made clear by the following remark, which follows by a routine computation.
	\begin{remark}
		\label{remark: A and A hat}
		Let $\b{I}:[0,1]\to [0,1]^{\Om}$ denote the diagonal function such that $\b{I}_x(s)=s$ for $s\in [0,1]$ and $x\in \Om$ and choose $\zeta\in \mathscr{Y}$. Then we have that
		\[\widehat{\cal{B}}_{-t\Delta,\xi,h}(\zeta,\b{I})=\cal{B}(\zeta).\]
	\end{remark}
	
	Thus we are reduced to studying the minimizers of $\widehat{\cal{B}}_{-t\Delta,\xi,h}$. We find it helpful here to give the precise property of $-t\Delta$ we need to do so.
	
	\begin{defin}
		\label{def:transitive}
		For a finite set $\Om$, we say a symmetric matrix $M\in \R^{\Om\times \Om}$ is transitive if for each choice $x,y\in \Om$, there is a choice of permutation matrix $P\in \R^{\Om\times \Om}$, such that $[P]_{xy}=1$ and such that $P^{-1}MP=M$.
		
		Equivalently, $M$ is such that the subgroup of permutation matrices which commute with $M$ act transitively on the set $\Om$.
	\end{defin}
	
	This definition is adopted from the associated terminology for graphs. Indeed, one may check that a graph is vertex-transitive if and only if its graph Laplacian is transitive in the sense of Definition \ref{def:transitive}. In particular, the matrices $\mu  I-t \Delta$ and $\Delta$ satisfy Definition \ref{def:transitive}. 
	
	Our first main result will then be the following.
	
	\begin{theorem}
		\label{theorem:overview:spherical:generic and symmetric}
		Let us assume that $\xi$ is generic and that $D$ is transitive. Then there is a pair $(\zeta,\b{I})$ which minimizes $\widehat{\cal{B}}_{D,\xi,h}$ over $\widehat{\mathscr{Y}}$, so that in particular
		\[\inf_{(\zeta,\b{\Phi})\in \widehat{\mathscr{Y}}}\widehat{\cal{B}}_{D,\xi,h}=\inf_{\zeta\in \mathscr{Y}}\cal{B}_{D,\xi,h}.\label{eqn:theorem:A equals A hat}\]
		The measure $\zeta\in \mathscr{Y}$ is given specifically as the limiting law of the overlap distribution for any $x\in \Om$. More specifically, for any continuous function $f:[-1,1]\to \R$ and $x\in \Om$
		\[\lim_{N\to \infty}\E\<f((\b{u}(x),\b{u}'(x))_N)\>=\int_0^1 f(r)\zeta(dr),\]
		where $\<*\>$ is the Gibbs measure with respect to $\mathcal{H}_{N,D,\xi,h,\mathrm{Sph}}.$
	\end{theorem}
	
	Before proceeding to the proof of this result, we first see how we may obtain two of our desired theorems from it.
	
	\begin{proof}[Proof of Theorem \ref{theorem:intro:main:sphere} and \ref{theorem:intro:main:sphere-generic}] In the case where $\xi$ is generic, Theorem \ref{theorem:intro:main:sphere} follows immediately from Theorem \ref{theorem:overview:spherical:generic and symmetric}, as does Theorem \ref{theorem:intro:main:sphere-generic}. To establish Theorem \ref{theorem:intro:main:sphere} in case of general $\xi$, we note that by Proposition I.B.6, both $\inf_{(\zeta,\b{\Phi})\in \widehat{\mathscr{Y}}}\widehat{\cal{B}}_{-t\Delta,\xi,h}$ and $\inf_{\zeta\in \mathscr{Y}}\cal{B}$ are continuous in $\xi$ (with respect to the metric on $\xi(r)=\sum_{p\ge 0}^{\infty}\beta_p^2r^p$ given by the $L^1$-norm on $\{\beta_p^2\}_{p\ge 0}\in \R^{\infty}$). As choices of generic $\xi$ are dense under metric, we see that (\ref{eqn:theorem:A equals A hat}) holds for all mixing functions $\xi$. In view of Theorem I.1.6, this completes the proof of Theorem \ref{theorem:intro:main:sphere} in the general case.
	\end{proof}
	
	We now proceed to the proof of Theorem \ref{theorem:overview:spherical:generic and symmetric}. To explain our method, we setting of inhomogeneous mixtures $\b{\xi}$, and write
	\[\xi_x(r)=\sum_{p\ge 0}\b{\beta}_{p}(x)^2r^p,\]
	for some parameters $\b{\beta}_{p}\in \R^{\Om}$ (subject to convergence conditions on $\b{\xi}$ above). With this notation, we may consider $\b{\xi}$ as a function of the array of vectors $\vec{\b{\beta}}=(\b{\beta}_{p}(x))_{x\in \Om,p\ge 0}$. In particular if we fix $(D,\b{h})$, we may emphasis this dependence with the abuse of notation
	\[Z_{N,\rm{Sph}}(\vec{\b{\beta}}):=Z_{N,\rm{Sph}}(D,\b{\xi},\b{h}).\]
	We first collect the following simple fact.
	
	\begin{lem}
		\label{lem:identification:trivial}
		The function $N^{-1}\E \log Z_{N,\rm{Sph}}(\vec{\b{\beta}})$ is convex and differentiable in $\vec{\b{\beta}}$, with 
		\[\frac{\partial}{\partial \b{\beta}_p(x)}N^{-1}\E\log  Z_{N,\rm{Sph}}(\vec{\b{\beta}})=\b{\beta}_p(x)\left(1-\E\<(\b{u}(x),\b{u}'(x))^p_N\>_{\b{\xi}}\right),\]
		where $\<*\>_{\b{\xi}}$ is the Gibbs measure with respect to $\cal{H}_{N,D,\b{\xi},\b{h},\rm{Sph}}$.
	\end{lem}
	\begin{proof}
		Convexity in $\vec{\b{\beta}}$ follows immediately from H\"{o}lder's inequality. Moreover it is clear by Gaussian integration by parts
		\[\frac{\partial}{\partial \b{\beta}_p(x)}N^{-1}\E \log Z_{N,\rm{Sph}}(\vec{\b{\beta}})=\b{\beta}_p(x) \left(\E\<(\b{u}(x),\b{u}(x))_N^p\>_{\b{\xi}}-\E\<(\b{u}(x),\b{u}'(x))_N^p\>_{\b{\xi}}\right).\]
		Recalling that the measure is supported on $S_N^{\Om}$ produces the final claim.
	\end{proof}
	
	With this shown we define
	\[\mathscr{B}(\vec{\b{\beta}}):=\inf_{(\zeta,\b{\Phi})\in \widehat{\mathscr{Y}}}\widehat{\cal{B}}_{D,\b{\xi},\b{h}}(\zeta,\b{\Phi})=\lim_{N\to \infty}|\Om|^{-1}N^{-1}\E \log Z_{N,\rm{Sph}}(\vec{\b{\beta}})\]
	where in the final equality we are using Theorem I.1.6.
	
	As the point-wise limit of convex functions is convex, Lemma \ref{lem:identification:trivial} shows that $\mathscr{B}(\vec{\b{\beta}})$ is convex as a function of $\vec{\b{\beta}}$. Given this fact, we will be able to show the following.
	
	\begin{lem}
		\label{lem:identification:spherical parisi}
		The function $\mathscr{B}(\vec{\b{\beta}})$ is convex and differentiable as a function of $\vec{\b{\beta}}$. Moreover, if we fix $\vec{\b{\beta}}$, assume that  $(\zeta,\b{\Phi})$ is a minimizer of $\hat{\cal{B}}_{\b{\xi}}$, then we have that
		\[\frac{\partial}{\partial \b{\beta}_p(x)}\mathscr{B}(\vec{\b{\beta}})=\frac{\b{\beta}_p(x)}{|\Om|}\left(1-\int_0^1 r^{p}(\Phi_x)_*(\zeta)(dr)\right).\]
	\end{lem}
	
	\begin{proof}
		
		We have already shown convexity, so we proceed to differentiability. For simplicity, we only treat partial derivatives, with the case of a general direction being identical. In particular, let us denote, for each choice of $\b{\xi}$, let us fix any choice of minimizer of $\widehat{\cal{B}}_{D,\b{\xi},\b{h}}$ as $(\zeta_{\b{\xi}},\b{\Phi}_{\b{\xi}})$ (we have shown the existence of such minimizers below in Lemma \ref{lem:minimization:existence of minimizers and bounds}). Let us further denote by $\partial^-_x$ and $\partial^+_x$, the left and right partial derivatives with respect to a variable $x$. Fix $p\ge 0$ and $x\in \Om$. Then with $\b{\xi}^{\epsilon}$ defined by taking $\b{\beta}_p^{\epsilon}(x)=\b{\beta}_p(x)+\epsilon$ and leaving all other values unaffected, we trivially have that \[\inf_{(\zeta,\b{\Phi})\in \widehat{\mathscr{Y}}}\widehat{\cal{B}}_{D,\b{\xi}^{\epsilon},\b{h}}(\zeta_{\b{\xi}^{\epsilon}},\b{\Phi}_{\b{\xi}^{\epsilon}})\le \widehat{\cal{B}}_{D,\b{\xi}^{\epsilon},\b{h}}(\zeta_{\b{\xi}},\b{\Phi}_{\b{\xi}}),\] so we conclude that
		\[\partial_{\b{\beta}_p(x)}^+\mathscr{B}(\vec{\b{\beta}})\le 
		\lim_{\epsilon^+\to 0}\frac{1}{\epsilon}\left(\widehat{\cal{B}}_{D,\b{\xi}^{\epsilon},\b{h}}(\zeta_{\b{\xi}},\b{\Phi}_{\b{\xi}})-\widehat{\cal{B}}_{D,\b{\xi},\b{h}}(\zeta_{\b{\xi}},\b{\Phi}_{\b{\xi}})\right)= \partial_{\b{\beta}_p(x)}\widehat{\cal{B}}_{D,\b{\xi},\b{h}}(\zeta_{\b{\xi}},\b{\Phi}_{\b{\xi}}).\]
		Similarly, we see that $\partial_{\b{\beta}_p(x)}^-\mathscr{B}(\vec{\b{\beta}})\ge \partial_{\b{\beta}_p(x)}\widehat{\cal{B}}_{D,\b{\xi},\b{h}}(\zeta_{\b{\xi}},\b{\Phi}_{\b{\xi}})$. On the other hand as $\mathscr{B}(\vec{\b{\beta}})$ is convex in $\b{\beta}_p(x)$, we have that $\partial_{\b{\beta}_p(x)}^-\mathscr{B}(\vec{\b{\beta}})\le \partial_{\b{\beta}_p(x)}^+\mathscr{B}(\vec{\b{\beta}})$ which not only establishes the existence of $\partial_{\b{\beta}_p(x)}\mathscr{B}(\vec{\b{\beta}})$, but also the identity 
		\[\partial_{\b{\beta}_p(x)}\mathscr{B}(\vec{\b{\beta}})=\partial_{\b{\beta}_p(x)}\widehat{\cal{B}}_{D,\b{\xi},\b{h}}(\zeta_{\b{\xi}},\b{\Phi}_{\b{\xi}}).\label{eqn:diagonal:ignore-1924}\]
		We now compute, recalling the form of $\widehat{\cal{B}}_{D,\b{\xi},\b{h}}$ (and justifying the change of variables with Equation (I.B.15)) the formula \[\partial_{\b{\beta}_p(x)}\widehat{\cal{B}}_{D,\b{\xi},\b{h}}(\zeta,\b{\Phi})=\frac{1}{|\Om|}\int_0^1(\Phi_x)_*(\zeta)([0,q])\b{\beta}_p(x)pr^{p-1}dr=\frac{\b{\beta}_p(x)}{|\Om|}\left(1-\int_0^1 r^p(\Phi_x)_*(\zeta)(dr)\right).\]
		This achieves the desired identity when combined with (\ref{eqn:diagonal:ignore-1924}).
	\end{proof}	
	
	Recall that if $f_n\To f$ is a pointwise limit of differentiable convex functions, such that $f$ is also differentiable, then we have that $\D f_n(x)\To \D f(x)$ (see Theorem 25.7 of \cite{convexanalysis}). From this and Lemmas \ref{lem:identification:trivial} and \ref{lem:identification:spherical parisi}, we obtain the following corollary.
	
	\begin{corr}
		\label{corr:identification:simple}
		$(\zeta,\b{\Phi})$ be any minimizer of $\widehat{\cal{B}}_{D,\b{\xi},\b{h}}$. Then for any $p\ge 0$ and $x\in \Om$ we have that
		\[\b{\beta}_p(x)\int_0^1 r^{p}(\Phi_x)_*(\zeta)(dr)=\b{\beta}_p(x)\lim_{N\to \infty} \E\<(\b{u}(x),\b{u}'(x))^p_N\>_{\b{\xi}}.\]
		In particular, if $\b{\xi}$ is generic, then for any continuous function $f:[-1,1]\to \R$, and $x\in \Om$,
		\[\int_0^1 f(r)(\Phi_x)_*(\zeta)(dr)=\lim_{N\to \infty} \E\<f((\b{u}(x),\b{u}'(x))_N)\>_{\b{\xi}}\]
	\end{corr}
	
	Finally recall the following result from our companion paper.
	
	\begin{lem}[\cite{Paper1}, Proposition I.B.4]
		\label{lem:general:parisi doesn't depend away from support}
		Let us fix $(\zeta,\b{\Phi})$ and $(\zeta,\b{\Phi}')$ in $\widehat{\mathscr{Y}}$. Let us assume that for any $s\in \supp(\zeta)$, we have $\b{\Phi}(s)=\b{\Phi}'(s)$. Then 
		\[\widehat{\cal{B}}_{D,\b{\xi},\b{h}}(\zeta,\b{\Phi})=\widehat{\cal{B}}_{D,\b{\xi},\b{h}}(\zeta,\b{\Phi}').\label{eqn:diagonal:ignore-1842}\]
	\end{lem}
	
	We are now ready to give the proof of Theorem \ref{theorem:overview:spherical:generic and symmetric}
	
	\begin{proof}[Proof of Theorem \ref{theorem:overview:spherical:generic and symmetric}]
		Let $(\zeta,\b{\Phi})$ be a minimizer of $\hat{\cal{B}}_{\b{\xi}}$. By Corollary \ref{corr:identification:simple}, we observe that if we denote by $\cal{L}_N^x$, the law of $(\b{u}(x),\b{u}'(x))_N$ under the expected Gibbs measure of $\cal{H}_{N,\b{\xi}}$, then we have the convergence $\cal{L}_N^x\To (\Phi_x)_{*}(\zeta)$ in law. Note though, that by the symmetry of the underlying model, for any $x,y\in \Om$, $\cal{L}_N^x\disteq \cal{L}_N^y$ so that $(\Phi_x)_{*}(\zeta)\disteq(\Phi_y)_{*}(\zeta)$. As the functions $\Phi_x$ are monotone non-decreasing and continuous this implies that $\Phi_x(s)=\Phi_y(s)$ for each $s\in \supp(\zeta)$. In particular, at these points we have that $\Phi_x(s)=s$.
		
		From this, we see by Lemma \ref{lem:general:parisi doesn't depend away from support}, that we see we may replace each $\Phi_x$ with the identity map without changing the value of $\hat{\cal{B}}_{\b{\xi}}$. In particular, a minimizer occurs in this restricted subset, completing the proof.
	\end{proof}

	Finally, we complete the section with the proof of Theorem \ref{theorem:intro:parisi-measure:sphere}.
	
	\begin{proof}[Proof of Theorem \ref{theorem:intro:parisi-measure:sphere}]
		To show convexity, note that the only terms of $2\cal{B}$ which are not constant or linear in $\zeta$ are
		\[ h^2 K(\delta(0);t)^{-1}+\int_0^1 K(\delta(u);t)du.\]
		To show convexity of the first term, we note that
		\[\frac{d^2}{dx^2}K(x;t)^{-1}=K(x;t)^{-3}\l(2 (K'(x;t))^2-K(x;t)K''(x;t)\r).\label{eqn:ignore-2045}\]
		Twice differentiating the expression
		\[\tr((K(x;t)-t  \Delta)^{-1})=x,\]
		we obtain that 
		\[\tr((K(x;t)-t  \Delta)^{-3})2 (K'(x;t))^2=\tr((K(x;t)-t  \Delta)^{-2})K''(x;t).\]
		In particular, we may rewrite (\ref{eqn:ignore-2045}) as
		\[K(x;t)^{-3} 2 (K'(x;t))^2\l(1-K(x;t)\frac{\tr((K(x;t)-t  \Delta)^{-3})}{\tr((K(x;t)-t  \Delta)^{-2})}\r).\label{eqn:ignore-2055}\]
		Then using that for $x,a\ge 0$, we have that $(a+x)^{-2}\ge x(a+x)^{-3}$, we see that
		\[\tr((K(x;t)-t  \Delta)^{-2})\ge \tr((K(x;t)-t  \Delta)^{-3})K(x;t).\]
		Together with (\ref{eqn:ignore-2055}), this implies the convexity of $K(x;t)^{-1}$. As the expression $\delta(u)$ is linear in $\zeta$, this implies convexity of the term $h^2 K(\delta(0);t)^{-1}$.
		
		Now finally to show strict convexity, we only need to show that $\int_0^1 K(\delta(u);t)du$ is strictly convex. For this fix $\mu\in \mathscr{Y}\setminus{\zeta}$, and for $0\le \epsilon\le 1$, consider the convex combination $\zeta_{\epsilon}:=(1-\epsilon)\zeta+\epsilon \mu$. It suffices to show that
		\[\frac{d^2}{d\epsilon^2}\l(\int_0^1 K(\delta(u);t)du\r)\bigg|_{\epsilon^+=0}>0\label{eqn:ignore-2108}\]
		Finally, let us denote for simplicity $\nu=\mu-\zeta$. Then we see that
		\[\frac{d^2}{d\epsilon^2}\l(\int_0^1 K(\delta(u);t)du\r)\bigg|_{\epsilon^+=0}=\int_0^1 K''(\delta(u);t)\l(\int_u^1 \nu([0,s])ds\r)^2.\]
		As we know that $K$ is strictly convex, we see that for (\ref{eqn:ignore-2108}) to fail, we must have that for each $u\in [0,1]$, that $\int_u^1 \nu([0,s])ds=0$. On the other hand, this implies that $\mu([0,s])=\zeta([0,s])$ for each $s\in [0,1]$, which implies that $\mu=\zeta$, which is a contradiction. In particular, this establishes (\ref{eqn:ignore-2108}), which completes the proof of strict convexity of $\cal{B}$.
		
		Now with notation as above, we see that
		\[
		\begin{split}
			\frac{d}{d\epsilon}\cal{B}(\zeta_{\epsilon})|_{\epsilon^+=0}=&- h^2 K(\delta(0);t)^{-2}K'(\delta(0);t)\int_0^1 \nu([0,s])ds\\
			&+\int_0^1 K'(\delta(s);t)\int_s^1\nu([0,u])du+\int_0^1\nu([0,s])\xi'(u)du=\\
			&\int_0^1 \nu([0,u])F_{\xi}(u)du=-\int_0^1 f_{\xi}(u)\nu(du).
		\end{split}
		\]
		In particular, we see that $\zeta$ is a local minimizer of $\cal{B}$ if and only if for all $\mu\in \mathscr{Y}$
		\[\int_0^1 f_{\xi}(u)\mu(du)\le \int_0^1 f_{\xi}(u)\zeta(du).\]
		This is clearly equivalent to the condition of (\ref{eqn:intro:minimization sphere}).
	\end{proof}

	\section{Minimization Identities for the Parisi Functional \label{section:minimization}}
	
	In this section, we will study the derivative of the functional $\widehat{\cal{B}}_{D,\b{\xi},\b{h}}$, recalled in the previous section, with respect to a certain family of perturbations. These perturbations essentially amount to a form of measure re-parametrization. While this family typically does not form a complete basis of tangent directions in the domain, they will give a number of important facts.
	
	The first is the existence of global minimizers for $\widehat{\cal{B}}_{D,\b{\xi},\b{h}}$, an important technical fact used above. The second is that if $\xi'_x(0)\neq 0$ for all $x\in \Om$, then any minimizer, say $(\zeta, \b{\Phi})$, is such that $\supp(\b{\Phi}_*(\zeta))\cap \partial [0,1]^{\Om}=\varnothing$. Roughly speaking, this shows that the (possibly plural) Parisi measures do not have support at either $0$ or $1$, and this result will be used crucially in our proof of concavity in Section \ref{section:concavity}. Finally, it gives explicit formulas that any minimizer must satisfy, which will be important in showing concentration of the radius of the model.
	
	We note the methodology of this section, especially Lemma \ref{lem:minimization:derivative in epsilon}, owes to the work of \cite{erikCS}, who compute the minimizers of similar functionals in the case of certain multi-species spherical spin glass models.
	
	To begin, let us fix $(\zeta,\b{\Phi})\in \widehat{\mathscr{Y}}$. Then, choosing some additional coordinate-wise non-decreasing absolutely continuous function $\b{\chi}:[0,1]\to [0,1]^{\Om}$, and we define $\b{\Psi}=\b{\chi}-\b{\Phi}$. We consider the family \[\b{\tilde{\Phi}}^{\epsilon}=\b{\Phi}+\epsilon\b{\Psi}=(1-\epsilon)\b{\Phi}+\epsilon \b{\chi}.\]
	We wish to compute the derivative of $\widehat{\cal{B}}_{D,\b{\xi},\b{h}}$ in the family $\b{\tilde{\Phi}}^{\epsilon}$, but as it may not live in $\widehat{\mathscr{Y}}$ due to violating the condition (\ref{eq:condition for Psi-spherical}), we need to introduce a slight reparametrization. Namely, define \[\alpha_{\epsilon}(q)=\frac{1}{|\Om|}\sum_{x\in \Om}\tilde{\Phi}^{\epsilon}_x(q)=q(1-\epsilon)+\epsilon\sum_{x\in \Om}\chi_x(q),\] 
	which is a continuous, strictly increasing function $\alpha_{\epsilon}:[0,1]\to [0,1]$. In particular, this shows that $\alpha_{\epsilon}$ is a homeomorphism from $[0,1]$ onto $[\alpha_{\epsilon}(0),\alpha_{\epsilon}(1)]$, so we may define for $q\in [\alpha_{\epsilon}(0),\alpha_{\epsilon}(1)]$,
	\[\b{\Phi}^{\epsilon}=\b{\tilde{\Phi}}^{\epsilon}\circ \alpha_{\epsilon}^{-1},\]
	and otherwise define it by the linear interpolations
	\[\b{\Phi}^{\epsilon}(q)=\frac{q}{\alpha_{\epsilon}(0)}\tilde{\b{\Phi}}^{\epsilon}(0),\; q\in [0,\alpha_{\epsilon}(0)),\;\; \b{\Phi}^{\epsilon}(q)=\frac{q-\alpha_{\epsilon}(1)+\tilde{\b{\Phi}}^{\epsilon}(1)(1-q)}{1-\alpha_{\epsilon}(1)},\; q\in (\alpha_{\epsilon}(1),1].\]
	We define $\zeta^{\epsilon}$ as the push-forward of $\zeta$ under $\alpha_{\epsilon}$, so that
	\[\zeta^{\epsilon}([0,u])=\begin{cases}
		1,\;\;\;\;\;\;\;\;\;\;\;\;\;\;\;\;\ u>\alpha_{\epsilon}(1)\\
		\zeta([0,\alpha_{\epsilon}^{-1}(u)]),\;\;\; u\in [\alpha_{\epsilon}(0),\alpha_{\epsilon}(1)]\\
		0,\;\;\;\;\;\;\;\;\;\;\;\;\;\;\;\;\;\;\; u<\alpha_{\epsilon}(0)
	\end{cases}.\]
	It is then easily checked that $(\zeta^{\epsilon},\b{\Phi}^{\epsilon})\in \widehat{\mathscr{Y}}$.
	
	We will now prepare to compute the derivative of $\widehat{\cal{B}}_{D,\b{\xi},\b{h}}$ in this family at $\epsilon=0$. There is a slight technicality, in that this only defines a family $\widehat{\mathscr{Y}}$ for positive $\epsilon>0$, so we may only consider the derivative in the positive direction. On the other hand if both $\Psi_x(0)=0$ and $\Psi_x(1)=1$ for all $x\in \Om$, one may actually check the above construction defines a family for $\epsilon\in (-\delta,\delta)$ for sufficiently small $\delta>0$. Thus in this case the derivative itself may be considered. For brevity, we will omit this from the notation, as it will not affect the value of the derivative, though we will return to this technicality in the proof of Lemma \ref{lem:minimization:eqns} below, as local minimization only implies that the derivative is non-negative in the prior case, while it implies actual vanishing in the latter.
	
	The computation of this derivative will be given by the following lemma, for which we will introduce some notation to make its statement simpler. Let consider the function $\b{\delta}:[0,1]\to [0,\infty)^{\Om}$ associated to $(\zeta,\b{\Phi})$ by (\ref{eqn:def:delta-A}). We first introduce the following constants, indexed by $y\in \Om$
	\[\cal{J}_y=\sum_{x,z\in \Om}[(D+\b{K}^D(\b{\delta}(0)))^{-1}]_{xy}[(D+\b{K}^D(\b{\delta}(0)))^{-1}]_{zy}\b{h}(x)\b{h}(z).\]
	Next, we introduce three families of functions, again indexed by $y\in \Om$
	\[\cal{G}_y(q)=\xi'_y(\Phi_y(q))+\sum_{x\in \Om}\left(\int_0^q \D_y K^D_x(\b{\delta}(u))\Phi'_x(u)du\right)+\cal{J}_y,\label{eqn:minimization:definition of G}\]
	\[\cal{H}_y(q)=K^D_y(\b{\delta}(q))+\int_q^1\zeta([0,u])\xi''_y(\Phi_y(u))\Phi_y'(u)du,\]
	\[\cal{F}_y(q)=\cal{H}_y(q)+\zeta([0,q])\cal{G}_y(q).\]
	
	\begin{lem}
		\label{lem:minimization:derivative in epsilon}
		With notation as above, and for any $q_*$ such that $\zeta([0,q_*))=1$ and $\Phi_x(q_*)<1$ for all $x\in \Om$, we have that
		\[\frac{d}{d\epsilon}\widehat{\cal{B}}_{D,\b{\xi},\b{h}}(\zeta^{\epsilon},\b{\Phi}^{\epsilon})=\frac{1}{2|\Om|}\sum_{y\in \Om}\left(\int_0^{q_*}\cal{F}_y(q)\Psi_y'(q)dq-\cal{F}_y(q_*)\Psi_y(q_*)+\cal{H}_y(0)\Psi_y(0)\right).\]
	\end{lem}
	\begin{proof}
		We restrict to $\epsilon$ is small enough that we still have for each $x\in \Om$ \[\Phi_x^{\epsilon}(\alpha_{\epsilon}(q_*))=\tilde{\Phi}_x^{\epsilon}(q_*)<1.\label{eqn:ignore-2002}\]
		We also denote by $\b{\delta}^{\epsilon}$ as the $\b{\delta}$-function defined with respect to $(\zeta^{\epsilon},\b{\Phi}^{\epsilon})$. We observe that		\[\delta_x^{\epsilon}(\alpha_\epsilon(1))=\Phi_x^{\epsilon}(1)-\Phi_x^{\epsilon}(\alpha_{\epsilon}(1))=1-\tilde{\Phi}^{\epsilon}_x(1)=-\epsilon \Psi_x(1),\]
		and that
		\[\delta_x^{\epsilon}(\alpha_\epsilon(q))=\delta_x(\alpha_\epsilon(1))+\int_{\alpha_{\epsilon}(q)}^{\alpha_{\epsilon}(1)}\zeta([0,q])(\Phi^{\epsilon}_x)'(q)dq=-\epsilon \Psi_x(1)+\int_q^1\zeta([0,q])(\tilde{\Phi}^{\epsilon}_x)'(q)dq.\]
		Finally note that for $q\in [0,\alpha_\epsilon(0)]$ we have that $\delta_x^{\epsilon}(q)=\delta_x^{\epsilon}(\alpha_\epsilon(0))$.
		
		Now we note that by using $\zeta^{\epsilon}([0,\alpha_\epsilon(q_*)))=1$ and (\ref{eqn:ignore-2002}), we may evaluate the functional using the $\epsilon$-dependent endpoint $\alpha_\epsilon(q_*)$. To do this, we note that
		\[\int_{0}^{\alpha_\epsilon(q_*)} K^D_x(\b{\delta}^{\epsilon}(q))(\Phi_x^{\epsilon})'(q)dq=K^D_x(\b{\delta}^{\epsilon}(0))\Phi_x^{\epsilon}(0)+\int_0^{q_*} K^D_x(\b{\delta}^{\epsilon}(\alpha_{\epsilon}(q)))(\tilde{\Phi}_x^{\epsilon})'(q)dq.\]
		Using these results we may compute that
		\[\begin{split}
			\frac{d}{d\epsilon}\int_{0}begin^{\alpha_\epsilon(q_*)} K^D_x(\b{\delta}^{\epsilon}(q))(\Phi_x^{\epsilon})'(q)dq=\Psi_x(0)K^D_x(\b{\delta}(0))+\int_0^{q_*} K^D_x(\b{\delta}(q))\Psi_x'(q)dq\\
			+\sum_{y\in \Om}\int_0^{q_*}\D_y K^D_x(\b{\delta}(q))\bigg(-\Psi_y(1)+\int_q^1\zeta([0,u])\Psi_y'(u)du \bigg)\Phi'_x(q)dq.
		\end{split}\]
		To contend with the final term we note that
		\[\begin{split}
			\int_0^{q_*}\D_y K^D_x(\b{\delta}(q))\bigg(\int_q^1\zeta([0,u])\Psi_y'(u)du \bigg)\Phi'_x(q)dq=\\
			\int_0^{q_*}\zeta([0,u])\left(\int_0^u \D_y K^D_x(\b{\delta}(q))\Phi'_x(q)dq\right)\Psi_y'(u)du+(\Psi_y(1)-\Psi_y(q_*))\int_0^{q_*}\D_y K^D_x(\b{\delta}(q))\Phi'_x(q)dq.
		\end{split}\]
		Next we note that 
		\[\delta^{\epsilon}_x(\alpha_{\epsilon}(q_*))=\Phi^{\epsilon}_x(1)-\Phi^{\epsilon}_x(\alpha_\epsilon(q_*))=1-\epsilon\Psi_x(q_*)\]
		so we may conclude that
		\[\frac{d}{d\epsilon}|\Om|\Lambda^D(\b{\delta}^{\epsilon}(\alpha_{\epsilon}(q_*)))=-\sum_{x\in \Om}K^D_x(\b{\delta}(q_*))\Psi_x(q_*).\]
		In conclusion, and noting the terms involving $\Psi_y(1)$ cancel, we see that
		\[\begin{split}
			\frac{d}{d\epsilon}\left(|\Om|\Lambda^D(\b{\delta}^{\epsilon}(\alpha_{\epsilon}(q_*)))+\sum_{x\in \Om}\int_{0}^{\alpha_\epsilon(q_*)} K^D_x(\b{\delta}^{\epsilon}(q))(\Phi_x^{\epsilon})'(q)dq\right)=\\
			\sum_{y\in \Om}\Psi_y(0)K^D_y(\b{\delta}(0))-\sum_{y\in \Om}\Psi_y(q_*)\left(\sum_{x\in \Om}\int_0^{q_*}\D_y K^D_x(\b{\delta}(q))\Phi'_x(q)dq+K^D_y(\b{\delta}(q_*))\right)\\
			+\sum_{y\in \Om}\int_0^{q_*}\left(\sum_{x\in \Om}\zeta([0,u])\left(\int_0^u \D_y K^D_x(\b{\delta}(q))\Phi'_x(q)dq\right)+K^D_y(\b{\delta}(u))\right)\Psi_y'(u)du.
		\end{split}\]
		Next, noting that $\zeta^{\epsilon}([0,\alpha_{\epsilon}(0)))=0$, we see that
		\[\int_0^1\zeta^{\epsilon}([0,q])\xi_x'(\Phi_x^{\epsilon}(q))(\Phi_x^{\epsilon})'(q)dq=\int_0^1\zeta([0,q])\xi_x'(\tilde{\Phi}_x^{\epsilon}(q))(\tilde{\Phi}_x^{\epsilon})'(q)dq+\xi_x(1)-\xi_x(\tilde{\Phi}^\epsilon_x(1)),\]
		so we compute as above that
		\[\begin{split}
			\frac{d}{d\epsilon}\int_0^1\zeta^{\epsilon}([0,q])\xi_x'(\Phi_x^{\epsilon}(q))(\Phi_x^{\epsilon})'(q)dq=\\
			\int_0^1\zeta([0,q])\xi'_x(\Phi_x(q))\Psi_x'(q)dq
			+\int_0^1\zeta([0,q])\xi''_x(\Phi_x(q))\Psi_x(q)\Phi'_x(q)dq
			-\xi_x'(1)\Psi_x(1).
		\end{split}\]
		Using integration by parts, we rewrite the second term as
		\[
		\int_0^1\left(\int_q^1\zeta([0,q])\xi''_x(\Phi_x(q))\Phi'_x(q)dq\right)\Psi_x'(u)du+\left(\int_0^1\zeta([0,q])\xi''_x(\Phi_x(q))\Phi'_x(q)dq\right)\Psi_x(0).
		\]
		We note that as $\zeta([0,q])=1$ for $q\in [q_*,1]$, we have that
		\[\int_{q_*}^1\left(\zeta([0,q])\xi'_x(\Phi_x(q))+\int_q^1\zeta([0,u])\xi''_x(\Phi_x(u))\Phi'_x(u)du\right)\Psi_x'(q)dq=\xi_x'(1)(\Psi_x(1)-\Psi_x(q_*)).\]
		Thus combined we see that
		\[\begin{split}
			\frac{d}{d\epsilon}\int_0^1\zeta^{\epsilon}([0,q])\xi_x'(\Phi_x^{\epsilon}(q))(\Phi_x^{\epsilon})'(q)dq=\left(\int_0^1\zeta([0,q])\xi''_x(\Phi_x(q))\Phi'_x(q)dq\right)\Psi_x(0)\\
			+\int_0^{q_*}\left(\zeta([0,q])\xi'_x(\Phi_x(q))+\int_q^1\zeta([0,u])\xi''_x(\Phi_x(u))\Phi'_x(u)du\right)\Psi_x'(q)dq-\xi'(1)\Psi_x(q_*).
		\end{split}\]
		Now finally we contend with the external field term to get
		\[\begin{split}
			\frac{d}{d\epsilon}[(D+\b{K}^D(\delta^{\epsilon}(0)))^{-1}]_{xy}=\\
			-\sum_{z\in \Om}[(D+\b{K}^D(\delta(0)))^{-1}]_{xz}[(D+\b{K}^D(\delta(0)))^{-1}]_{zy}\left(\int_0^1\zeta([0,q])d\Psi_z(q)-\Psi_z(1)\right).
		\end{split}\]
		Now we note that for $q\ge q_*$, we have that
		\[\cal{F}_y(q)=K^D_y(\b{\delta}(q))+\xi'_y(1)+\sum_{x\in \Om}\left(\int_0^{q}\D_y K^D_x(\b{\delta}(q))\Phi'_x(q)dq\right)+\cal{J}_y.\]
		Altogether, we obtain the desired result.
	\end{proof}
	
	To further make use of this computation, we will want the following result.
	
	\begin{lem}
		\label{lem:minimizers:hard lemma}
		Let us fix arbitrary $(\zeta,\b{\Phi})$ as above, and assume $q_*<1$ is such that $\zeta([0,q_*))=1$ and $\Phi_x(q_*)<1$ for all $x\in \Om$. Then for any smooth compactly supported function $f:(0,q_*)\to \R$ and $y\in \Om$, we have that
		\[\int_0^1 \cal{F}_y(q)f(q)dq=-\int_0^1 f'(q)\cal{G}_y(q)\zeta(dq).\]
	\end{lem}
	\begin{proof}
		By integration by parts, we have that
		\[\int_0^1 f'(q)K^D_y(\b{\delta}(q))dq=\int_0^1 f(q)\zeta([0,q])\sum_{x\in \Om}\left(\int_0^q \D_y K^D_x(\b{\delta}(u))\Phi'_x(u)du\right),\]
		and that
		\[
		\begin{split}
			\int_0^1& f'(q)\zeta([0,q])\sum_{x\in \Om}\left(\int_0^q \D_y K^D_x(\b{\delta}(u))\Phi'_x(u)du\right)dq\\
			=&-\int_0^1 f(q)\zeta([0,q])\sum_{x\in \Om}\left(\int_0^q \D_y K^D_x(\b{\delta}(u))\Phi'_x(u)du\right)dq\\
			&-\int_0^1 f(q)\sum_{x\in \Om}\left(\int_0^q \D_y K^D_x(\b{\delta}(u))\Phi'_x(u)du\right)\zeta(dq).
		\end{split}
		\]
		Canceling a common term, we see that 
		\[
		\begin{split}
			&\int_0^1 f'(q)\l(K^D_y(\b{\delta}(q))+\zeta([0,q])\sum_{x\in \Om}\left(\int_0^q \D_y K^D_x(\b{\delta}(u))\Phi'_x(u)du\right)\r)dq\\
			&=-\int_0^1 f(q)\sum_{x\in \Om}\left(\int_0^q \D_y K^D_x(\b{\delta}(u))\Phi'_x(u)du\right)\zeta(dq).
		\end{split}
		\]
		Similarly, we have that
		\[
		\begin{split}
			&\int_0^1 f'(q)\l(\xi'_y(\Phi_y(q))+\zeta([0,q])\l(\int_q^1\zeta([0,u])\xi''_y(\Phi_y(u))\Phi_y'(u)du\r)\r)dq\\
			&=-\int_0^1 f(q)\l(\int_q^1\zeta([0,u])\xi''_y(\Phi_y(u))\Phi_y'(u)du\r)\zeta(dq).
		\end{split}
		\]
		Combining these with the trivial observation
		\[\cal{J}_y\int_0^1 f'(q)\zeta([0,q])dq=-\cal{J}_y\int_0^1 f(q)\zeta(dq),\]
		completes the proof.
	\end{proof}
	
	Combining these lemmas, we will obtain a fairly tractable set of equations for any minimizers. However, as we have not even yet shown that any minimizer exists, we first consider the functional restricted to certain compact subsets of $\widehat{\mathscr{Y}}$. In particular, we define for $\epsilon\in (0,1)$
	\[\widehat{\mathscr{Y}}_\epsilon=\{(\zeta,\b{\Phi})\in \widehat{\mathscr{Y}}: \Phi_x(\sup \supp(\zeta))\le 1-\epsilon \text{ for all } x\in \Om\}.\]
	We note that this set is non-zero and compact. In particular, $\widehat{\cal{B}}_{D,\b{\xi},\b{h}}$ must possess a global minimizer over this set. The following result gives some useful identities for any such minimizer.
	\begin{lem}
		\label{lem:minimization:eqns}
		Let $(\zeta,\b{\Phi})$ be any local minimizer of $\widehat{\cal{B}}_{D,\b{\xi},\b{h}}$ over $\widehat{\mathscr{Y}}_\epsilon$. Then for $y\in \Om$, and any $q\in \supp(\zeta)\cap (0,q_*)$, we have $\cal{G}_y(q)=0$. Finally, if $0\in \supp(\zeta)$, then $\cal{G}_y(0)\le 0$ and if $q_*\in \supp(\zeta)$, then $\cal{G}_y(q_*)\ge 0$.
	\end{lem}
	\begin{proof}
		Let $f:[0,1]\to [0,1]$ be a strictly increasing function such that $f(0)=0$ and $f(q)=\Phi_y(q)$ for $q\in [q_*,1]$. Consider the Lemma \ref{lem:minimization:derivative in epsilon} with respect to $\chi_y(q)=f(q)$ and $\chi_x(q)=0$ for $x\in \Om\setminus \{y\}$. In this case, $\b{\Phi}^{\epsilon}$ lies in $\widehat{\mathscr{Y}}_\epsilon$ for $\epsilon\in (-\delta,\delta)$ with $\delta>0$ small, so  Lemma \ref{lem:minimization:derivative in epsilon} yields that
		\[\int_0^{q_*}\cal{F}_y(q)df(q)-\int_0^{q_*}\cal{F}_y(q)\Phi_y'(q)dq=0.\]
		As this holds for all choices of $f$, this implies that $\cal{F}_y$ is constant a.e. on $[0,q_*]$.
		
		Next, let $f$ be a smooth function on $[0,q_*]$ with $f(0)=f(q_*)=0$. By continuity, we see that Lemma \ref{lem:minimizers:hard lemma} yields that
		\[\int_0^{q_*} \cal{G}_y(q)f'(q)\zeta(dq)=-\int_0^{q_*}\cal{F}_y(q)f(q)dq=0,\]
		where in the last equality we have used that $\cal{F}_y(q)$ is constant a.e.
		In particular, we see that $\zeta(\{q\in (0,q_*):\cal{G}_y(q)\neq 0\})=0$. As $\cal{G}_y$ is continuous though, this implies that it vanishes on $\supp(\zeta)\cap (0,q_*)$. 
		
		Note as well that $\cal{F}_y$ is clearly continuous away from the atoms of $\zeta$. As $\cal{G}_y$ vanishes on $\supp(\zeta)\cap (0,q_*)$ though, this implies that $\cal{F}_y$ is continuous and thus constant on $(0,q_*)$.
		
		We now turn our attention to the final claims. We observe that taking $f:[0,1]\to [0,1]$ to be a strictly increasing function such that $f(0)\ge 0$ and $f(q)\le \Phi_x(q)$ for $q\in [q_*,1]$, we may study the derivative as above. However, in this case, the family is only in the domain for $\epsilon\ge 0$, so we only obtain from Lemma \ref{lem:minimization:derivative in epsilon} that
		\[\int_0^{q_*}\cal{F}_y(q)f(q)dq-\cal{F}_y(q_*)f(q_*)+\cal{H}_y(0)f(0)\ge \int_0^{q_*}\cal{F}_y(q)\Phi_y'(q)dq-\cal{F}_y(q_*)\Phi_y(q_*).\label{eqn:minimization:ignore-1154}\]
		If we let $\cal{F}_{y,t}$ denote the value taken by $\cal{F}_{y}$ on $(0,q_*)$, we may write (\ref{eqn:minimization:ignore-1154}) as
		\[(\cal{F}_{y,t}-\cal{F}_y(q_*))(f(q_*)-\Phi_y(q_*))+(\cal{H}_y(0)-\cal{F}_{y,t})f(0)\ge 0.\]
		We see from this that
		\[\cal{F}_{y,t}\le \cal{F}_y(q_*) \text{ and } \cal{F}_{y,t}\le \cal{H}_y(0).\]
		We note however that
		\[\cal{F}_{y,t}=\lim_{q\to 0}\cal{F}_y(q)=\cal{H}_y(0)+\zeta(\{0\})\cal{G}_y(0),\]
		\[\cal{F}_{y,t}=\lim_{q\to q_*}\cal{F}_y(q)=\cal{F}_y(q_*)-\zeta(\{q_*\})\cal{G}_y(q_*).\]
		Thus if $\zeta(\{0\})\neq 0$, then $\cal{G}_y(0)\le 0$ as desired. However, if $\zeta(\{0\})=0$ and $0\in \supp(\zeta)$, we see that $0$ is in the closure of $\supp(\zeta)\cap (0,q_*)$, and so $\cal{G}_y(0)=0$ by continuity. In either case, if $0\in \supp(\zeta)$, we see that $\cal{G}_y(0)\le 0$. A similar argument establishes the claim for $\cal{G}_y(q_*)$.
	\end{proof}
	
	We may use this result to guarantee the existence of a global minimizer for $\widehat{\cal{B}}_{D,\b{\xi},\b{h}}$.
	
	\begin{lem}
		\label{lem:minimization:existence of minimizers and bounds}
		The functional $\widehat{\cal{B}}_{D,\b{\xi},\b{h}}$ has at-least one minimizer over $\widehat{\mathscr{Y}}$. Moreover, there exists some small $c>0$, such that any such minimizer, say $(\zeta,\b{\Phi})$, satisfies
		\[\supp(\b{\Phi}_*(\zeta))\in [0,1-c]^{\Om}.\]
		Additionally, if $\xi'_x(0)\neq 0$ for each $x\in \Om$, then for possibly smaller $c>0$, any minimizer $(\zeta,\b{\Phi})$ satisfies
		\[\supp(\b{\Phi}_*(\zeta))\in [c,1-c]^{\Om}.\]
	\end{lem}
	\begin{proof}
		First two claims follow immediately by combining Lemma \ref{lem:minimization:eqns} and Lemma I.3.11.
		
		For the remaining claims, fix a minimizing pair $(\zeta,\b{\Phi})$. We define $q_m=\inf(\supp(\zeta))$, and observe that $\b{\delta}(q_m)=\b{\delta}(0)$. So we have that 
		\[\cal{G}_y(q_m)=\xi'_y(\Phi_y(q))+\sum_{x\in \Om}\D_x K^D_y(\b{\delta}(0))\Phi_x(q)+\cal{J}_y.\]
		If $q_m=0$, our minimization conditions imply that $0\le \cal{G}_y(0)=\xi'_y(0)+\cal{J}_y$. However, by our assumptions, we have that $\xi'_y(0)+\cal{J}_y>0$, which is a contradiction. Thus we must have $q_m>0$, and so $\cal{G}_y(q_m)=0$ for all $y\in \Om$. Using the formula for $\D_x K_y^D$, we may rewrite this system of equations as
		\[\Phi_x(q_m)=\sum_{y\in \Om}[(D+\b{K}^D(\b{\delta}(0)))^{-1}]_{xy}^2(\xi_y'(\Phi_y(q_m))+\cal{J}_y),\]
		We note that the sum on the right is bounded below by
		\[[(D+\b{K}^D(\b{\delta}(0)))^{-1}]_{xx}^2(\xi_x'(\Phi_x(q_m))+\cal{J}_x)=\delta_x(0)^2(\xi_x'(\Phi_x(q_m))+\cal{J}_x)\ge \delta_x(0)^2\xi'_x(0).\]
		By the above work we have that $\delta_x(0)\ge 1-\Phi_x(q_M)\ge c$, so we see that for each $x\in \Om$
		\[\Phi_x(q_m)\ge \xi'_x(0)c.\]
		So by shrinking $c$ by an amount only dependent on $\min_{x\in \Om}\xi_x'(0)>0$, we obtain the second claim.
	\end{proof}

	\section{The Euclidean Model and Concavity \label{section:concavity}}
	
	We now focus our attention on the Euclidean model. The goal of this section will be to show a certain function is concave with a unique minimizer, but to begin we will need to recall some more relevant results from our companion paper \cite{Paper1}. We also return from the above case of a general finite set $\Om$ to specifically the case $\Om=[[1,L]]^{d}$ and $D=-t\Delta$. However, we note that all of our results hold in the setting of the previous sections.
	
	To start, we will need a functional. Begin by fixing a choice $\b{q}\in (0,\infty)^{\Om}$ and define $q_{t}=\frac{1}{|\Om|}\sum_{x\in \Om}\b{q}(x)$. Next we take a measure $\zeta\in \mathscr{Y}(q_t)$. Then, for each each $x\in \Om$, we choose a continuous, coordinate-wise non-decreasing function $\Phi_x:[0,q_t]\to [0,\b{q}(x)]$, which much be such that for each $s\in [0,q_t]$ we have that
	\[\frac{1}{|\Om|}\sum_{x\in \Om}\b{q}(x)^{-1}\Phi_x(s)=q_t^{-1}s.\label{eq:condition for Psi}\]
	We will assume as well that there is $0<q_*<q_t$ is such that $\zeta([0,q_*))=0$ and $\Phi_x(q_*)<\b{q}(x)$ for each $x\in \Om$. We will denote the set of pairs $(\zeta,\b{\Phi})$ satisfying these conditions as $\widehat{\mathscr{Y}}(\b{q})$. We observe that when $\b{q}(x)=1$ for all $x\in \Om$, this coincides with the $\widehat{\mathscr{Y}}(\b{q})$ considered above.
	
	Given these choices we define the continuous coordinate-wise non-increasing function $\b{\delta}:[0,q_t]\to [0,\infty)^{\Om}$ by
	\[\delta_x(s)=\int_{s}^{q_t}\zeta([0,u])\Phi'_x(u)du.\label{eqn:section2:delta-q}\]
	With all these choices made, we define the functional
	\[
	\begin{split}
		\widehat{\cal{P}}_{\b{q}}(\zeta,\b{\Phi})=\frac{1}{2}\bigg(\log\left(2\pi\right)+\Lambda^{-t\Delta}&(\b{\delta}(q_*))+\frac{1}{L^d}\sum_{x\in \Om}\bigg(-\mu \b{q}(x)+\int_{0}^{q_*} K_x^{-t\Delta}(\b{\delta}(q))\Phi'_x(q)dq\\&-2\int_0^{q_t}\zeta([0,u])B'(2(\b{q}(x)-\Phi_x(u)))\Phi'_x(u)du\bigg)\bigg).
	\end{split}
	\]
	
	We showed in \cite{Paper1} that this paper may be used to calculate the limiting free energy.
	
	\begin{theorem}[Theorem I.1.3]
		\label{theorem:paper 1:euclidean result}
		We have that
		\[\lim_{N\to \infty}N^{-1}L^{-d}\E \log Z_N=\sup_{\b{q}\in (0,\infty)^{\Om}}\left(\inf_{(\zeta,\b{\Phi})\in \widehat{\mathscr{Y}}(\b{q})}\widehat{\cal{P}}_{\b{q}}(\zeta,\b{\Phi})\right).\]
	\end{theorem}
	
	The main result of this section will be to show that map
	\[\b{q}\mapsto \inf_{(\zeta,\b{\Phi})\in \widehat{\mathscr{Y}}(\b{q})}\widehat{\cal{P}}_{\b{q}}(\zeta,\b{\Phi})\]
	is concave, and has a unique maximizer. By symmetry, this maximizer must be a constant vector. However, first we need more preliminary results which will make out work on the spherical model useful here. As may be clear from the similarities between their forms, this functional can be quite simply related to a spherical case of the spherical one studied above. 
	
	Indeed, for $(\zeta^{\b{q}},\b{\Phi}^{\b{q}})$ in $\widehat{\mathscr{Y}}(\b{q})$, we may define $(\zeta,\b{\Phi})\in \widehat{\mathscr{Y}}$ by $\Phi_x(s)=\b{q}(x)\Phi_x^{\b{q}}(q_t^{-1}s)$ and $\zeta^{\b{q}}([0,s])=\zeta([0,q_t^{-1}s])$. This mapping gives a bijection between $\widehat{\mathscr{Y}}(\b{q})$ and $\widehat{\mathscr{Y}}$.
	
	Next, define a matrix $D_{\b{q}}\in \R^{\Om\times \Om}$ such that $[D_{\b{q}}]_{xy}=\sqrt{\b{q}(x)\b{q}(y)}t\Delta_{xy}$ and let $\b{h}=0$. Further, let us define the mixing function $B_q(r)=B(2q(1-r))$, and form the family $\b{B}_{\b{q}}=(B_{\b{q}(x)})_{x\in \Om}$.
	Then, following Lemma I.2.3, we have the relation
	\[\widehat{\cal{P}}_{\b{q}}(\zeta,\b{\Phi})=\frac{h^2}{2\mu^2}+\frac{1}{2L^d}\sum_{x\in \Om}\left(-\mu\b{q}(x)+\log(\b{q}(x))\right)+\widehat{\cal{B}}_{D_{\b{q}},\b{B}_{\b{q}},\b{0}}(\zeta^{\b{q}},\b{\Phi}^{\b{q}}).\label{eqn:relation between A and P}\]
	This correspondence will allow us to translate results from $\widehat{\cal{B}}$ to $\widehat{\cal{P}}_{\b{q}}$.
	
	We note that the external field term does not appear in the functional, only as a constant term. That is, we may effectively view the associated spherical model as if it doesn't have an external field. The reason for this is due to a change of variables which one may make in the case of the Euclidean model, which renders the external field fairly inert. In particular, we have the following remark, which follows from Lemma I.2.1.
	
	\begin{remark}
		\label{remark:h doesn't matter}
		If we include for the moment the choice of $h$, in the notation $Z_{N}(h)$, for the partition function, we have that
		\[N^{-1}|\Om|^{-1}\log Z_{N}(h)\disteq N^{-1}|\Om|^{-1} \log Z_{N}(0)+\frac{h^2}{2 \mu}.\]
	\end{remark}
	
	We now begin translating the above results for spherical models to this functional. First, by noting that $B_{q}'(0)=-2qB'(2q)>0$ we see that Lemma \ref{lem:minimization:existence of minimizers and bounds} yields the following.
	
	\begin{corr}
		\label{corr:concavity:avoids 0}
		For each $\b{q}\in (0,\infty)^{\Om}$, the functional $\widehat{\cal{P}}_{\b{q}}$ possesses a minimizer on $\widehat{\mathscr{Y}}(\b{q})$.\\
		Moreover, for any $\b{q}_0\in (0,\infty)^{\Om}$, one may find a small $\epsilon,\delta>0$ so that if one takes $\b{q}\in \prod_{x\in \Om}[\b{q}_0(x)-\delta,\b{q}_0(x)+\delta]$, and any minimizer $(\zeta,\b{\Phi})$ of $\widehat{\cal{P}}_{\b{q}}$ over $\widehat{\mathscr{Y}}(\b{q})$, then $\supp(\b{\Phi}_*(\zeta))\subseteq [\epsilon,1-\epsilon]^{\Om}$.
	\end{corr}
	
	For the next corollary, we need to introduce following variant of the family of functions $\b{\cal{G}}$ defined above: For $y\in \Om$ and $s\in [0,q_t]$, we define
	\[\cal{G}_y^{\b{q}}(s)=-2B'(2(\b{q}(y)-\Phi_y(s)))+\sum_{x\in \Om}\left(\int_0^s \D_y K^{-t\Delta}_x(\b{\delta}(u))\Phi'_x(u)du\right).\label{eqn:minimization:definition of G for P}\]
	Then we obtain the following from Lemma \ref{lem:minimization:eqns} and Corollary \ref{corr:concavity:avoids 0}
	
	\begin{corr}
		\label{corr:concavity:minimization in G for P}
		For each $\b{q}\in (0,\infty)^{\Om}$, and any minimizer $(\zeta,\b{\Phi})$ of $\widehat{\cal{P}}_{\b{q}}$, we have that $\cal{G}_y^{\b{q}}(u)=0$ for all $u\in \supp(\zeta)$.
	\end{corr}

	Finally, we collect the following result from Theorems \ref{theorem:intro:main:sphere} and \ref{theorem:intro:parisi-measure:sphere}.
	
	\begin{corr}
		\label{corr:intro:Euclidean theorem at a point q}
		If one chooses $q\in (0,\infty)$ and forms the constant vector $\b{q}\in (0,\infty)^{\Om}$, then we have that
		\[\inf_{(\zeta,\b{\Phi})\in \widehat{\mathscr{Y}}(\b{q})}\widehat{\cal{P}}_{\b{q}}(\zeta,\b{\Phi})=\inf_{\zeta\in \mathscr{Y}(q)}\cal{P}_{q}(\zeta).\]
		Moreover, the right-hand side is strictly convex in $\zeta$ with a unique minimizer $\zeta_q$ given by the unique solution the equations
		\[\zeta\left(\big\{s\in [0,q]:f_{q}(s)=\sup_{0<s'<q}f_{q}(s') \big\}\right)=1,\label{eqn:convexity:minimization eqn:Euclidean: measure}\]
		where for $s\in (0,q)$, we define
		\[f_{q}(s)=\int_0^s F_{q}(u)du,\;\;\;\; F_{\beta,q}(s)=-2B'(2(q-s))+\int_0^s  K'( \delta(u);t)du,\]
		where $\delta$ is as in (\ref{eqn:def:delta-P}).
	\end{corr}
	With all of these results collected, we define the function
	\[\widehat{\mathscr{P}}(\b{q})=\inf_{(\zeta,\b{\Phi})\in \widehat{\mathscr{Y}}(\b{q}_0)}\widehat{\cal{P}}_{\b{q}}(\zeta,\b{\Phi}).\]
	We may now state the main result of this section.
	\begin{prop}
		\label{prop:concavity:unique concavity in q}
		The function $\widehat{\mathscr{P}}(\b{q})$ is maximized over $\b{q}\in (0,\infty)^{\Om}$ at a unique point $\b{q}^*$ which satisfies $\b{q}^*(x)=q^*$ for all $x\in \Om$ and some $q^*\in (0,\infty)$. This $q^*$ may be characterized as the unique value of $q$, such that
		\[\int_0^{q}\zeta_{q}([0,u])du=R_1(\mu;t),\]
		where $\zeta_q$ is as in Corollary \ref{corr:intro:Euclidean theorem at a point q}.
	\end{prop}
	
	Before discussing the proof of this, we observe the following corollary of Corollary \ref{corr:intro:Euclidean theorem at a point q} and Proposition \ref{prop:concavity:unique concavity in q}.
	
	\begin{corr}
		\label{corr:main:critical points euclidean}
		For fixed $q$, there is a unique minimizer $\zeta_q$ of $\cal{P}_{q}$, given as the unique $\zeta$ which solves
		\[\zeta\left(\big\{s\in [0,q]:f_{q}(s)=\sup_{0<s'<q}f_{q}(s') \big\}\right)=1,\label{eqn:minimization eqn:Euclidean: measure-extended}\]
		where for $s\in (0,q)$, we define
		\[f_{q}(s)=B(2(q-s))+\int_0^s\left(\int_0^u  K'( \delta(r);t)dr\right)du.\]
		Denoting this minimizer $\zeta_q$, there is a unique $q^*$ such that
		\[\int_0^{q^*}\zeta_{q^*}([0,u])du=R_1(\mu;t).\]
		This $q^*$ has the property that $\b{q}^*$ is the unique maximizer of $\widehat{\mathscr{P}}(\b{q})$.
	\end{corr}
	
	We observe that Theorems \ref{theorem:intro:main:critical points euclidean} and \ref{theorem:intro:main:Euclidean} follow immediately from Corollary \ref{corr:main:critical points euclidean} and Theorem	\ref{theorem:paper 1:euclidean result}. 
	
	We now begin preparing for the proof of Proposition \ref{prop:concavity:unique concavity in q}, for which the first major step will be the following.
	
	\begin{lem}
		\label{lem:concavity:concavity in q}
		The function $\b{q}\mapsto \widehat{\mathscr{P}}(\b{q})$ is concave in $\b{q}\in (0,\infty)^{\Om}$.
	\end{lem}

	As we know there is at least one maximizer of $\widehat{\mathscr{P}}(\b{q})$ (see Corollary I.2.9), we note that if Lemma \ref{lem:concavity:concavity in q} established instead strict convexity, then Proposition \ref{prop:concavity:unique concavity in q} would follow immediately. Instead, we show this uniqueness after establishing Lemma \ref{lem:concavity:concavity in q}, analyze the behavior of $\widehat{\mathscr{P}}(\b{q})$ around a symmetric maximizer.
	
	We now proceed to the proof of Lemma \ref{lem:concavity:concavity in q}. To do this, we will show that, for fixed $\b{q}\in (0,\infty)^{\Om}$, there are minimizers of $\widehat{\cal{P}}_{\b{q}}$ which lie in a certain subset of the domain. We then show that for a small neighborhood of $\b{q}'\in (0,\infty)^{\Om}$ around $\b{q}$, we may naturally identify this subset of the domain with subsets of the domain of $\widehat{\cal{P}}_{\b{q}'}$, which both contain in the minimizer, and are such that $\widehat{\cal{P}}_{\b{q}'}$ becomes a linear function in $\b{q}'$. Thus locally, we may view $\widehat{\mathscr{P}}(\b{q})$ as the infimum over a fixed set of linear functions of $\b{q}$, which obviously implies concavity. 
	
	In the case where $\b{\Phi}=\b{I}$, these restrictions and mapping are quite simple. For $\epsilon>0$, we relate a measure $\zeta$ on $[0,q]$ to one $\zeta'$ on $[0,q+\epsilon]$ by letting $\zeta'([0,s+\epsilon])=\zeta([0,s])$ for $s\in [0,q]$ and $\zeta'([0,\epsilon])=0$. Of course, we may also go the other direction, and associate a measure $[0,q+\epsilon]$ to one on $[0,q]$ as long as the support of this measure is disjoint from $[0,\epsilon]$. Under these mappings, one may show that $\widehat{\cal{P}}_{\b{q}}$ is linear in $\epsilon$.
	
	Unfortunately, a number of issues make this picture more complicated for general $\b{\Phi}$. We begin by defining mappings between the domains of $\widehat{\cal{P}}_{\b{q}}$ and $\widehat{\cal{P}}_{\b{q}'}$. To begin, fix $(\zeta,\b{\Phi})\in \widehat{\mathscr{Y}}(\b{q})$. Given an additional choice of $\b{r}\in (0,\infty)^{\Om}$ we define an injection $\varphi_{\b{q},\b{q}+\b{r}}:\widehat{\mathscr{Y}}(\b{q})\to \widehat{\mathscr{Y}}(\b{q}+\b{r})$. Intuitively, we should like to push the existing measure data from $\prod_{x\in \Om}[0,\b{q}(x)]$ to $\prod_{x\in \Om }[\b{r}(x),\b{q}(x)+\b{r}(x)]$, by translation, filling the remaining portions with zero measure. That is, letting $r_t=\frac{1}{|\Om|}\sum_{x\in \Om}\b{r}(x)$, we would like to define
	\[\tilde{\Phi}_x(s)=\Phi_x(s-r_t)+\b{r}(x)\text{ for } s\in [r_t,q_t+r_t],\;\; \tilde{\Phi}(s)=s\left(\frac{\b{r}(x)}{r_t}\right) \text{ for } s\in [0,r_t],\]
	\[\tilde{\zeta}([0,s])=\zeta([0,s-r_t]) \text{ for } s\in [r_t,q_t+r_t],\;\;\tilde{\zeta}([0,r_t))=0.\]
	Unfortunately this choice need not preserve (\ref{eq:condition for Psi}). To rectify this, we define a function $\alpha:[0,q_t+r_t]\to [0,q_t+r_t]$ by
	\[\alpha(s)=\frac{q_t+r_t}{|\Om|}\sum_{x\in \Om}(\b{q}(x)+\b{r}(x))^{-1}\tilde{\Phi}_x(s).\]
	This clearly defines a monotone increasing bijection, and we see that $\b{\Phi}'=\tilde{\b{\Phi}}\circ \alpha^{-1}$ satisfies (\ref{eq:condition for Psi}) by definition. Similarly, we define $\zeta'$ by letting $\zeta'([0,s])=\tilde{\zeta}([0,\alpha^{-1}(s)])$. Finally we define $\varphi_{\b{q},\b{q}+\b{r}}(\zeta,\b{\Phi})=(\zeta',\b{\Phi}')$.
	
	Our next step will relate $\hat{\cal{P}}_{\b{q}+\b{r}}(\varphi_{\b{q},\b{q}+\b{r}}(\zeta,\b{\Phi}))$ and $\widehat{\cal{P}}_{\b{q}}(\zeta,\b{\Phi})$. We will also use $\b{\delta}$ and $\b{\delta}'$ to denote the functions respective to $\widehat{\cal{P}}_{\b{q}}(\zeta,\b{\Phi})$ and $\widehat{\cal{P}}_{\b{q}'}(\varphi_{\b{q},\b{q}+\b{r}}(\zeta,\b{\Phi}))$. We first note that
	\[
	\begin{split}
		&\int_0^{q_t+r_t}\zeta'([0,s])B'\left(2(\b{q}(x)+\b{r}(x)-\Phi_x'(s))\right)(\Phi_x')'(s)ds=\\
		&\int_0^{q_t+r_t}\tilde{\zeta}([0,s])B'\left(2(\b{q}(x)+\b{r}(x)-\tilde{\Phi}_x(s))\right)(\tilde{\Phi}_x)'(s)ds=\\
		&\int_0^{q_t}\zeta([0,s])B'\left(2(\b{q}(x)-\Phi_x(s))\right)\Phi_x'(s)ds.
	\end{split}
	\]
	We similarly observe that for $s>r_t$ 
	\[\delta'_x(\alpha(s))=\int_{s}^{q_t+r_t}\tilde{\zeta}([0,s])(\tilde{\Phi}_x)'(s)ds=\delta_x(s-r_t),\]
	while for $s\le r_t$ we have that $\delta'_x(\alpha(s))=\delta_x(0)$. Noting that if $\zeta([0,q_*))=1$ and $\Phi_x(q_*)<\b{q}(x)$ for each $x\in \Om$, then we have that $\zeta'([0,\alpha(q_*+r_t)))=1$ and $\Phi_x'(\alpha(q_*+r_t))<\b{q}(x)+\b{r}(x)$. So using this point, we see that
	\[\int_{0}^{\alpha(q_*+r_t)} K_x^{-t\Delta}(\b{\delta}'(q))(\Phi_x')'(q)dq=\int_0^{\alpha(r_t)}K_x^{-t\Delta}(\b{\delta}(0))(\Phi_x')'(q)dq+\int_{r_t}^{q_*+r_t} K_x^{-t\Delta}(\tilde{\b{\delta}}(q))\tilde{\Phi}_x'(q)dq=\]
	\[K_x^{-t\Delta}(\b{\delta}(0))\tilde{\Phi}_x(r_t)+\int_{r_t}^{q_*+r_t} K_x^{-t\Delta}(\tilde{\b{\delta}}(q))\tilde{\Phi}_x'(q)dq=K_x^{-t\Delta}(\b{\delta}(0))\b{r}(x)+\int_{0}^{q_*} K_x^{-t\Delta}(\b{\delta}(q))\Phi'_x(q)dq.\]
	Noting similarly that $\Lambda^{-t\Delta}(\delta'(\alpha(q_*+r_t))))=\Lambda^{-t\Delta}(\delta(q_t))$, we see that
	\[\hat{\cal{P}}_{\b{q}+\b{r}}(\varphi_{\b{q},\b{q}+\b{r}}(\zeta,\b{\Phi}))=\widehat{\cal{P}}_{\b{q}}(\zeta,\b{\Phi})+\frac{1}{2L^d}\sum_{x\in \Om}\b{r}(x)K^{-t\Delta}_x(\b{\delta}(0)).\label{eqn:concavity:value of Parisi under varphi forward}\]
	In particular, fixing $\b{q}$, the function $\b{r}\mapsto \hat{\cal{P}}_{\b{q}+\b{r}}(\varphi_{\b{q},\b{q}+\b{r}}(\zeta,\b{\Phi}))$ is linear.
	
	Recalling that the pointwise infimum of a family of concave functions is concave, as well as the fact that concavity is a local property, we see that Lemma \ref{lem:concavity:concavity in q} follows if we show the following.
	
	\begin{lem}
		\label{lem:concavity:localness}
		Fix $\b{q}^*\in (0,\infty)^{\Om}$. Then there is some $\b{q}_0\in (0,\infty)^{\Om}$ and $\epsilon>0$, with $\b{q}^*\in \prod_{x\in \Om}(\b{q}_0(x),\b{q}_0(x)+\epsilon)$, and such that for all $\b{q}\in \prod_{x\in \Om}[\b{q}_0(x),\b{q}_0(x)+\epsilon]$
		\[\inf_{(\zeta,\b{\Phi})\in \widehat{\mathscr{Y}}(\b{q})}\widehat{\cal{P}}_{\b{q}}(\zeta,\b{\Phi})=\inf_{(\zeta,\b{\Phi})\in \widehat{\mathscr{Y}}(\b{q}_0)}\hat{\cal{P}}_{\b{q}_0}(\varphi_{\b{q}_0,\b{q}}(\zeta,\b{\Phi})).\]
	\end{lem}
	
	To do this we will study the image of $\varphi_{\b{q}_0,\b{q}}$ by roughly defining a partial inverse function. However this will not be a true partial inverse. Instead it will be up to an equivalence relation that respects the values of $\widehat{\cal{P}}_{\b{q}}(\zeta,\b{\Phi})$, which we now define.
	
	\begin{defin}
		Fix $\b{q}\in (0,\infty)^{\Om}$, and a probability measure $\zeta$ on $[0,q_t]$. We will say two functions $\b{\Phi},\b{\Phi}':[0,q_t]\to \prod_{x\in \Om}[0,\b{q}(x)]$ are $\zeta$-equivalent if $\b{\Phi}(s)=\b{\Phi}'(s)$ for $s\in \supp(\zeta)$. We will say pairs $(\zeta,\b{\Phi})$ and $(\zeta',\b{\Phi})$ are $\cal{P}$-equivalent if $\zeta=\zeta'$ and $\b{\Phi}$ is $\zeta$-equivalent to $\b{\Phi}'$.
	\end{defin}
	
	With this definition, we see that (\ref{eqn:relation between A and P}) and  Lemma \ref{lem:general:parisi doesn't depend away from support} above implies that if $(\zeta,\b{\Phi})$ and $(\zeta,\b{\Phi}')$ are $\cal{P}$-equivalent then $\widehat{\cal{P}}_{\b{q}}(\zeta,\b{\Phi})=\widehat{\cal{P}}_{\b{q}}(\zeta,\b{\Phi}')$.
	
	We will now define our partial inverse $\varphi_{\b{q},\b{q}+\b{r}}$, up to this equivalence relation, which we will denote $\hat{\varphi}_{\b{q}+\b{r},\b{q}}$, with domain given by the subset $\widehat{\mathscr{Y}}(\b{q}+\b{r})$ defined by
	\[\widehat{\mathscr{Y}}_{\b{r}}(\b{q})=\{(\zeta,\b{\Phi})\in \widehat{\mathscr{Y}}(\b{q}+\b{r}):\b{\Phi}(r_t)=\b{r},\;\;\zeta([0,r_t))=0\}.\]
	We define
	\[\hat{\Phi}_x(s)=\Phi_x(s+r_t)-\b{r}(x) \text{ and }\hat{\zeta}([0,s])=\zeta([0,s+r_t])\text{ for } s\in [0,q_t].\] We further define $\beta:[0,q_t]\to [0,q_t]$ by
	\[\beta(s)=\frac{q_t}{|\Om|}\sum_{x\in \Om}\b{q}(x)^{-1}\hat{\Phi}_x(s).\]
	Then letting $\b{\Phi}''=\hat{\b{\Phi}}\circ \beta^{-1}$ and $\zeta''([0,s])=\hat{\zeta}([0,\beta^{-1}(s)])$, we define $\hat{\varphi}_{{\b{q}+\b{r},\b{q}}}(\zeta,\b{\Phi})=(\zeta'',\b{\Phi}'')$.
	
	Next, let us denote by $\b{\delta}''$ and $\b{\delta}$ as the functions corresponding to $(\zeta'',\b{\Phi}'')$ and $(\zeta,\b{\Phi})$, respectively. Then as $(\zeta,\b{\Phi})\in \mathscr{Y}_{\b{r}}(\b{q})$ we see for $s\in [0,r_t]$ that $\b{\delta}(s)=\b{\delta}(r_t)$ and for $s\in [0,q_t]$ and $x\in \Om$ that 
	\[\delta''_x(\beta(s))=\delta_x(s+r_t).\]
	From this, it is easy to check, as in (\ref{eqn:concavity:value of Parisi under varphi forward}), that we have
	\[\widehat{\cal{P}}_{\b{q}}(\hat{\varphi}_{\b{q}+\b{r},\b{q}}(\zeta,\b{\Phi}))=\widehat{\cal{P}}_{\b{q}}(\zeta,\b{\Phi})-\frac{1}{2|\Om|}\sum_{x\in \Om}\b{r}(x)K^{-t\Delta}_x(\b{\delta}(0)).\label{eqn:concavity:value of Parisi under varphi backward}\]
	
	We now verify that $\varphi_{\b{q},\b{q}+\b{r}}\circ \hat{\varphi}_{\b{q}+\b{r},\b{q}}(\zeta,\b{\Phi}):=(\zeta^0,\b{\Phi}^0)$ is $\cal{P}$-equivalent to $(\zeta,\b{\Phi})$, so that in particular \[\hat{\cal{P}}_{\b{q}+\b{r}}(\zeta,\b{\Phi})=\hat{\cal{P}}_{\b{q}+\b{r}}(\varphi_{\b{q},\b{q}+\b{r}}\circ \hat{\varphi}_{\b{q}+\b{r},\b{q}}(\zeta,\b{\Phi})).\] If we denote $\gamma(s)=\beta^{-1}(\alpha^{-1}(s)-r_t)+r_t$, then it is clear that for $s\in[r_t,q_t+r_t]$, we have $\b{\Phi}^0(s)=\b{\Phi}(\gamma(s))$. Note that as both $\b{\Phi}^0$ and $\b{\Phi}$ satisfy (\ref{eq:condition for Psi}), we see that in fact $\gamma(s)=s$ for $s\in [r_t,q_t+r_t]$, so that $\b{\Phi}^0(s)=\b{\Phi}(s)$ on this domain. We also see that $\zeta^0([0,s])=\zeta([0,\gamma(s)])=\zeta([0,s])$, and as $\zeta([0,r_t))=0$, we conclude that $\zeta''=\zeta$. Combining this with the previous result, that $\b{\Phi}^0(s)=\b{\Phi}(s)$ for $s\in [r_t,q_t+r_t]$, we see that $\b{\Phi}^0$ and $\b{\Phi}$ are $\zeta$-equivalent, giving the desired claim.
	
	The next result we will need is a similar composition law. Take $\b{r},\b{r'}\in [0,\infty)^{\Om}$. Then  $\varphi_{\b{q},\b{q}+\b{r}+\b{r}'}(\zeta,\b{\Phi})$ is $\cal{P}$-equivalent to $\varphi_{\b{q}+\b{r},\b{q}+\b{r}+\b{r}'}\circ \varphi_{\b{q},\b{q}+\b{r}}(\zeta,\b{\Phi})$. This follows from a similar argument as before, which will be omitted.
	
	From both of these statements, we will be able to obtain the following result.
	\begin{lem}
		\label{lem:concavity:image of phi}
		Fix choices of $\b{q},\b{r}\in (0,\infty)^{\Om}$. Then every element of
		\[\{(\zeta,\b{\Phi})\in \widehat{\mathscr{Y}}_{\b{q}}:\text{ there is }s>r_t \text{ such that } \zeta([0,s))=0 \text{ and } \b{\Phi}(s)\ge \b{r}\}.\label{eqn:def:concavity:domain}\]
		is $\cal{P}$-equivalent to one in $\mathrm{im}(\varphi_{\b{q}-\b{r},\b{q}})$.
	\end{lem}
	\begin{proof}
		Fix a pair $(\zeta,\b{\Phi})$ in the set defined in (\ref{eqn:def:concavity:domain}). We note that by assumption, $(\zeta,\b{\Phi})\in \mathscr{Y}_{\b{q}+\b{r},\b{q}+\b{r}-\b{\Phi}(s)}$. Moreover, we have that $\b{\Phi}(s)\ge \b{r}$, so that we may consider \[\varphi_{\b{q}-\b{r},\b{q}}\circ \varphi_{\b{q}-\b{\Phi}(s),\b{q}-\b{r}}\circ \hat{\varphi}_{\b{q},\b{q}-\b{\Phi}(s)}(\zeta,\b{\Phi})\in \text{im}(\varphi_{\b{q}-\b{r},\b{q}}).\]
		We first note that invoking the claim on compositions, this is $\cal{P}$-equivalent to 
		\[\varphi_{\b{q}-\b{\Phi}(s),\b{q}}\circ \hat{\varphi}_{\b{q},\b{q}-\b{\Phi}(s)}(\zeta,\b{\Phi}).\]
		But as this is further $\cal{P}$-equivalent to $(\zeta,\b{\Phi})$ by the first claim, we have established the result.
	\end{proof}
	
	\begin{proof}[Proof of Lemma \ref{lem:concavity:localness}]
		Let us choose $\epsilon,\delta>0$ so that the conclusion of Corrolary \ref{corr:concavity:avoids 0} holds for $\b{q}\in \prod_{x\in \Om}[\b{q}_0(x)-\epsilon,\b{q}_0(x)+\epsilon]$. Let $(\zeta_{\b{q}},\b{\Phi}_{\b{q}})$ as well denote any choice of minimizer for $\widehat{\cal{P}}_{\b{q}}$. We note that the inequality $\b{\Phi}(\epsilon)\ge \delta$ forces $\epsilon\ge \delta$. Thus we see that for $\b{q}\in \prod_{x\in \Om}[\b{q}_0(x)-\delta,\b{q}_0(x)+\delta]$, we have that $\zeta_{\b{q}}([0,\epsilon))=0$ and $(\Phi_{\b{q}})_x(\epsilon)\ge \delta\ge \b{q}(x)-\b{q}_0(x)$ for all $x\in \Om$. In particular, by Lemma \ref{lem:concavity:image of phi} we have that $(\zeta_{\b{q}},\b{\Phi}_{\b{q}})$ is $\cal{P}$-equivalent to an element in the image of $\varphi_{\b{q}_0,\b{q}}$. As this element must also be in minimizer, this implies we may restrict the minimization to the image of $\varphi_{\b{q}_0-\delta,\b{q}}$ for $\b{q}\in \prod_{x\in \Om}[\b{q}_0(x)-\delta,\b{q}_0(x)+\delta]$, which gives the desired claim.
	\end{proof}
	
	With the proof of Lemma \ref{lem:concavity:localness} established, we have established Lemma \ref{lem:concavity:concavity in q} as well. What is left is to prove Proposition \ref{prop:concavity:unique concavity in q}.
	
	The main idea starts by noting that the set of maximizers of $\widehat{\mathscr{P}}(\b{q})$ is non-empty and convex. Moreover by transitivity of $\Delta$ (see Definition \ref{def:transitive}) we may also find some $q^*\in (0,\infty)$ such that the constant vector $\b{q}^*\in (0,\infty)^{\Om}$ is a maximizer. If this is not the unique maximizer, then we can find maximizers arbitrarily close to $\b{q}^*$. We then apply the functions $\varphi$ and $\hat{\varphi}$ to move the minimizing pair of $\widehat{\cal{P}}_{\b{q}^*}$ to another maximizer $\b{q}'$. We show that maximization condition on $\b{q}^*$ implies that the value of the functional will not change after we do this. However, this shows that the pushed measure is infact a minimizer of $\widehat{\cal{P}}_{\b{q}'}$. Thus both satisfy the condition of Corollary \ref{corr:concavity:minimization in G for P} at these different points, which when combined with Corollary \ref{corr:concavity:minimization in G for P}, we show is impossible.
	
	\begin{proof}[Proof of Proposition \ref{prop:concavity:unique concavity in q}]	
		First we will need some preliminary calculations. For these assume we have some $q^*\in (0,\infty)$, $\zeta\in \mathscr{Y}(q^*)$ and $\b{\epsilon}\in \R^{\Om}$ such that if we let $\epsilon_*:=-\min_{x\in \Om}\b{\epsilon}(x)$ then $\zeta([0,\epsilon_*))=0$. 
		Then we define
		\[
		(\zeta^{\b{\epsilon}},\b{\Phi}^{\b{\epsilon}})=\varphi_{\b{q}^*-\b{\epsilon}_*,\b{q}^*+\b{\epsilon}}\circ \hat{\varphi}_{\b{q}^*,\b{q}^*-\b{\epsilon}_*}(\zeta,\b{I})\in \widehat{\mathscr{Y}}(\b{q}^*+\b{\epsilon})
		\]
		To more precisely compute $(\zeta^{\b{\epsilon}},\b{\Phi}^{\b{\epsilon}})$, let us write $\epsilon_t=\frac{1}{|\Om|}\sum_{x\in \Om}\b{\epsilon}(x)$. Define for $x\in \Om$ and $s\in [\epsilon_t+\epsilon_*,q^*+\epsilon_t]$
		\[\widetilde{\Phi}^{\b{\epsilon}}_x(s)=s-\epsilon_t+\b{\epsilon}(x),\;\;\; \widetilde{\zeta}^{\b{\epsilon}}([0,s])=\zeta([0,s-\epsilon_t]),\]
		and for $s\in [0,\epsilon_t+\epsilon_*)$
		\[\widetilde{\Phi}^{\b{\epsilon}}_x(s)=s\left(\frac{\b{\epsilon}(x)+\epsilon_*}{\epsilon_t+\epsilon_*}\right),\;\;\; \widetilde{\zeta}^{\b{\epsilon}}([0,s])=0,\]
		and finally for $s\in [0,q^*+\epsilon_t]$
		\[\alpha_{\b{\epsilon}}(s)=\frac{q_t+\epsilon_t}{|\Om|}\sum_{x\in \Om}(q+\b{\epsilon}(x)+\epsilon_*)^{-1}\widetilde{\Phi}^{\b{\epsilon}}_x(s).\]
		It is then routine to check that
		\[(\zeta^{\b{\epsilon}},\b{\Phi}^{\b{\epsilon}})=(\widetilde{\zeta}^{\b{\epsilon}}\circ \alpha_{\epsilon}^{-1},\widetilde{\b{\Phi}}^{\b{\epsilon}}\circ \alpha_{\epsilon}^{-1}).\]
		Letting $\b{\delta}^{\b{\epsilon}}$ and $\b{\delta}$ denote the functions corresponding to the data of $(\zeta^{\b{\epsilon}},\b{\Phi}^{\b{\epsilon}})$ and $(\zeta,\b{I})$ in the respective definitions $\hat{\cal{P}}_{\b{q}^*+\b{\epsilon}}$ and $\cal{P}_{q^*}$, we see that for $x\in \Om$ and $s\in [\epsilon_t+\epsilon_*]$
		\[\delta_x^{\b{\epsilon}}(\alpha_{\b{\epsilon}}(s))=\delta_{x}(s-\epsilon_t).\]
		Now noting that $\zeta([0,\epsilon_*))=0$ and $\widetilde{\zeta}^{\b{\epsilon}}([0,\epsilon_*+\epsilon_t))=0$ we see that for $x\in \Om$ and $s\in [0,\epsilon_*+\epsilon_t]$
		\[\delta_x^{\b{\epsilon}}(\alpha_{\b{\epsilon}}(s))=\delta_x(0).\]
		In particular by (\ref{eqn:concavity:value of Parisi under varphi forward}) and (\ref{eqn:concavity:value of Parisi under varphi backward}) we see that
		\[
		\hat{\cal{P}}_{\b{q}^*+\b{\epsilon}}(\zeta^{\b{\epsilon}},\b{\Phi}^{\b{\epsilon}})=\cal{P}_{q^*}(\zeta)+\frac{1}{2|\Om|}\sum_{x\in \Om}\b{\epsilon}(x)\left(K_x^{-t\Delta}(\b{\delta}(0))-\mu\right).\label{eqn:lemma:shifts around max in q}
		\]
		
		Now with these preliminaries aside, fix $q^*$ to be such that $\b{q}^*$ maximizes $\widehat{\mathscr{P}}$. Fixing $\zeta$ as well to be the measure which minimizes $\cal{P}_{q^*}$, we note that by Theorem \ref{theorem:overview:spherical:generic and symmetric} we have that $(\zeta,\b{I})$ is a maximizer of $\widehat{\cal{P}}_{\b{q}^*}$ with
		\[\cal{P}_{q^*}(\zeta)=\widehat{\cal{P}}_{\b{q}}(\zeta,\b{I}).\]
		In particular, by Corollary \ref{corr:concavity:avoids 0}, we have that $0\notin \supp (\zeta)$. Thus we see by (\ref{eqn:lemma:shifts around max in q}) that for $\b{q}^*$ to indeed maximize $\cal{Q}$, we must have that $K^{-t\Delta}(\delta(0))=\mu$, so that $\zeta$ satisfies
		\[\delta(0)=\int_0^{q^*}\zeta([0,u])du=R_1(\mu;t).\label{eqn:ignore-eyegore}\]

		Now let us assume that the set of maximizers of $\widehat{\mathscr{P}}$ contains something other than $\b{q}^*$. As this set is convex, we may find some $\b{r}\in \R^\Om\setminus \{\b{0}\}$ such that $\b{q}^*+t\b{r}$ maximizes $\b{Q}$ for all $t\in [0,1]$. Let $r_*=-\min_{x\in \Om}\b{r}(x)$. As $0\notin \supp(\zeta)$, we may find small $\epsilon_0>0$ so that $\zeta([0,\epsilon_0r_*])=0$. Thus we may apply the above computations with $\b{\epsilon}=\epsilon_0\b{r}$ to get that
		\[\widehat{\cal{P}}_{\b{q}^*+\epsilon}(\zeta^{\epsilon},\b{\Phi}^{\epsilon})=\cal{P}_{q^*}(\zeta).\label{eqn:convaity:ignore-1520}\]
		However, as $\cal{P}_{q^*}(\zeta)=\widehat{\cal{P}}(\b{q}^*)=\widehat{\cal{P}}(\b{q}^*+\b{\epsilon})$, as they both maximizers of $\widehat{\cal{P}}$, we see that $(\zeta^{\epsilon},\b{\Phi}^{\epsilon})$ is a minimizer of $\hat{\cal{P}}_{\b{q}^*+\epsilon}$. Thus by Corollary \ref{corr:concavity:minimization in G for P} we have that $\b{\cal{G}}^{\b{q}^*+\b{\epsilon}}(u)=\b{0}$ for $u\in \supp(\zeta^{\epsilon})$. Employing the above equations, we see equivalently that\[\b{\cal{G}}^{\b{q}^*+\b{\epsilon}}(\alpha_{\epsilon}(u+\epsilon_t))=\b{0}\text{  for  } u\in \supp(\zeta).\label{eqn:concavity:ignore-1614}\]
		
		To study this equation, note that the function $\b{\cal{G}}^{\b{q}^*}$ associated to  $(\zeta,\b{I})$ are all copies of some fixed function $\cal{G}^{q^*}$ by symmetry. Using Corollary \ref{corr:concavity:minimization in G for P} again we see that for $u\in \supp(\zeta)$ $\cal{G}^{q^*}(u)=0$. Now we compute $\b{\cal{G}}^{\b{\epsilon}}(\alpha_{\epsilon}(u+\epsilon_t))$. For any $y\in \Om$ and $u\in [\epsilon_*,q^*]$ we have
		\[\cal{G}_y^{\b{\epsilon}}(\alpha_{\epsilon}(u+\epsilon_t))=-2B'(2(q^*+\epsilon(x)-\widetilde{\Phi}^{\epsilon}_x(u+\epsilon_t)))+\sum_{x\in \Om}\left(\int_0^{u+\epsilon_t}\D_y K^{-t\Delta}_x(\b{\delta}^{\epsilon}(\alpha^{\epsilon}(s))d\widetilde{\Phi}_x^{\epsilon}(ds)\right)=\]
		\[-2B(2(q^*-u))+\sum_{x\in \Om}\left(\int_{\epsilon_*}^{u}\D_y K^{-t\Delta}_x(\b{\delta}(s))ds+(\epsilon_t+\epsilon_*)\left(\frac{\b{\epsilon}(x)+\epsilon_*}{\epsilon_t+\epsilon_*}\right)\D_y K^{-t\Delta}_x(\b{\delta}(0))\right)=\]
		\[\cal{G}(u)+\sum_{x\in \Om}\left(-\epsilon_*\D_y K^{-t\Delta}_x(\b{\delta}(0))+(\b{\epsilon}(x)+\epsilon_*)\D_y K^D_x(\b{\delta}(0))\right)=\]
		\[\cal{G}(u)+\sum_{x\in \Om}\b{\epsilon}(x)\D_y K^{-t\Delta}_x(\b{\delta}(0)).\]
		Now if $u\in \supp(\zeta)$ we have that $\cal{G}_y^t(\alpha_{\epsilon}(u+\epsilon_t))=\cal{G}(u)=0$. In particular, we have that for all $y\in \Om$
		\[\sum_{x\in \Om}\b{\epsilon}(x)\D_y K^{-t\Delta}_x(\b{\delta}(0))=0.\label{eqn:ignore-1643}\]
		Now recall that
		\[\D_y K^D_x(\b{\delta}(0))=-[(\delta(0)I-t\Delta)^{-1}]_{xy}^{2}.\]
		As $\delta(0)I-t\Delta$ is positive definite, the matrix $-((\delta(0)I-t\Delta)^{-1})^{2\odot}$ is negative definite by the Schur Product Theorem, and in particular, this matrix is invertible. Thus as $\b{\epsilon}\neq \b{0}$ the equation (\ref{eqn:ignore-1643}) yields a contradiction. Thus we have shown that the minimizer $\b{q}^*$ is unique.
		
		All that is left is to show that $q^*$ is the only value of $q\in(0,\infty)$ for which the minimizing measure $\zeta_q$ satisfies (\ref{eqn:ignore-eyegore}). To do so, choose another such value $q'$. Note by the work proceeding (\ref{eqn:ignore-eyegore}) (in particular (\ref{eqn:lemma:shifts around max in q})) this would imply that $\b{q}'$ is in fact a local maximum of $\widehat{\cal{P}}$. By convexity this means it would be a global maximum, but as we have just shown this is unique, we must have that $q'=q^*$.
	\end{proof}

	\section{Proof of Theorem \ref{theorem:intro:main:Euclidean parameters identification}\label{section:euclidean proofs}}
	
	In this section, we will prove Theorem \ref{theorem:intro:main:Euclidean parameters identification}, which identifies the saddle pair $(\qc,\zetac)$, of our functional. Our first step, as above, will be to reduce ourselves to the case $h=0$. For this we need a more precise version of Remark \ref{remark:h doesn't matter}, again following from Lemma I.2.1.
	
	\begin{remark}	
		\label{remark:h doesn't matter gibbs}
		For the moment, let us denote the dependence of the Gibbs measure associated to $\cal{H}_N$ by writing $\<*\>_{h}$. Let $\b{e}_1\in (\R^{N})^{\Om}$ denote the vector such that $[\b{e}_1(x)]_j=\delta_{j1}$. Then we have the following distributional equality: for any function $f:(\R^{N})^{\Om}\to \R$ we the
		\[\<f(\b{u})\>_{h}\disteq \left\<f\left(\b{u}-\sqrt{N}\frac{h}{\mu}\b{e}_1\right)\right\>_0,\]
		and moreover this distributional equality holds jointly over all choices of $f$.
	\end{remark}
	
	This illustrates that the strange behavior in Theorem \ref{theorem:intro:main:Euclidean parameters identification} in the case of $h\neq 0$ comes from working with $\b{u}$ and not $\b{u}-\sqrt{N}\frac{h}{\mu}\b{e}_1$, with the extra terms coming from this shift. To show this, we first give the following more powerful form of Theorem \ref{theorem:intro:main:Euclidean parameters identification} in the case of $h=0$, whose proof will be the focus of this section.
	
	\begin{prop}
		\label{prop:euclidean proofs:main result}
		Assume that $h=0$, and let $(\qc,\zetac)$ be the pair from Corollary \ref{corr:main:critical points euclidean}. Then for any $x\in \Om$ and any continuous $f:\R\to \R$ of sub-exponential growth, we have that
		\[\lim_{N\to \infty}\E\left\<f(\|\b{u}(x)\|^2_N)\right\>=f(\qc),\label{eqn:euclidean proofs: f(u^2)=f(q)}\]
		\[\lim_{N\to \infty}\E \left\<f((\b{u}(x),\b{u}'(x))_N)\right\>=\int_0^{\qc}f(r)\zetac(dr).\label{eqn:euclidean proofs: f((u,u'))=zeta(r)}\]
	\end{prop}
	
	\begin{proof}[Proof of Theorem \ref{theorem:intro:main:Euclidean parameters identification}]
		From Remark \ref{remark:h doesn't matter gibbs} this is equivalent to showing, in the case that $h=0$, that we have that 
		\[\lim_{N\to \infty}\E\left\<\left|\left\|\b{u}-\sqrt{N}\frac{h}{\mu}\b{e}_1\right\|^2_N-\left(\qc+\frac{h^2}{\mu^2}\right)\right|\right\>=0,\]
		and that for any bounded continuous function $f:\R\to \R$ we have that
		\[\lim_{N\to \infty}\E\left\<f\left(\left(\b{u}-\sqrt{N}\frac{h}{\mu}\b{e}_1,\b{u}'-\sqrt{N}\frac{h}{\mu}\b{e}_1\right)_N\right)\right\>=\int_{0}^{q} f\left(r+\frac{h^2}{\mu^2}\right)\zetac(dr).\]
		We note that
		\[\left\|\b{u}(x)-\sqrt{N}\frac{h}{\mu}e_1\right\|^2_N=\|\b{u}(x)\|^2+\frac{h^2}{\mu^2}+\frac{2h}{\mu}\left(\frac{\b{u}_1(x)}{\sqrt{N}}\right),\]
		and that
		\[\left(\b{u}-\sqrt{N}\frac{h}{\mu}\b{e}_1,\b{u}'-\sqrt{N}\frac{h}{\mu}\b{e}_1\right)_N=(\b{u}(x),\b{u}'(x))_N+\frac{h^2}{\mu^2}+\frac{h}{\mu}\left(\frac{\b{u}_1(x)}{\sqrt{N}}+\frac{\b{u}_1'(x)}{\sqrt{N}}\right).\]
		By continuity of $f$, both of these results would follow if we prove that
		\[\lim_{N\to \infty}\E\left\<\left|\left(\frac{\b{u}_1(x)}{\sqrt{N}}\right)\right|^2\right\>=0.\]
		However note that, for each fixed $x\in \Om$, the marginal measure of $\E\<*\>$ on $\R^N$ is isotropic, as the underlying Hamiltonian is when $h=0$. Using this and Proposition \ref{prop:euclidean proofs:main result}, we have that
		\[\lim_{N\to \infty}\E\left\<(\b{u}_1(x))^2\right\>=\lim_{N\to \infty}\E\left\<\|\b{u}_1(x)\|^2_N\right\>=\qc,\]
		completing the proof.
	\end{proof}
	
	We now move onto the proof of Proposition \ref{prop:euclidean proofs:main result}, which will compose the remainder of this section. Thus for the remainder of the section we will always assume that we are in the case of $h=0$. We will begin by setting up for the proof of (\ref{eqn:euclidean proofs: f(u^2)=f(q)}). For this it is useful to establish a variety of notation. First, for any subset $A\subseteq \R^{\Om}$, we will denote by
	\[Z_{N}(A)=\int_{A}e^{-\cal{H}_N(\b{u})}d\b{u}, \;\;\;\; \<f(\b{u})\>_{A}=\frac{1}{Z_N(A)}\int_{A}f(\b{u})e^{-\cal{H}_N(\b{u})}d\b{u},\]
	the partition function and Gibbs measure of $\cal{H}_N$ on the subset $A$.
	
	Next, for $q\in (0,\infty)$, let us define the annulus
	\[D_N^{\epsilon}(q)=\{\b{u}\in (\R^N)^{\Om}:|\|\b{u}\|^2_N-q|\le \epsilon\}.\]
	The free energy for restrictions of this type were also computed in \cite{Paper1}. In particular, by Corollary I.2.8 we for $\epsilon \in (0,\qc)$ that
	\[\lim_{N\to \infty}N^{-1}|\Om|^{-1}\E \log Z_{N}(D_N^\epsilon(\b{q}_c))=\sup_{\b{q}\in (\qc-\epsilon,\qc+\epsilon)^{\Om}}\widehat{\mathscr{P}}(\b{q}),\]
	\[\lim_{N\to \infty}N^{-1}|\Om|^{-1}\E \log Z_N\left(D_N^\epsilon(\b{q}_c)^c\right)=\sup_{\b{q}\in (0,\infty)^{\Om}\setminus (\qc-\epsilon,\qc+\epsilon)^{\Om}}\widehat{\mathscr{P}}(\b{q}).\]
	By Proposition \ref{prop:concavity:unique concavity in q} we see that
	\[\sup_{\b{q}\in (\qc-\epsilon,\qc+\epsilon)^{\Om}}\widehat{\mathscr{P}}(\b{q})=\widehat{\mathscr{P}}(\b{q}_c),\]
	and if we define 
	\[\widehat{\mathscr{P}}^{\epsilon}(\b{q}_c):=\sup_{\b{q}\in (0,\infty)^{\Om}\setminus (\qc-\epsilon,\qc+\epsilon)^{\Om}}\widehat{\mathscr{P}}(\b{q}),\]
	that we also have that
	\[\widehat{\mathscr{P}}^{\epsilon}(\b{q}_c)<\widehat{\mathscr{P}}(\b{q}_c).\]
	From this and some concentration results, we will obtain the following useful estimate.
	
	\begin{lem}
		\label{lem:euclidean proofs:bound on e-complement}
		For $\epsilon>0$ define
		\[G_N(\epsilon):=\left\<I\left(\|\b{u}\|_N^2\notin D_N^\epsilon(\b{q}_c)\right)\right\>.\]
		Then there is a continuous function $l(\epsilon)>0$ such that for sufficiently large $N$ (possibly dependent on $\epsilon$)
		\[\P\left(G_N(\epsilon)\ge \exp(-Nl(\epsilon))\right)\le 2\exp(-Nl(\epsilon)).\]
		In particular, for sufficiently large $N$, we have that $\E G_N(\epsilon)\le 3\exp(-Nl(\epsilon))$.
	\end{lem}
	\begin{proof}
		The key is to note that
		\[G_N(\epsilon)=\frac{Z_{N}(D_N^\epsilon(\b{q}_c)^c)}{Z_N},\]
		and then employ a number of concentration results to use the above results. For this, let us define $\gamma(\epsilon):=\widehat{\mathscr{P}}(\b{q}_c)-\widehat{\mathscr{P}}^{\epsilon}(\b{q}_c)>0$. Convexity of $\widehat{\mathscr{P}}$ implies that it, and thus $\gamma(\epsilon)$ is continuous. Now we define the events
		\[\cal{E}_1(\epsilon)=\{||\Om|^{-1}N^{-1} \log Z_{N}(D_N^\epsilon(\b{q}_c)^c)-|\Om|^{-1}N^{-1}\E \log Z_{N}(D_N^\epsilon(\b{q}_c)^c)|>\gamma(\epsilon)/4\},\]
		\[\cal{E}_2(\epsilon)=\{||\Om|^{-1}N^{-1} \log Z_N-|\Om|^{-1}N^{-1}\E \log Z_N|>\gamma(\epsilon)/4\}.\]
		Then by Corollary I.A.3, we have for $i=1,2$ that
		\[\P\left(\cal{E}_i(\epsilon)\right)\le 2\exp\left(-\frac{N|\Om|\gamma(\epsilon)^2}{64B(0)}\right).\]
		Conditional on $\cal{E}_1(\epsilon)^c\cap \cal{E}_2(\epsilon)^c$ we have that
		\[|\Om|^{-1}N^{-1}\log G_N(\epsilon)\le \frac{\gamma(\epsilon)}{2}+|\Om|^{-1}N^{-1}\log\left(\frac{\E\log Z_{N}(D_N^\epsilon(\b{q}_c)^c)}{\E\log Z_N}\right).\]
		But by the above computations we have that
		\[\lim_{N\to \infty}\left(\frac{\gamma(\epsilon)}{2}+|\Om|^{-1}N^{-1}\log\left(\frac{\E\log Z_{N}(D_N^\epsilon(\b{q}_c)^c)}{\E\log Z_N}\right)\right)=-\frac{\gamma(\epsilon)}{4}.\]
		Thus we obtain the desired result by taking \[\delta(\epsilon)=\min\left(\frac{\gamma(\epsilon)}{4},\frac{N|\Om|\gamma(\epsilon)^2}{64B(0)}\right).\]
	\end{proof}
	We are now equipped to show the first part of Proposition \ref{prop:euclidean proofs:main result}.
	
	\begin{proof}[Proof of (\ref{eqn:euclidean proofs: f(u^2)=f(q)}) of Proposition \ref{prop:euclidean proofs:main result}]
		We first show that
		\[\lim_{N\to \infty}\E\<f(\|\b{u}(x)\|^2_N)\><\infty.\label{eqn:ignore:sub-exponential functions are finite}\]
		The first step is to bound the tail of $\E\<I(\|\b{u}(x)\|^2_N\ge r)\>$. For this, we define
		\[A_{max}^N(r):=\left\{\b{u}\in (\R^N)^{\Om}:\max_{x\in \Om}\|\b{u}(x)\|_N\ge r\right\}.\] 
		Note that
		\[\E\<I(\|\b{u}(x)\|^2_N\ge r)\>\le \E\<I(\b{u}\in A^N_{max}(\sqrt{r}))\>=\E \left[Z_N(A^N_{max}(\sqrt{r}))/Z_N\right].\]
		By Remark I.2.7 there is some small $c>0$ and large enough $r$ such that
		\[\E Z_N(A^N_{max}(\sqrt{r}))\le e^{-Ncr}.\]
		Let us next define the event
		\[\cal{E}(r)=\left\{N^{-1}\log Z_N>\frac{cr}{2}\right\}.\]
		By Corollary I.A.3 again, we see that for some small $c'>0$ and $r$ large enough that $\P(\cal{E}(r))\le \exp(-Nc'r)$.
		In particular, we have that
		\[\E\<I(\b{u}\in A^N_{max}(\sqrt{r}))\>\le \P(\cal{E}(r))+e^{-N\frac{cr}{2}}\E Z_N(A^N_{max}(\sqrt{r})).\]
		Thus we see there is some small $c''>0$ such that for sufficiently large $r$,
		\[\E\<I(\|\b{u}(x)\|^2_N\ge r)\>\le \exp(-Nc'' r).\]
		As $f$ is subexponential, this immediately shows (\ref{eqn:ignore:sub-exponential functions are finite}).
		
		Now with this bound, note that for any $\epsilon \in (0,\qc)$ we have that
		\[|\E\<f(\|\b{u}(x)\|^2_N)\>-\E\<f(\|\b{u}(x)\|^2_N)I(|\|\b{u}(x)\|^2_N-\qc|\le \epsilon)\>|\]
		\[ \le\E\left\<f(\|\b{u}(x)\|^2_N)I\left(\|\b{u}\|_N^2\notin D_N^\epsilon(\b{q}_c)\right)\right\>\le \sqrt{\E\left\<f(\|\b{u}(x)\|^2_N)^2\>\E\<I\left(\|\b{u}\|_N^2\notin D_N^\epsilon(\b{q}_c)\right)\right\>}.\]
		As $f^2$ is also sub-exponential the first term on the right is bounded in $N$, and so we see by Lemma \ref{lem:euclidean proofs:bound on e-complement} that
		\[\lim_{N\to \infty}|\E\<f(\|\b{u}(x)\|^2_N)\>-\E\<f(\|\b{u}(x)\|^2_N)I(|\|\b{u}(x)\|^2_N-\qc|\le \epsilon)\>|=0.\]
		Thus taking $\epsilon\to 0$ and using continuity of $f$ completes the proof.
	\end{proof}
	
	Now we prepare to show (\ref{eqn:euclidean proofs: f((u,u'))=zeta(r)}). For this we need a result which shows that we may shrink the Gibbs measure. For this, let us use the short hand $\<*\>^{\qc,\epsilon}=\<*\>_{D_N^{\epsilon}(\b{q}_*)}$.
	
	\begin{lem}
		\label{lem:euclidean proofs:convergence of f and f epsilon}
		For any $\epsilon>0$ and bounded $f$ we have that
		\[\E\<f((\b{u}(x),\b{u}'(x))_N\>-\E\<f((\b{u}(x),\b{u}'(x))_N)I(|\|\b{u}(x)\|^2_N-\qc|\le \epsilon)\>|\le 4\|f\|_{\infty}\E G_N(\epsilon),\]
		so that in particular
		\[\limsup_{N\to \infty}\log |\E\<f((\b{u},\b{u}')_N)\>-\E\<f((\b{u},\b{u}')_N)\>^{\qc,\epsilon}|=0.\]
	\end{lem}
	\begin{proof}
		Without loss of generality we may assume that $f$ is positive and bounded by $1$. Let $G_N(\epsilon)$ be as in Lemma \ref{lem:euclidean proofs:bound on e-complement}. We first note that
		\[|\E \<f((\b{u},\b{u}')_N)\>-\E \<f((\b{u},\b{u}')_N)I(\|\b{u}\|_N^2\in D_{\epsilon}(\b{q}_c))\>|\le  \E \<I(\|\b{u}\|_N^2\notin D_{\epsilon}(\b{q}_c))\>,\]
		which by Lemma \ref{lem:euclidean proofs:bound on e-complement} is negligible as $N\to \infty$. But we may then rewrite
		\[\begin{split}
			\<f((\b{u},\b{u}')_N)\>^{\qc,\epsilon}- \<f((\b{u},\b{u}')_N)I(\|\b{u}\|_N^2\in D_{\epsilon}(\b{q}_c))\>&=\\ \<f((\b{u},\b{u}')_N)\>^{\qc,\epsilon}\left(1-\frac{Z_{N}(D_{\epsilon}(\b{q}_c))}{Z_N}\right)&=\<f((\b{u},\b{u}')_N)\>^{\qc,\epsilon}G_N(\epsilon).
		\end{split}\]
		
		Again using Lemma \ref{lem:euclidean proofs:bound on e-complement}, this is negligible, completing the proof.
	\end{proof}
	
	With this result, we are now able to describe the rest of our method to establish (\ref{eqn:euclidean proofs: f((u,u'))=zeta(r)}), beginning with a heuristic motivation.
	
	We define $S_N(q)=\{u\in \R^N:\|u\|_N=q\}$. Associated to this are the quantities
	\[Z_{N}(\qc)=\int_{S_N(q)^\Om}e^{-\cal{H}_N(\b{u})}\omega(d\b{u}), \;\;\;\; \<f(\b{u})\>^{\qc}=\frac{1}{Z_{N}(\b{q}_c)}\int_{S_N(\qc)^{\Om}}f(\b{u})e^{-\cal{H}_N(\b{u})}\omega(d\b{u}),\]
	where $\omega$ is surface measure induced by the inclusion $S_N(\qc)^{\Om}\subseteq (\R^N)^{\Om}$. One may imagine that when $\epsilon$ is very small, $\E\<f((\b{u},\b{u}')_N)\>^{\qc,\epsilon}$ should be approximated by the $\E\<f((\b{u},\b{u}')_N)\>^{\qc}$. 
	So we in particular, the problem is now to characterize minimizer of $|\Om|^{-1}N^{-1}\E\log Z_{N,q_*}$ in terms of the Gibbs measure $\<*\>_{\qc}$, reducing us to a spherical model. 
	
	The key realization is that the restriction of the isotropic field $V_N$ given by (\ref{eqn:intro:V-def}) to any sphere $S_N(q)$ has a generic mixing function for any $q$. We show this below, but it is also fairly clear from the description of $B$ in (\ref{eqn:B-decomposition}). In particular, the identification of the minimizer should follow essentially from the methods in Section \ref{section:identification}.
	
	Unfortunately, we are unable to make this approach work directly, as we are not able to justify the approximation of $\E\<f((\b{u},\b{u}')_N)\>_{\qc,\epsilon}$ by $\E\<f((\b{u},\b{u}')_N)\>_{\qc}$. Instead, we take an alternative approach. First we decompose $V_N$ into a number of factors on $S_N(\qc)$, and then extend these factors to $D_N^{\epsilon}(\qc)$ as a whole. Adding a small multiple of one of these terms to $D_N^{\epsilon}(q_*)$, one can consider the free energy of this perturbed partition function, and show that if one sends $N\to\infty$ and then $\epsilon\to 0$, one may apply the same differentiation method as in Section \ref{section:identification}.
	
	Thus to begin, we will study the restrictions of a single copy of our original isotropic field $V_N:\R^N\to \R$ (\ref{eqn:intro:V-def}). As $V_N$ is an isotropic field on $\R^N$, the restriction of $V_N$ to $S_N(\qc)$ is also isotropic, and thus the restriction coincides with some mixed $p$-spin model (on a sphere of non-standard radius) with mixing function $B_q(r)=B(2(q-r))$. Recall (\cite{tucapure}) the pure $p$-spin model, $H_{N,p}$, which is a homogeneous polynomial of degree-$p$ on $\R^N$, which is also a centered Gaussian function with covariance given for $u,u'\in \R^N$ by
	\[\E H_{N,p}(u)H_{N,p}(u')=N(u,u')_N^p.\]
	We note that 
	\[B_q(r)=B(2(q-r))=\sum_{p=0}^{\infty}\left(\frac{(-2)^p}{p!}\right)B^{(p)}(q)r^p.\]
	It is then standard to show (see \cite{talagrandOG}) that as the restriction of $V_N$ to $S_N(\qc)$ is a centered Gaussian function with covariance $B_q$, there is a family of pairwise independent pure $p$-spin models $\{V_{N,p}(u)\}_{p\ge 0}$ (with non-standard normalization) such that for $u\in S_N(\qc)$ we have
	\[V_N(u)=\sum_{p=0}^{\infty}V_{N,p}(u).\label{eqn:euclidean proofs:decomposition}\]
	Their normalization is given so that
	\[\E V_N(u)V_N(u')=\gamma_p(2\qc)^2N(u,u')_N^p,\]
	where $\gamma_p(q)\ge 0$ is given by
	\[\gamma_p(q)^2=\left(\frac{(-2)^p}{p!}\right)B^{(p)}(q).\]
	Using the decomposition of $B$ given by (\ref{eqn:B-decomposition}) we see that for $p\ge 1$
	\[(-1)^pB^{(p)}(q)=\int_0^\infty \exp(-\lambda^2q)\lambda^{2p}\nu(d\lambda)>0,\]
	so that $\gamma_p(q)$ is strictly positive for all $q\in [0,\infty)$. 
	
	Now we stress that as (\ref{eqn:euclidean proofs:decomposition}) only holds for $u\in S_N(\qc)$. However, $V_{N,p}$ is simply a random homogeneous polynomial of order $p$, and canonically extends to $\R^N$. Thus we may consider both $V_N$ and $V_{N,p}$ as coupled Gaussian functions on $\R^N$. To understand them, we compute their covariance.
	\begin{lem}
		\label{lem:euclidean proofs:covariance of decomposition}
		For $u,u'\in \R^N$ we have that
		\[\E V_N(u)V_{N,p}(u')=N(u,u')_N^p \gamma_p(\qc+\|u\|^2_N)^2.\label{eqn:euclidean proofs:covariance of decomposition}\]
	\end{lem}
	\begin{proof}
		We first show this in the case where for some $\lambda>0$ we have that
		\[B(x)=\exp(-\lambda^2 x).\]
		In this case we have that
		\[\gamma_p(q)^2=\left(\frac{(2\lambda^2)^p}{p!}\right)\exp(-\lambda^2q).\]
		In this case, it is easy to explicitly construct a model for the field $V_N$. For this, observe that we may write
		\[N\exp(-\lambda^2 \|u-u'\|^2_N)=\sum_{p=0}^{\infty}\left(\frac{(2\lambda^{2})^p}{p!}\right)N(u,u')_N^p\exp(-\lambda \|u\|^2_N)\exp(-\lambda \|u'\|^2_N).\]
		Thus if we take a family pairwise independent standard pure $p$-spin models $\{H_{N,p}(u)\}_{p\ge 0}$, we have that as functions on $\R^N$
		\[V_N(u)\disteq \sum_{p=0}^{\infty}\left(\sqrt{\frac{(2\lambda^2)^{p}}{p!}}\right)H_{N,p}(u)e^{-\lambda^2 \|u\|^2_N}.\]
		From this we see that
		\[V_{N,p}(u)=\left(\sqrt{\frac{(2\lambda^2)^{p}}{p!}}\right)H_{N,p}(u)e^{-\lambda^2 q}.\]
		As the functions $H_{N,p}$ are pairwise independent, we have that
		\[\E V_N(u)V_{N,p}(u')=\E H_{N,p}(u)H_{N,p}(u')\left(\frac{(2\lambda^2)^{p}}{p!}\right)e^{-\lambda^2 (\|u\|^2_N+q)}=N(u,u')^p \gamma_p(\|u\|_N^2+q)^2.\]
		This establishes the case $B(x)=\exp(-\lambda^2 x)$. The case where $B(x)=c_0$ is trivial. We observe however that by taking the sum of independent isotropic fields, both sides of (\ref{eqn:euclidean proofs:covariance of decomposition}) are linear in the choice of $B$. In particular, this establishes the result when $B$ is of the form
		\[B(x)=c_0+\sum_{i=1}^{n}c_i\exp(-\lambda^2_i x)=c_0+\int_0^{\infty}e^{-\lambda^2 x}\left(\sum_{i=1}^nc_i\delta(x-\lambda_i)dx\right),\label{eqn:euclidean proofs:density}\]
		for any choice of non-negative $c_i$ and positive $\lambda_i$.
		
		However, positive linear combinations of Dirac measures are dense in the space of finite measures on $(0,\infty)$ under the total variation norm, as is seen by taking the quantile transform. Thus we only need to show that both the left- and right-hand sides of (\ref{eqn:euclidean proofs:covariance of decomposition}) are continuous with respect to the total variation norm. As the function $\lambda \mapsto e^{-\lambda^2 q}\lambda^{2p}$ is continuous and uniformly bounded on $[0,\infty)$ when $q>0$, we see that $B^{(p)}(q)$, and thus $\gamma_p(q)$ is thus continuous with respect to the total variation norm. This shows that the right-hand side of (\ref{eqn:euclidean proofs:covariance of decomposition}) is continuous with respect to the total variation norm. Similarly, as we have that
		\[|\E V_N(u)V_{N,p}(u')|\le \sqrt{\E V_N(u)^2\E V_{N,p}(u')^2}=\sqrt{NB(0)N\gamma_p(2\qc)^2},\]
		we see that the left is continuous with respect to the total variation norm as well.
	\end{proof}
	
	Now if we fix $p\ge 1$, and define our perturbed partition function
	\[Z_{N,\epsilon,\qc}(\beta,p)=\int_{D_N^{\epsilon}(\b{\qc})}\exp\left(-\cal{H}_N(\b{u})-\sum_{x\in \Om}\beta V_{N,p,x}(\b{u}(x))\right)d\b{u},\]
	where each $V_{N,p,x}$ is the $p$-spin model associated with the field $V_{N,x}$ as above. We denote as well the associated Gibbs measure as $\<*\>^{\qc,\epsilon}_{\beta,p}$. Note that, restricted to $S_N(\b{q}_*)$, this is the same form of perturbation of the Hamiltonian considered in Section \ref{section:identification}, and that when $\beta=0$, we have that $\<*\>^{\qc,\epsilon}_{0,p}=\<*\>^{\qc,\epsilon}$.
	
	We now give a computation $\partial_\beta \E \log Z_{N,\epsilon,\qc}(\beta,p)$ as $\epsilon\to 0$.
	
	\begin{lem}
		\label{lem:euclidean proofs:free energy of beta}
		Fix any $y\in \Om$. Then we have that
		\[
		\begin{split}
			\lim_{\epsilon\to 0}\limsup_{N\to \infty}\bigg|&|\Om|^{-1}N^{-1}\partial_{\beta}\E \log Z_{N,\epsilon,\qc}(\beta,p)\\&-(1+\beta)\gamma_p(2\qc)^2\left((\qc)^p-\E\<(\b{u}(y),\b{u}'(y))_N^p\>^{\qc,\epsilon}_{\beta,p}\right)\bigg|=0.
		\end{split}
		\]
	\end{lem}
	\begin{proof}
		We first observe that 
		\[\partial_{\beta}\E \log Z_{N,\epsilon,\qc}(\beta,p)=\sum_{x\in \Om}\E\<-V_{N,p,x}(\b{u}(x))\>^{\qc,\epsilon}_{\beta,p}=|\Om|\E\<-V_{N,p,y}(\b{u}(y))\>^{\qc,\epsilon}_{\beta,p}\]
		where we have used transitivity of the model in the last equality. To compute this let us define the function
		\[C_{\beta,p}(q)=\gamma_p(\qc+q)^2+\beta \gamma_p(2\qc)^2.\]
		For $\b{u},\b{u}'\in (\R^N)^{\Om}$ we may then compute the covariance
		\[\mathrm{Cov}\left(V_{N,p,y}(\b{u}(y))(\cal{H}_N(\b{u}')+\sum_{x\in \Om}\beta V_{N,p,x}(\b{u}'(x)))\right)\]
		\[=\E[V_{N,p,y}(\b{u}(y))(V_{N,y}(\b{u}'(y))+\beta V_{N,p,y}(\b{u}'(y)))=N(\b{u}(y),\b{u}'(y))_N^pC_{\beta,p}(\|\b{u}'(y)\|^2_N),\]
		where we have used Lemma \ref{lem:euclidean proofs:covariance of decomposition} in the final step.
		To use this observe that if we write
		\[\<V_{N,p,y}(\b{u}(y))\>^{\qc,\epsilon}_{\beta,p}=\frac{\int_{D_N^{\epsilon}(\b{\qc})}V_{N,p,y}(\b{u}(y))\exp\left(-\cal{H}_N(\b{u})-\sum_{x\in \Om}\beta V_{N,p,x}(\b{u}(x))\right)d\b{u}}{\int_{D_N^{\epsilon}(\b{\qc})}\exp\left(-\cal{H}_N(\b{u})-\sum_{x\in \Om}\beta V_{N,p,x}(\b{u}(x))\right)d\b{u}},\]
		then by Gaussian integration by parts we have that
		\[N^{-1}\E\<-V_{N,p,y}(\b{u}(y))\>^{\qc,\epsilon}_{\beta,p}\]\[=\E\<\|\b{u}(y)\|_N^{2p}C_{\beta,p}(\|\b{u}(y)\|^2_N)\>^{\qc,\epsilon}_{\beta,p}-\E\<(\b{u}(y),\b{u}'(y))^p_NC_{\beta,p}(\|\b{u}'(y)\|^2_N)\>^{\qc,\epsilon}_{\beta,p}.\]
		The function $C_{\beta,p}(q)$ is smooth, so by Taylor's theorem there is some $C>0$ such that 
		\[|\E\<(\b{u}(y),\b{u}'(y))^p_NC_{\beta,p}(\|\b{u}'(y)\|^2_N)\>^{\qc,\epsilon}_{\beta,p}-(\b{u}(y),\b{u}'(y))^p_NC_{\beta,p}(\qc)\>^{\qc,\epsilon}_{\beta,p}|\]\[ \le C\epsilon \E\<|(\b{u}(y),\b{u}'(y))^p_N|\>^{\qc,\epsilon}_{\beta,p}\le C\epsilon(\qc+\epsilon)^p.\]
		Similarly we have that
		\[|\E\<\|\b{u}(y)\|_N^{2p}C_{\beta,p}(\|\b{u}(y)\|^2_N)\>^{\qc,\epsilon}_{\beta,p}-(\qc)^pC_{\beta,p}(\qc)|\le C \epsilon (\qc+\epsilon)^p.\]
		As we have that $C_{\beta,p}(\qc)=(1+\beta)\gamma_p(2\qc)^2$, we have shown that there is $C>0$ (depending continuously on $\beta$) such that 
		\[\left|\Om|^{-1}N^{-1}\partial_{\beta}\E \log Z_{N,\epsilon,\qc}(\beta,p)-(1+\beta)\gamma_p(2\qc)^2\left(q^p-\E\<(\b{u}(y),\b{u}(y))_N^p\>^{\qc,\epsilon}_{\beta,p}\right)\right|\le C \epsilon,\]
		which gives the desired result.
	\end{proof}
	
	Next we want to study the behavior of as a function of $|\Om|^{-1}N^{-1}\E \log Z_{N,\epsilon,\qc}(\beta,p)$ as $\epsilon\to 0$. For this we define a functional on $\zeta\in\mathscr{Y}(\qc)$ by letting $\delta$ (\ref{eqn:def:delta-P}) and $q_*$ such that $\zeta([0,q_*])=1$, and defining
	\[
	\begin{split}
		\cal{P}_{q}(\zeta;\beta,p)=\frac{1}{2}\bigg(&\log\left(\frac{2\pi}{\beta }\right)+\frac{\beta h^2}{\mu}-\beta\mu q+\beta(q-q_*)K(\beta(q-q_*);t)\\
		&-\frac{1}{L^d}\log\left(K(\beta(q-q_*);t)I-t\Delta\right)
		+\int_{0}^{q_*}\beta  K(\beta\delta(u);t)du\\
		&+\int_0^{\qc} \zeta([0,u])\left(-2B'(2(q_*-u))+\left(\beta^2+2\beta\right)\gamma_p(2q_*)^2pu^{p-1}\right)du\bigg).	
	\end{split}
	\]
	Note that this coincides with $\cal{P}_{q}(\zeta;0,p)=\cal{P}_q(\zeta)$. Moreover, for $\beta \neq 0$, the only term that has changed is the final one, where the covariance function $B_q$ has been replaced with the covariance of the perturbed function. 
	
	We now define the function
	\[\mathscr{P}_{\qc}(\beta,p):=\inf_{\zeta\in \mathscr{Y}(\qc)}\cal{P}_{\qc}(\zeta;\beta,p).\]
	Our computation is then as follows.
	\begin{lem}
		\label{lem:euclidean proofs: free energy with beta}
		For any $\beta\in \R$ we have that
		\[\limsup_{\epsilon\to 0}\limsup_{N\to \infty}\left||\Om|^{-1}N^{-1}\E \log Z_{N,\epsilon,q_*}(\beta,p)-\mathscr{P}_{\qc}(\beta,p)\right|=0.\]
	\end{lem}
	
	Before showing this, it will be convenient to prove an auxiliary result, whose statement and proof explains the origin of the functional $\cal{P}_{\qc}(\zeta;\beta,p)$. Let us define the partition function of 
	\[Z_{N,q_*}(\beta,p)=\int_{S_N(\qc)^{\Om}}\exp\left(-\cal{H}_N(\b{u})-\sum_{x\in \Om}\beta V_{N,p,x}(\b{u}(x))\right)\omega(d\b{u}).\]
	Note when $\beta=0$ this coincides with the partition function $Z_{N}(\qc)$ above. The free energy may then be computed in terms of $\mathscr{P}_{\qc}(\beta,p)$ as well.
	\begin{lem}
		\label{lem:euclidean proofs:free energy of spherical}
		We have that
		\[\lim_{N\to \infty}|\Om|^{-1}N^{-1}\E \log Z_{N,\qc}(\beta,p)=\mathscr{P}_{\qc}(\beta,p).\]
	\end{lem}
	\begin{proof}
		Focusing on a copy of $V_N$ for the moment, we define the function
		\[V_{N,p,\beta}(u):=V_N(u)+\beta V_{N,p}(u).\label{eqn:euclidean proofs:def:V N,p,beta}\] 
		This is a centered Gaussian function on $\R^N$, and for $u,u'\in S_N(\qc)$ we have that
		\[\E[V_{N,p,\beta}(u)V_{N,p,\beta}(u')]=N\left(B(2(q-(u,u')_N))+(\beta^2+2\beta)\gamma_p(2q_*)^2 (u,u')^p_N\right).\label{eqn:euclidean proofs:ignore-101}\]
		In particular, $V_{N,p,\beta}(u)$ is isotropic on $S_N(\qc)$, and one can note the derivative of this expression appears in the rightmost term of $\cal{P}_{\qc}(\zeta;\beta,p)$. Thus $Z_{N,q_*}(\beta)$ is essentially in the form of a homogeneous spherical model considered above, except on product of spheres of radius $q_*$ instead of $1$. Re-scaling thus allows us to compute the free energy as above. In particular, by using the rescaling above, this follows from Theorem \ref{theorem:intro:main:sphere}.
	\end{proof}
	
	Thus Lemma \ref{lem:euclidean proofs: free energy with beta} expresses that one may approximate $|\Om|^{-1}N^{-1}\E \log Z_{N,\epsilon,q_*}(\beta,p)$ by 
	\\
	$|\Om|^{-1}N^{-1}\E \log Z_{N,\qc}(\beta,p)$ in the limit. The methods to establish such approximations were developed in our first companion paper \cite{Paper1} and apply immediately.
	
	\begin{proof}[Proof of Lemma \ref{lem:euclidean proofs: free energy with beta}]
		By Lemma \ref{lem:euclidean proofs:free energy of spherical} we see that it suffices to show that
		\[\lim_{\epsilon\to 0}\limsup_{N\to \infty}\left||\Om|^{-1}N^{-1}\E \log Z_{N,\epsilon,q_*}(\beta,p)-|\Om|^{-1}N^{-1}\E \log Z_{N,\qc}(\beta)\right|=0.\label{eqn:euclidean proofs:ignore-8}\]
		Returning for the moment to a single copy of the function $V_{N,p,\beta}$, we observe that $V_{N,p,\beta}$ is a centered Gaussian function on $\R^N$, and if we define the function
		\[D_{\beta,p}(q,q')=\beta\gamma_p(\qc+q)^2+\beta\gamma_p(\qc+q')^2+\beta^2 \gamma_p(2q_*),\]
		its covariance is given for $u,u'\in \R^N$ by
		\[\E[V_{N,p,\beta}(u)V_{N,p,\beta}(u')]=N\left(B(\|u-u'\|^2_N)+D_{\beta,p}(\|u\|_N^2,\|u'\|_N^2)\right).\]
		Thus we see that (\ref{eqn:euclidean proofs:ignore-8}) immediately from an application of the thickening result Proposition I.A.7
	\end{proof}
	
	We also recover the following result, which follows immediately from the identification of $\mathscr{P}_{\qc}(\beta,p)$ as essentially coming from a spherical model in Lemma \ref{lem:euclidean proofs:free energy of spherical}, combined with Lemma \ref{lem:identification:spherical parisi}.
	\begin{lem}
		\label{lem:euclidean proofs:p differential}
		The function $\mathscr{P}_{\qc}(\beta,p)$ is convex and differentiable in $\beta$. Furthermore, if we let $\zeta^{\beta,p}$ be a minimizer of the functional $\cal{P}_{\qc}(\zeta;\beta,p)$, then we have that
		\[\frac{\partial}{\partial \beta}\mathscr{P}_{\qc}(\beta,p)=(1+\beta)\gamma_p(2q_*)^2\left((\qc)^p-\int_0^{\qc} u^p\zeta^{\beta,p}(du)\right).\]
	\end{lem}
	
	We have now collected a sufficient amount of preliminary results to prove (\ref{eqn:euclidean proofs: f((u,u'))=zeta(r)}). In essence, with these results, this follows from a sub-sequence argument as in Section \ref{section:identification}, only with more difficulty coming from the fact we have only computed our functions after a double limit.
	
	\begin{proof}[Proof of (\ref{eqn:euclidean proofs: f((u,u'))=zeta(r)}) of Proposition \ref{prop:euclidean proofs:main result}]
		Fix $x\in \Om$ and let $\cal{L}_N$ denote the law of $(\b{u}(x),\b{u}'(x))_N$ under the 2-replica Gibbs measure $\E\<*\>$. By (\ref{eqn:euclidean proofs: f(u^2)=f(q)}), it is clear that this sequence is tight and hence relatively compact by Prokhorov's theorem, and all sublimits are supported on $[-\qc,\qc]$. We note that to show (\ref{eqn:euclidean proofs: f((u,u'))=zeta(r)}), it sufficient to show that for every sub-sequence of $N$, there is a further sub-sequence $N'$ such that $\cal{L}_{N'}$ converges to $\zetac$ in the weak topology. Indeed, this establishes that $\cal{L}_N$ converges to $\zetac$ weakly, and so (\ref{eqn:euclidean proofs: f((u,u'))=zeta(r)}) holds for bounded $f$. The general case then follows by employing the Cauchy-Schwartz inequality and (\ref{eqn:euclidean proofs: f(u^2)=f(q)}) to show one may neglect the behavior of a sub-exponentially growing function outside of $(-\qc-1,\qc+1)$ and thus replace $f$ by a compactly-supported function.
		
		To avoid notational clutter, we will denote all additional sub-sequences of our initial sub-sequence as $N'$. To begin we may assume that $\cal{L}_{N'}$ converges to some law $\mu$. By Lemma \ref{lem:euclidean proofs:convergence of f and f epsilon}, if we fix $\epsilon>0$, the law of $(\b{u}(x),\b{u}'(x))_N$ under $\E\<*\>^{\qc,\epsilon}$ also converges to $\mu$. Now fix $p\ge 1$. By Lemma \ref{lem:euclidean proofs:free energy of beta}, and possibly passing to a sub-sequence, we see that
		\[\lim_{\epsilon\to 0}\lim_{N'\to \infty}|\Om|^{-1}(N')^{-1}\partial_{\beta}\E \log Z_{N',\epsilon,\qc}(\beta,p)\big|_{\beta=0}=\gamma_p(2\qc)^2\left((\qc)^p-\int_{-\qc}^{\qc}r^p\mu(dr)\right).\label{eqn:euclidean proofs :ignore 83}\]
		Now we define
		\[f_{N,p}(\beta,\epsilon):=|\Om|^{-1}N^{-1}\E \log Z_{N',\epsilon,\qc}(\beta,p).\]
		These functions are convex in $\beta$. Moreover, fixing $\epsilon>0$ and restricting to $|\beta|\le 1$, clearly uniformly bounded in $N$ by Jensen's inequality. Thus by Helly's selection theorem, we may choose for each $\epsilon>0$, a further subsequence of $N'$, such that $f_{N',p}(\beta,\epsilon)$ converges in $N'$ to some convex limit function, say $f_\epsilon(\beta,p)$. By Lemma \ref{lem:euclidean proofs: free energy with beta} we have that the convex functions $f_{\epsilon'}(\beta,p)$ converge to $\mathscr{P}_{\qc}(\beta,p)$ as $\epsilon\to 0$. By diagonalization, we may choose a countable sequence $\epsilon_{N'}$ and sub-sequence $N'$ such that $f_{N',p}(\beta,\epsilon_{N'})$ converges to $\mathscr{P}_{\qc}(\beta,p)$. Now using again the fact that if a differentiable sequence of convex functions converges to a differentiable convex function (see Theorem 25.7 of \cite{convexanalysis}), we conclude by Lemma \ref{lem:euclidean proofs:p differential} and (\ref{eqn:euclidean proofs :ignore 83}) that
		\[\gamma_p(2\qc)^2\left((\qc)^p-\int_{-\qc}^{\qc}r^p\mu(dr)\right)=\gamma_p(2\qc)^2\left((\qc)^p-\int_{0}^{\qc}r^p\zetac(dr)\right).\label{eqn:euclidean proofs:ignore 23}\]
		Employing a final diagonalization over $p$, we find that (\ref{eqn:euclidean proofs:ignore 23})  holds for all $p$, so that $\mu=\zetac$, completing the proof.
	\end{proof}

	\section{The RS Regime and the Larkin Mass\label{section:parisi-larkin}}
	
	In this section we will provide proofs for all of the results which concern the RS and RSB regime. For convenience, we will also restore ourselves to the case of arbitrary $\beta$ in this section, instead of reducing to the case of $\beta=1$ as above. The proofs primarily boil down to calculus.
	
	To begin we first give the evaluation of both functionals on Dirac mass measures.
	\begin{lem}
		\label{lem:larkin and minimization:RS evaluation for fixed q}
		For $q_*\in [0,1)$, we have
		\[
		\begin{split}
			\cal{B}_\beta(\delta_{q_*})=\frac{1}{2}\bigg(&\log\left(\frac{2\pi}{\beta}\right)+\beta h^2K(\beta (1-q_*);t)^{-1}+\beta K(\beta(1-q_*);t)\\
			&-\frac{1}{L^d}\log \det(K(\beta(1-q_*);t)I-t\Delta)+\beta^2(\xi(1)-\xi(q_*))\bigg). \label{eqn:larkin and minimization: RS evaluated for spherical}
		\end{split}
		\]
		For $q_*\in [0,q)$ we have
		\[
		\begin{split}
			\cal{P}_{\beta,q}(\delta_{q_*})=\frac{1}{2}\bigg(&\log\left(\frac{2\pi}{\beta}\right)+\frac{\beta h^2}{\mu}+\beta q\left(K(\beta(q-q_*);t)-\mu\right)\\
			&-\frac{1}{L^d}\log \det(K(\beta (q-q_*);t)-t\Delta ) + \beta^2\left(B(0)-B(2(q-q_*))\right) \bigg).
		\end{split}
		\]
		\label{eqn:larkin and minimization: RS evaluated for euclidean}
	\end{lem}
	\begin{proof}
		For (\ref{eqn:larkin and minimization: RS evaluated for spherical}) we simply note that for $0<u<q_*$ we have that $\delta(u)=1-q_*$ so that
		\[\int_0^{q_*}K(\beta \delta(u);t)du=q_*K(\beta(1-q_*);t),\]
		a term we may cancel. A similar cancellation happens in (\ref{eqn:larkin and minimization: RS evaluated for euclidean}).
	\end{proof}
	
	We now show our results in the spherical case.
	\begin{proof}[Proof of Theorem \ref{theorem:Spherical RS description}]
		We first study the the minimization problem for the equation $\cal{B}_\beta(\delta_{q_*})$ over $q_*\in [0,1)$. First we note that as $q_*\to 1$, we have that $K(\beta(1-q_*);t)\to \infty$. In particular, letting $r=K(\beta(1-q_*);t)$, we see from (\ref{eqn:larkin and minimization: RS evaluated for spherical}) that as $q_*\to 1$ the term diverges like $\beta r-\log(r)\to \infty$. In particular the equation does have a minimizer on $[0,1)$. Next we compute from (\ref{eqn:larkin and minimization: RS evaluated for spherical}) that
		\[\partial_{s}\cal{B}_{\beta}(\delta_{s})=\frac{\beta^2}{2}\left(h^2K(\beta (1-s);t)^{-2}K'(\beta (1-s);t)-sK'(\beta(1-s);t)-\xi'(s)\right).\]
		On the edge $0$ of $[0,1)$, we have that 
		\[\partial_s\cal{B}_\beta(\delta_s)\big|_{s=0}=\frac{\beta^2}{2}\left(h^2K(\beta;t)^{-2}K'(\beta;t)-\xi'(0)\right).\]
		Both of these terms are negative, so in particular, any minimizer must satisfy $\partial_s\cal{B}_\beta(\delta_s)=0$.
		
		Now let $\zeta=\delta_{q_*}$, where $q_*$ is some minimizer of $\cal{B}_\beta(\delta_s)$. We directly compute the first function from Theorem \ref{theorem:intro:parisi-measure:sphere} as 
		\[
		\begin{split}
			F_{\beta}(s)=&-h^2K(\beta (1-q_*);t)^{-2}K'(\beta (1-q_*);t)+\xi'(s)\\
			&+\begin{cases}
				sK'(\beta(1-q_*);t);\;\;\; s\in[0,q_*]\\
				\beta^{-1}\left(K(\beta(1-q_*);t)-K(\beta(1-s);t)\right) +q_*K'(\beta(1-q_*);t);\;\;\; s\in (q_*,1)
			\end{cases}.
		\end{split}
		\]
		Moreover as $f_\beta(s)=\int_0^s F_\beta(r)dr$, we see that (\ref{eqn:convexity:minimization eqn:Euclidean: measure}) may be rewritten as
		\[\sup_{0<s<1}f_\beta(s)=f_\beta(q_*).\label{eqn:ignore-luna-3}\]
		So the model is RS if and only if some minimizer $q_*$ satisfies (\ref{eqn:ignore-luna-3}). Thus all we need to show is that (\ref{eqn:ignore-luna-3}) is equivalent to (\ref{eqn:intro:min equation RS for sphere}).
		
		For this note that
		\[0=\partial_{s}\cal{B}_{\beta}(\delta_{s})\big|_{s=q_*}=-\beta^2 F_\beta(q_*).\]
		In particular $q_*$ is indeed a critical point of $f_\beta(q)$. Next we show that (\ref{eqn:ignore-luna-3}) follows from the weaker statement
		\[\sup_{q_*<s<1}f_\beta(s)=f_\beta(q_*).\label{eqn:ignore-luna-4}\]
		For this, note that as (\ref{eqn:ignore-luna-4}) implies that $f''_\beta(q_*)=F_\beta'(q_*)\le 0$. Next we note that for $s\in [0,q_*)$ we have that
		\[F'_\beta(s)=\xi''(s)+K'(\beta(1-q_*);t)\le \xi''(q_*)+K'(\beta(1-q_*);t)=F'_\beta(q_*)\le 0.\]
		This shows that $f_\beta$ is concave on $[0,q_*]$. As $f_\beta'(q_*)=F_\beta(q_*)=0$, we see that \[\sup_{0<s<q_*}f_\beta(s)=f_\beta(q_*),\]
		so indeed (\ref{eqn:ignore-luna-4}) is equivalent to (\ref{eqn:ignore-luna-3}).
		
		Now observing that $f_\beta(s)-f_\beta(q_*)=\int_{q_*}^sF_\beta(r)dr$ we see that to show (\ref{eqn:ignore-luna-4}), it suffices to show that
		\[\sup_{q_*<s<1}\int_{q_*}^sF_\beta(r)dr=0.\]
		Using the fact that $F_\beta(q_*)=0$ we see $q\in [q_*,q]$ that
		\[	F_{\beta}(s)=F_\beta(s)-F_\beta(q_*)=\xi'(s)-\xi'(q_*)+\beta^{-1}\left(K(\beta(1-q_*);t)-K(\beta(1-s);t)\right).
		\]
		Integrating this expressions immediately shows that for $s\in [q_*,1)$, $\int_{q_*}^s F_{\beta}(r)dr=g_\beta(s)-g_\beta(q_*)$, completing the proof.
	\end{proof}
	
	We now move onto the analogous result in the Euclidean case.
	
	\begin{proof}[Proof of Theorem \ref{theorem:Euclidean RS description}]
		To begin let us consider some pair $(\delta_{q_*},q)$, and consider $q$ fixed. First we compute
		\[\partial_{q_*}\mathcal{P}_{\beta,q}(\delta_{q_*})=\frac{\beta^2}{2}\left(- q_*K'(\beta(q-q_*);t)+2B'(2(q-q_*))\right).\]
		Proceeding as in the proof of Theorem \ref{theorem:Spherical RS description}, we may show that there is a minimizer $q_*\in [0,q)$ to  $\mathcal{P}_{\beta,q}(\delta_{q_*})$, satisfying $\partial_{q_*}\mathcal{P}_{\beta,q}(\delta_{q_*})=0$. Moreover, fixing such a minimizer $q_*$ one may show that (\ref{eqn:intro:minimization eqn:Euclidean: measure}) is equivalent the condition that
		\[\sup_{q_*<s<q}\hat{g}_{\beta,q}(s)=\hat{g}_{\beta,q}(q_{*}),\label{eqn:ignore-luna-34}\]
		where here
		\[
		\begin{split}
			\hat{g}_{\beta,q}(s)=&\beta^2 B(2(q-s))+\beta (q-s)K(\beta(q-s);t)-\frac{1}{L^d}\log(K(\beta(q-s);t)-t\Delta)\\
			&-s\left(-2\beta^2 B'(2(q-q_*))-\beta K(\beta(q-q_*);t)\right).	
		\end{split}
		\]
		The second equation (\ref{eqn:intro:minimization eqn:Euclidean: Larkin}) is simpler and reads that
		\[\beta (q-q_*)=R_1(\mu;t) \text{ or equivalently that }K(\beta(q-q_*);t)=\mu.\]
		Plugging this equation into the formula of Lemma \ref{lem:larkin and minimization:RS evaluation for fixed q} immediately shows the claimed formula for the free energy when it is RS. Moreover, employing this again, we see that the equation $\partial_{q_*}\mathcal{P}_{\beta,q}(\delta_{q_*}))=0$ has a unique solution given by 
		\[q_*=2B'\left(\frac{2}{\beta}R_1(\mu;t)\right)K'(R_1(\mu;t);t).\]
		As $K'(R_1(\mu ;t);t)=-R_2(\mu;t )^{-1}$, we see that this implies that $q_*=q_{L}$. This shows the potential pair can only have the claimed form, so we only need to check that (\ref{eqn:ignore-luna-34}) is equivalent to our claimed criterion for being RS. For this, we note that
		\[\hat{g}_{\beta,q}(q-\beta^{-1}s)=\beta ^2 B\left(\frac{2s}{\beta}\right)+sK(s;t)-\frac{1}{L^d}\log(K(s;t)-t\Delta)\]
		\[+(q-\beta^{-1}s)\left(2\beta^2 B'\left(\frac{2}{\beta}R_1(\mu;t)\right)+\beta \mu \right).\]	
		Removing some constant terms from the second line then yields the desired result.
	\end{proof}
	
	With these characterizations in hand, we spend the rest of the section proving Theorems \ref{theorem:intro:zero:temperature larkin} and \ref{theorem:larkin-mass}. First we observe that the Larkin equation (\ref{eqn:intro:fake-larkin}) may be rewritten as
	\[g_{\beta}''(R_1(\mu;t))=4 B''\left(\frac{2}{\beta}R_1(\mu;t)\right)-R_2(\mu;t)^{-1}= 0.\]
	This clearly becomes negative as $\mu\to \infty$. So for it to have no solutions at fixed $\beta$, $g_{\beta}''(R_1(\mu;t))$ must be negative for all values $\mu$. However, noting that
	\[\partial_\beta g_\beta''(R_1(\mu;t))=\frac{8}{\beta^2}R_1(\mu,t)B'''\left(\frac{2}{\beta}R_1(\mu;t)\right)<0,\]
	we see two things. First, if there are no solutions for some $\beta$, the same is true for all $\beta'<\beta$. Moreover, it is clear that $g_\beta''(R_1(1;t))\to \infty$ as $\beta\to 0$. Together these justify the definition of the Larkin temperature. 
	
	Second, this shows that the positive temperature Larkin mass $\mu_{Lar}(\beta;t)$ is increasing in $\beta$, with $\lim_{\beta\to \infty }\mu_{Lar}(\beta;t)=\mu_{Lar}(\infty;t)$. In particular, claim 2 of Theorem \ref{theorem:intro:zero:temperature larkin} follows from claim 2 of Theorem \ref{theorem:larkin-mass}. Finally, note by definition that for all $\mu\ge \mu_{Lar}(\beta;t)$, one has that $g_\beta''(R_1(\mu;t))\le 0$. Thus both claims of Theorem \ref{theorem:larkin-mass} follow from the following result.
	
	\begin{lem}
		If $\mu'\ge 0$ is such that for all $\mu\ge \mu'$, one has $g_\beta''(R_1(\mu;t))\le 0$, then the model is RS as $(\beta,\mu')$.
	\end{lem}
	\begin{proof}
		Applying $K(*;t)$, we see that our assumption on $\mu'$ implies that for all $x\le R_1(\mu';t)$, we have that $g_{\beta}''(x)\le 0$. This shows that $g_{\beta}$ is concave $[0,R_1(\mu;t)]$. As $R_1(\mu;t)$ is already a critical point of $g_{\beta}$, this shows that it maximizes it, thus showing (\ref{eqn:ignore-luna-2}).	
	\end{proof}
	
	Finally, for the first claim of Theorem \ref{theorem:intro:zero:temperature larkin}, fix $\mu<\mu_{Lar}(\infty;t)$ and note that
	\[4B''(0)-R_2(\mu;t)^{-1}>4B''(0)-R_2(\mu_{Lar}(\infty;t))^{-1}=0.\]
	By continuity of $B''$ we see that $g_\beta''(R_1(\mu;t))$ is negative for sufficiently large $\mu$, establishing that it is RSB.

	\pagebreak
	
	\bibliographystyle{abbrv}
	\bibliography{mainbib}
\end{document}